\documentclass[11pt]{amsart}
\usepackage{amssymb, amsmath}
\usepackage[active]{srcltx}
\usepackage[pagebackref=false, colorlinks]{hyperref}

\hypersetup{pdffitwindow=blue,linkcolor=blue,citecolor=blue,
urlcolor=blue,menucolor=blue}

\usepackage{xcolor}
\usepackage{wasysym}
\usepackage{bm}
\usepackage{fullpage}

\usepackage{etoolbox}
\makeatletter
\patchcmd{\l@chapter}{1.0em}{0.8em}{}{}
\makeatother

\numberwithin{equation}{section}

\def\vecb{{\text{\boldmath$b$}}}

\def\vece{{\text{\boldmath$e$}}}

\def\vecm{{\text{\boldmath$m$}}}

\def\vecu{{\text{\boldmath$u$}}}
\def\vecv{{\text{\boldmath$v$}}}

\def\vecx{{\text{\boldmath$x$}}}

\def\vecy{{\text{\boldmath$y$}}}
\def\vecz{{\text{\boldmath$z$}}}
\def\vecalf{{\text{\boldmath$\alpha$}}}

\def\vec0{{\text{\boldmath$0$}}}

\def\scrB{{\mathcal B}}
\def\scrC{{\mathcal C}}
\def\scrD{{\mathcal D}}

\def\scrF{{\mathcal F}}

\def\scrI{{\mathcal I}}
\def\scrJ{{\mathcal J}}

\def\scrN{{\mathcal N}}

\def\scrO{{\mathcal O}}

\def\scrR{{\mathcal R}}

\def\scrU{{\mathcal U}}

\def\scrY{{\mathcal Y}}

\def\fC{{\mathfrak C}}

\def\supp{\operatorname{supp}}
\def\Lip{\operatorname{Lip}}

\def\diag{\operatorname{diag}}

\def\dist{\operatorname{dist}}

\def\DI{\operatorname{\mathbf{DI}}}

\def\Leb{\operatorname{Leb}}

\def\id{\operatorname{id}}

\newcommand{\figref}[1]{\hyperref[#1]{Figure \ref{#1}}}
\newcommand{\lemref}[1]{\hyperref[#1]{Lemma \ref{#1}}}
\newcommand{\thmref}[1]{\hyperref[#1]{Theorem \ref{#1}}}
\newcommand{\conjref}[1]{\hyperref[#1]{Conjecture \ref{#1}}}
\newcommand{\propref}[1]{\hyperref[#1]{Proposition \ref{#1}}}
\newcommand{\corref}[1]{\hyperref[#1]{Corollary \ref{#1}}}
\newcommand{\defref}[1]{\hyperref[#1]{Definition \ref{#1}}}
\newcommand{\rmkref}[1]{\hyperref[#1]{Remark \ref{#1}}}
\newcommand{\qref}[1]{\hyperref[#1]{Question \ref{#1}}}
\newcommand{\secref}[1]{\hyperref[#1]{\S\ref{#1}}}
\newcommand{\appref}[1]{\hyperref[#1]{Appendix \ref{#1}}}

\newcommand{\cC}{\mathcal{C}}

\newcommand{\cK}{\mathcal{K}}

\newcommand{\cN}{\mathcal{N}}

\newcommand{\cR}{\mathcal{R}}

\renewcommand{\d}[1]{\ensuremath{\operatorname{d}\!{#1}}}

\newcommand{\tDelta}{\widetilde{\Delta}}

\newcommand{\uK}{\underline{K}}

\newcommand{\bn}{\mathbf{0}}

\newcommand{\col}{\: : \:}

\newcommand{\ve}{\varepsilon}

\newcommand{\vol}{\operatorname{vol}}

\newcommand{\matr}[4]{\left( \begin{matrix} #1 & #2 \\ #3 & #4 \end{matrix} \right) }

\newcommand{\G}{\Gamma}
\newcommand{\SL}{\text{SL}}

\newcommand{\N}{{\mathbb N}}
\newcommand{\Z}{{\mathbb Z}}

\newcommand{\R}{{\mathbb R}}

\newcommand{\GL}{\text{GL}}

\newtheorem{thm}{Theorem}[section]
\newtheorem{lem}[thm]{Lemma}%
\newtheorem{prop}[thm]{Proposition}%
\newtheorem{cor}[thm]{Corollary}%

\theoremstyle{definition}
\newtheorem{remark}{Remark}

\begin{document}
\begin{normalsize}
\begin{abstract}
Let $\psi$ be a continuous decreasing function defined on all large positive real numbers. We say that a real
$m\times n$ matrix $A$ is $\psi$-Dirichlet if for every sufficiently large real number $t$ one can find $\bm{p} \in \Z^m$, $\bm{q} \in \Z^n\smallsetminus\{\bm{0}\}$ satisfying $\|A\bm{q}-\bm{p}\|^m< \psi({t})$ and $\|\bm{q}\|^n<{t}$. This property was introduced by Kleinbock and Wadleigh in 2018, 
generalizing the property of $A$ being Dirichlet improvable which dates back to Davenport and Schmidt (1969). 
In the present paper, we give sufficient conditions on $\psi$ to ensure
that the set of $\psi$-Dirichlet matrices   has zero or full Lebesgue measure.
Our proof is dynamical and relies on the effective equidistribution and doubly mixing of certain expanding horospheres 
in the space of lattices. Another main ingredient of our proof is an asymptotic measure estimate for certain compact neighborhoods of the critical locus (with respect to the supremum norm) in the space of lattices. Our method also works for the analogous weighted problem where the relevant supremum norms are replaced by certain weighted quasi-norms.
\end{abstract}
\title[Upper bound...]{A measure estimate in geometry of numbers and improvements to Dirichlet's theorem}
\author{Dmitry Kleinbock}
\address{Brandeis University, Waltham MA, USA, 02454-9110}
\email{kleinboc@brandeis.edu}

\author{Andreas Str\"ombergsson}
\address{Department of Mathematics, Uppsala University, Box 480, SE-75106, Uppsala, SWEDEN}
\email{astrombe@math.uu.se}

\author{Shucheng Yu}
\address{Department of Mathematics, Uppsala University, Box 480, SE-75106, Uppsala, SWEDEN}
\email{shucheng.yu@math.uu.se}
\date{\today}
\thanks{D.K.\ was supported by NSF grant  DMS-1900560. A.S.\ and S.Y.\ were supported by the Knut and Alice Wallenberg Foundation}
{\maketitle}

\enlargethispage{10pt}

\rule{0pt}{0pt}\vspace{-45pt}
\tableofcontents

\setlength{\parskip}{6pt}

\section{Introduction}

\subsection{Background} 
Let $m, n$ be two positive integers and let $M_{m,n}(\R)$ be the space of $m$ by $n$ real matrices. The starting point of our work is the following higher dimensional generalization of the classical Dirichlet's Diophantine approximation theorem, see e.g. \cite[\S 1.5]{Cassels1957}.
\begin{thm}
For any $A\in M_{m,n}(\R)$ and $t>1$, there exists $(\bm{p},\bm{q})\in \Z^m\times (\Z^n\smallsetminus\{\bm{0}\})$ satisfying the following system of inequalities:
 \begin{align}\label{equ:dirichletoriginal}
\|A\bm{q}-\bm{p}\|^m\leq \frac{1}{t}\quad \textrm{and}\quad \|\bm{q}\|^n<t.
\end{align}
Here $\|\cdot\|$ denotes the supremum norm on $\R^m$ and $\R^n$ respectively.
\end{thm}

A natural question to ask is whether one can improve \eqref{equ:dirichletoriginal} by replacing $1/t$ by a smaller function, that is, consider the following system of inequalities:
\begin{align}\label{equ:dirichlet}
\|A\bm{q}-\bm{p}\|^m<\psi(t)\quad \textrm{and}\quad \|\bm{q}\|^n<t
\end{align}
where $\psi$ is a positive, continuous, decreasing function which decays to zero at infinity. 
Historically there have been two directions to pursue in this regard: looking for solvability of \eqref{equ:dirichlet} for an unbounded set of $t > 0$ vs.\ for all large enough $t$. The former is sometimes referred to as  \textit{asymptotic approximation}, and has culminated in definitive results such as the Khintchine-Groshev theorem. In this paper we are interested in the latter, less studied set-up of \textit{uniform approximation}.
Following the definition in 
Kleinbock and Wadleigh \cite{KleinbockWadleigh2018}, 
we say that an $m$ by $n$ real matrix $A$ is \textit{$\psi$-Dirichlet} if the system of inequalities \eqref{equ:dirichlet}
has solutions in $(\bm{p},\bm{q})\in \Z^m\times (\Z^n\smallsetminus \{\bm{0}\})$ for all sufficiently large $t$. It is clear that $A\in M_{m,n}(\R)$ is $\psi$-Dirichlet if and only if $A+A'$ is $\psi$-Dirichlet for any $A'\in M_{m,n}(\Z)$. Thus with slight abuse of notation, we denote by $\DI_{m,n}(\psi)\subset M_{m,n}(\R/\Z)$ the set of $\psi$-Dirichlet matrices. 

Let $\psi_1(t)=1/t$. The problem of improving Dirichlet's theorem was initiated by Davenport and Schmidt  \cite{hDwS70a,hDwS69b} where they showed that the set 
\begin{align}\label{equ:DI}
\DI_{m,n}:=\bigcup_{0<c<1}\DI_{m,n}(c\psi_1)
\end{align}
of \textit{Dirichlet improvable matrices}
is of Lebesgue measure zero, while having full Hausdorff dimension $mn$. More recently, Kleinbock and Mirzadeh \cite[Theorem 1.5]{dKsM2020} showed that for any fixed $0<c<1$, the Hausdorff dimension of $\DI_{m,n}(c\psi_1)$ is strictly smaller than $mn$.
There have also been extensive studies on the Hausdorff dimensions of the (even smaller) set of the \textit{singular matrices}, 
\begin{align*}
\textbf{Sing}_{m,n}:=\bigcap_{0<c<1}\DI_{m,n}(c\psi_1).
\end{align*} 
After a series of breakthrough work, it is now known that the Hausdorff dimension of $\textbf{Sing}_{m,n}$ is $mn-\tfrac{mn}{m+n}$ whenever $\max\{m,n\}>1$; see \cite{Cheung2011,yCnC2016,sKdKeLgM2017,tDlFdSmU2017,tDlFdSmU2019}. 

On the other hand, for a general decreasing function $\psi$
with $t\mapsto t\psi(t)$ increasing,
Kleinbock and Wadleigh proved a zero-one law for the Lebesgue measure of $\DI_{1,1}(\psi)$ depending on the divergence or convergence of a certain series involving $\psi$ \cite[Theorem 1.8]{KleinbockWadleigh2018}. See also \cite{HussainKleinbockWadleighWang2018} for the relevant dimension theory of $\DI_{1,1}(\psi)$, \cite{KleinbockRao2022} for a similar zero-one law with the supremum norm replaced by the Euclidean norm and \cite{KleinbockWadleigh2019,KimKim2022} for analogous results in the \textit{inhomogeneous} setting.

The arguments in \cite{KleinbockWadleigh2018} rely on the theory of continued fractions and are not applicable for higher dimensions. Nevertheless, for general dimensions, building on ideas from \cite{Dani1985,KleinbockMargulis1999},
a dynamical approach was proposed in \cite[\S 4]{KleinbockWadleigh2018}, reformulating the problem as a \textit{shrinking target problem}, which asks whether a generic orbit in a dynamical system hits a given sequence of shrinking targets 
{infinitely often}. 
To describe this dynamical interpretation, let us first fix some notation. Let $d=m+n$ and let $X_d:=\SL_d(\R)/ \SL_d(\Z)$ 
be the homogeneous space which parameterizes the space of unimodular lattices in $\R^d$ via $g\SL_d(\Z) \leftrightarrow g\Z^d $. We note that $\SL_d(\R)$ acts on $X_d$ naturally via the regular action: $g\Lambda=gh\Z^d$ for any $g\in\SL_d(\R)$ and $\Lambda=h\Z^d\in X_d$. 
For any $s\in \R$, 
let $a_s$ be the diagonal matrix
\begin{align*}
a_s:=\matr{e^{s/m}I_m}00{e^{-s/n}I_n}\in \SL_d(\R).
\end{align*}
Let $\Delta: X_d\to [0,\infty)$ be the function defined by
\begin{align}\label{def:delta}
\Delta(\Lambda):=\sup_{\bm{v}\in \Lambda\smallsetminus \{\bm{0}\}}\log\Bigl(\frac{1}{\|\bm{v}\|}\Bigr).
\end{align}
Finally, let us denote 
\begin{align}\label{equ:submainfold}
\mathcal{Y}:=\left\{\Lambda_A:= \left(\begin{matrix}
I_m & A\\
0 & I_n\end{matrix}\right)\Z^d\in X_d\col  A\in M_{m,n}(\R)\right\}.
\end{align}
The submanifold $\mathcal{Y}\subset X_d$ can be naturally identified with the $mn$-dimensional torus $M_{m,n}(\R/\Z)$ via $\Lambda_A\leftrightarrow A\in M_{m,n}(\R/\Z)$. Throughout the paper, we denote by $\textrm{Leb}$ the probability Lebesgue measure on $\mathcal{Y}\cong M_{m,n}(\R/\Z)$; for simplicity of notation, for any function $f$ on $\scrY$ we will abbreviate the space average $\int_{\scrY}f(\Lambda_A)\d\,\textrm{Leb}(A)$ by either $\textrm{Leb}(f)$ or $\int_{\scrY}f(\Lambda_A)\,\d A$.

It was shown in \cite[Proposition 4.5]{KleinbockWadleigh2018} that for any given $\psi$ as above, there exists a unique continuous function $r=r_{\psi}: [s_0,\infty)\to (0,\infty)$ such that 
\begin{align}\label{equ:dyintkw}
\textrm{$A\in M_{m,n}(\R)$ is \textit{not} $\psi$-Dirichlet\ $\Leftrightarrow$\ $a_s \Lambda_A\in \Delta^{-1}[0, r(s)]$ for an unbounded set of $s>s_0$.}
\end{align} 
This equivalence is usually called the \textit{Dani Correspondence}.
In view of this interpretation, our task is 
to find conditions which ensure that for almost every (or almost no) $A\in M_{m,n}(\R)$, the orbit $\{a_s\Lambda_A\}_{s>s_0}$ hits the shrinking target $\Delta^{-1}[0, r(s)]$ for an unbounded set of $s$-values. We note that this was also the strategy used in \cite{KleinbockMargulis1999} giving a dynamical proof of the classical Khintchine-Groshev Theorem, where the relevant shrinking targets are certain cusp neighborhoods in $X_d$. For our case, by Mahler's compactness criterion, the shrinking targets $\Delta^{-1}[0, r(s)]$ are compact neighborhoods of the 
critical locus 
$\Delta^{-1}\{0\}$,
whose explicit description is given by 
Haj\'os's Theorem \cite{Hajos1941} 
(cf.\ Theorem \ref{thm:hajosmin} below). 
The fact that these shrinking targets are compact sets causes new difficulties compared to the situation in \cite{KleinbockMargulis1999}, see the discussion in Section \ref{sec:discussion}.

\subsection{Main results}\label{MAINRESsec}
In the present paper, based on the dynamical interpretation described above, we give sufficient conditions on $\psi$ implying that 
$\DI_{m,n}(\psi)$ is of zero or full Lebesgue measure. In fact, with some 
modifications, our arguments also work for the analogous \textit{weighted} problem where the supremum norms in \eqref{equ:dirichlet} are replaced by certain \textit{weighted quasi-norms},
as introduced in \cite{Kleinbock1998}.
We thus prove our main result in that generality. 
We first introduce the relevant notation.

Let $\bm{\alpha}\in \R^m$ and $\bm{\beta}\in \R^n$ be two \textit{weight vectors}, that is 
\begin{align*}
\bm{\alpha}=(\alpha_1,\ldots, \alpha_m)\in (\R_{>0})^m\quad\textrm{and}\quad \bm{\beta}=(\beta_1,\ldots, \beta_n)\in (\R_{>0})^n
\end{align*}
with $\sum_i\alpha_i=\sum_j\beta_j=1$. We say that $A\in M_{m,n}(\R)$ is \textit{$\psi_{\bm{\alpha},\bm{\beta}}$-Dirichlet} if the system of inequalities 
\begin{align}\label{equ:dirichletwei}
\|A\bm{q}-\bm{p}\|_{\bm{\alpha}}<\psi(t)\quad \textrm{and}\quad \|\bm{q}\|_{\bm{\beta}}<t
\end{align}
has solutions in $(\bm{p},\bm{q})\in \Z^m\times (\Z^n\smallsetminus \{\bm{0}\})$ for all sufficiently large $t$. Here 
\begin{align*}
\|\bm{x}\|_{\bm{\alpha}}:=\max\left\{|x_i|^{1/\alpha_i}\col 1\leq i\leq m\right\}\quad\textrm{and}\quad \|\bm{y}\|_{\bm{\beta}}:=\max\left\{|y_j|^{1/\beta_j}\col 1\leq j\leq n\right\}
\end{align*}
are the two quasi-norms associated with $\bm{\alpha}$ and $\bm{\beta}$ respectively. Again it is easy to see that $A\in M_{m,n}(\R)$ is $\psi_{\bm{\alpha},\bm{\beta}}$-Dirichlet if and only if $A+A'$ is $\psi_{\bm{\alpha},\bm{\beta}}$-Dirichlet for any $A'\in M_{m,n}(\Z)$, and 
we denote by $\DI_{\bm{\alpha},\bm{\beta}}(\psi)\subset M_{m,n}(\R/\Z)$ the set of $\psi_{\bm{\alpha},\bm{\beta}}$-Dirichlet matrices.
We note that when $\bm{\alpha}=(\tfrac{1}{m},\ldots, \tfrac{1}{m})$ and $\bm{\beta}=(\tfrac{1}{n},\ldots,\tfrac{1}{n})$, then $\DI_{\bm{\alpha},\bm{\beta}}(\psi)=\DI_{m,n}(\psi)$.
 
We now state our main result which gives sufficient conditions on $\psi$ determining when $\DI_{\bm{\alpha},\bm{\beta}}(\psi)$ is of full or zero Lebesgue measure.

\begin{thm}\label{thm:improvedirichleconwei}
Fix $m,n\in \N$ and two weight vectors $\bm{\alpha}\in \R^m$ and $\bm{\beta}\in \R^n$. Let $d=m+n$, and let 
\begin{align*}
\varkappa_d=\frac{d^2+d-4}{2}\quad \textrm{and}\quad \lambda_d=\frac{d(d-1)}{2}.
\end{align*}
Let $t_0>0$ and let $\psi: [t_0,\infty)\to {(0,\infty)}$ be a continuous, decreasing function such that
\begin{align}\label{equ:con1psi}
\textrm{the function $t\mapsto t\psi(t)$ is increasing}
\end{align} 
and
\begin{align}\label{equ:con2psi}
\psi(t)<\psi_1(t)=1/t\quad \textrm{for all $t\geq t_0$}.
\end{align} 
Let $F_{\psi}(t):=1-t\psi(t)$. If the series 
\begin{align}\label{equ:crticalseriesconj}
\sum_{k\geq t_0}k^{-1}F_{\psi}(k)^{\varkappa_d}\log^{\lambda_d}\left(\frac{1}{F_{\psi}(k)}\right)
\end{align} 
converges, then $\DI_{\bm{\alpha},\bm{\beta}}(\psi)$ is of full Lebesgue measure. Conversely, if the series \eqref{equ:crticalseriesconj} diverges, and 
\begin{align}\label{equ:condipsi}
\liminf_{t_1\to\infty}\frac{\sum_{t_0\leq k\leq t_1}k^{-1}F_{\psi}(k)^{\varkappa_d}\log^{\lambda_d+1}\left(\frac{1}{F_{\psi}(k)}\right)}
{\left(\sum_{t_0\leq k\leq t_1}k^{-1}F_{\psi}(k)^{\varkappa_d}\log^{\lambda_d}\left(\frac{1}{F_{\psi}(k)}\right)\right)^2}=0,
\end{align} 
then $\DI_{\bm{\alpha},\bm{\beta}}(\psi)$ is of zero Lebesgue measure. 
\end{thm}

\begin{remark}\label{technassumptionREM}
When $m=n=1$, Theorem \ref{thm:improvedirichleconwei} is not new; in fact \cite[Theorem 1.8]{KleinbockWadleigh2018} is stronger in the sense that it gives a tight zero-one law without the extra assumption \eqref{equ:condipsi}. 
We believe that an analogous tight zero-one law should also hold for general dimensions $m,n$,
i.e.\ that 
Theorem \ref{thm:improvedirichleconwei} should hold
with the assumption \eqref{equ:condipsi} removed.
See Remark \ref{rmk:teassu} 
for a discussion of why assumption \eqref{equ:condipsi} is needed in our proof.
\end{remark}

\setlength{\parskip}{2pt}

\begin{remark}\label{rmk:kw}
The function $F_{\psi}(t)=1-t\psi(t)$ encodes $\psi
$ via $\psi(t)=\frac{1-F_{\psi}(t)}{t}$. In view of the assumptions \eqref{equ:con1psi} and \eqref{equ:con2psi}, $F_{\psi}
$ is a decreasing function and takes values in $(0,1)$. In particular, the limit $\lim_{t\to\infty}F_{\psi}(t)$ exists and lies in $[0,1)$. 

If $\lim_{t\to\infty}F_{\psi}(t)>0$,
then the conclusion of Theorem \ref{thm:improvedirichleconwei} follows from
the work of Kleinbock and Weiss \cite{KleinbockWeiss2008}.
Indeed, in this case 
the series \eqref{equ:crticalseriesconj} diverges,
and for any fixed $\beta$ we have
$\sum_{t_0\leq k\leq t_1}k^{-1}F_{\psi}(k)^{\varkappa_d}\log^{\beta}\bigl(\frac{1}{F_{\psi}(k)}\bigr)\asymp\log t_1$
as $t_1\to\infty$, so that also the assumption \eqref{equ:condipsi} holds;
moreover, $\lim_{t\to\infty}F_{\psi}(t)>0$ implies that there exists some $c\in(0,1)$
%
such that $\psi(t)<c/t$, implying that $\DI_{\bm{\alpha},\bm{\beta}}(\psi)\subset \DI_{\bm{\alpha},\bm{\beta}}(c\psi_1)$; 
and by
\cite[Theorem 1.4]{KleinbockWeiss2008}, $\DI_{\bm{\alpha},\bm{\beta}}(c\psi_1)$ is a null set.
%
%
\end{remark}

\setlength{\parskip}{6pt}

\begin{remark}\label{rmk:examples}
Let us give some explicit examples to illustrate our results. 
We note that each function $\psi$ appearing below is
strictly decreasing on $[t_0,\infty)$ for $t_0$ sufficiently large.
\begin{itemize}
\item[(1)]
Let $\psi(t)=\frac{1-c(\log t )^{-\tau}}{t}$ ($\Leftrightarrow F_{\psi}(t)=c(\log t )^{-\tau}$) for some $c>0$ and $\tau\geq 0$. 
In this case the series \eqref{equ:crticalseriesconj} diverges if and only if $\tau\leq  \frac{1}{\varkappa_d}$. It is also easy to check that condition \eqref{equ:condipsi} is satisfied whenever $\tau\leq \frac{1}{\varkappa_d}$. 
Hence Theorem \ref{thm:improvedirichleconwei} implies that for such $\psi$, $\DI_{\bm{\alpha},\bm{\beta}}(\psi)$ 
is of full measure if $\tau>\frac1{\varkappa_d}$,
and of zero measure if $\tau\leq\frac1{\varkappa_d}$.
\item[(2)] Let $\psi(t)=\frac{1-c(\log t )^{-1/\varkappa_d}(\log\log t)^{-\tau}}{t}$ ($\Leftrightarrow F_{\psi}(t)=c(\log t)^{-1/\varkappa_d}(\log\log t)^{-\tau}$) for some $c>0$ and $\tau\in \R$. 
In this case the series \eqref{equ:crticalseriesconj} diverges if and only if $\tau\leq \frac{\lambda_d+1}{\varkappa_d}$,
while the condition \eqref{equ:condipsi} is satisfied if and only if 
$\tau<\frac{\lambda_d}{\varkappa_d}$. 
Hence Theorem \ref{thm:improvedirichleconwei} implies that for such $\psi$, $\DI_{\bm{\alpha},\bm{\beta}}(\psi)$
is of full measure if $\tau>\frac{\lambda_d+1}{\varkappa_d}$,
and of zero measure if $\tau< \frac{\lambda_d}{\varkappa_d}$.
However, for $\tau$ in the range 
$\frac{\lambda_d}{\varkappa_d}\leq \tau\leq \frac{\lambda_d+1}{\varkappa_d}$,
Theorem \ref{thm:improvedirichleconwei} gives \textit{no information}
(although we believe that $\DI_{\bm{\alpha},\bm{\beta}}(\psi)$
is of zero measure also for these $\tau$; cf.\ Remark \ref{technassumptionREM}).
We point out that in the special case $\tau=\frac{\lambda_d}{\varkappa_d}$,
the quotient in \eqref{equ:condipsi} remains \textit{bounded}
as $t_1\to\infty$; in this case our method of proof allows us to conclude that
at least the set $\DI_{\bm{\alpha},\bm{\beta}}(\psi)$ is not of full 
Lebesgue measure;
see Remark~\ref{rmk:final}.
\end{itemize}
\end{remark}

\begin{remark}
Let us also point out that the assumption \eqref{equ:con2psi} is imposed only to avoid making the statement of Theorem \ref{thm:improvedirichleconwei} unnecessarily complicated (since otherwise $F_{\psi}(t)$ could be negative and then the series \eqref{equ:crticalseriesconj} is not well-defined). Indeed, if \eqref{equ:con2psi} fails but $\psi$ satisfies the other assumptions in Theorem \ref{thm:improvedirichleconwei}, then $\DI_{\bm{\alpha},\bm{\beta}}(\psi)$ is certainly of full Lebesgue measure. This is true since in this case 
after possibly enlarging $t_0$, we have $\psi(t)\geq1/t\geq\frac{1-(\log t)^{-\tau}}t$ for all $t\geq t_0$ and $\tau\geq 0$, and therefore $\DI_{\bm{\alpha},\bm{\beta}}(\psi)\supset \DI_{\bm{\alpha},\bm{\beta}}\left({t\mapsto}\frac{1-(\log t)^{-\tau}}t\right)$, where the last set is of full Lebesgue measure whenever $\tau>\frac{1}{\varkappa_d}$ by Remark \ref{rmk:examples}(1).
\end{remark}

One of the main ingredients in our proof of Theorem \ref{thm:improvedirichleconwei}
is a measure estimate in geometry of numbers, which we believe is of independent interest.
Let $\mu_d$ be the unique left $\SL_d(\R)$-invariant probability measure on $X_d=\SL_d(\R)/\SL_d(\Z)$.
We are interested in the sets 
$\Delta^{-1}[0, r]$ in $X_d$, as $r\to0^+$.
As we have discussed, these sets shrink toward the critical locus $\Delta^{-1}\{0\}$
as $r\to0^+$,
and by Haj\'os's Theorem \cite{Hajos1941} 
(cf.\ Theorem \ref{thm:hajosmin} below),
the set $\Delta^{-1}\{0\}$ has a simple explicit description 
as a finite union of
compact submanifolds of positive codimension $\frac{d^2+d}2-1=\varkappa_d+1$ 
in $X_d$.
In particular this implies that $\mu_d\bigl(\Delta^{-1}\{0\}\bigr)=0$
and $\mu_d\left(\Delta^{-1}[0, r]\right)\to0$ as $r\to0^+$.
The following theorem gives an asymptotic estimate 
on the exact rate of convergence in the limit just mentioned. 
\begin{thm}\label{thm:mainthm}
We have 
\begin{align}\label{thm:mainthmRES}
\mu_d\left(\Delta^{-1}[0, r]\right)\asymp_d r^{\frac{(d-1)(d+2)}{2}}\log^{\frac{d(d-1)}{2}}\Bigl(\frac{1}{r}\Bigr)=r^{\varkappa_d+1}\log^{\lambda_d}\Bigl(\frac{1}{r}\Bigr),\qquad \textrm{as $r\to 0^+$}.
\end{align}
\end{thm}
Our proof of Theorem \ref{thm:mainthm} proceeds 
by bounding the sets $\Delta^{-1}[0,r]$ from above and below by 
more explicit sets whose Haar measure we can estimate directly.
In the proof of the upper bound 
we make crucial use of Haj\'os's Theorem.
We remark that
Haj\'os's proof (from 1941) of the theorem,
which settled a conjecture 
of Minkowski from 1896,
is surprisingly complicated, with the first step being 
a translation of the question
into an algebraic statement about factorizations of finite abelian groups
(see also \cite{SteinSzabo1994} for a nice presentation).
It seems difficult to extend this proof in any direct way
from the case of 
$\Delta^{-1}\{0\}$
to 
deduce restrictions on the sets 
$\Delta^{-1}[0, r]$ which are sufficiently strong to imply the desired upper bound on 
$\mu_d\left(\Delta^{-1}[0, r]\right)$.
Instead we apply Haj\'os's Theorem, in 
combination with a compactness argument,
to obtain a convenient containment relation for 
$\Delta^{-1}[0, r]$ valid for \textit{all} sufficiently small $r$
(see Lemma \ref{MinkowskiHajosCor2} and Remark \ref{MinkowskiHajosCor2REM}).
This initial restriction 
serves as the starting point for our analysis 
where we use direct, geometric arguments to derive further, $r$-dependent restrictions
on $\Delta^{-1}[0, r]$ for $r$ small,
strong enough to finally imply the desired upper bound
on $\mu_d\left(\Delta^{-1}[0, r]\right)$.

\begin{remark}
In the case $d=2$, the following explicit formula holds
\cite[p.\ 74]{StrombergssonVenkatesh2005}:
\begin{align*}
\mu_2\left(\Delta^{-1}[0, r]\right)=\begin{cases}
\frac{12}{\pi^2}\bigl(e^{-2r}+2r-1+\Psi(e^{-2r})\bigr)&\text{if }\:0\leq r\leq\frac12\log2,
\\[3pt]
1-\frac{12}{\pi^2}e^{-2r}&\text{if }\:r\geq\frac12\log2,
\end{cases}
\end{align*}
where the function $\Psi: (0,1]\to\R$ is defined by $\Psi(1)=0$
and $\Psi'(x)=(x^{-1}-1)\log(x^{-1}-1)$.
It follows that in this case 
we have an explicit asymptotic expansion sharpening \eqref{thm:mainthmRES}:
\begin{align*}
\mu_2\left(\Delta^{-1}[0, r]\right)
\sim\frac{24}{\pi^2}r^2\log\Bigl(\frac{1}{r}\Bigr)+\frac{12(3-2\log2)}{\pi^2}r^2
-\frac{16}{\pi^2}r^3\log\Bigl(\frac1r\Bigr)+\cdots 
\qquad\text{as }\: r\to0^+.
\end{align*}
The explicit formula for 
$\mu_2\left(\Delta^{-1}[0, r]\right)$,
stated in a different notation, 
was independently obtained in \cite{KleinbockYu2020}  using a different method.
%
\end{remark}

\begin{remark}
Theorem \ref{thm:mainthm} is also relevant for the study of the 
Hausdorff dimension of the set $\DI_{m,n}(c\psi_1)$.
As we have mentioned,
Kleinbock and Mirzadeh 
recently proved that the Hausdorff dimension of $\DI_{m,n}(c\psi_1)$ is less than $mn$ 
for every $0<c<1$ \cite[Theorem 1.5]{dKsM2020}.
They derived this as an application of their main result, 
\cite[Theorem 1.2]{dKsM2020}, which gives
an explicit upper bound on the Hausdorff dimension of a certain kind of dynamically defined subsets in the space $X_d$.
It seems that by using 
Theorem \ref{thm:mainthm} (cf.\ also Theorem \ref{thm:volumethickening} below),
together with a further analysis of the quantities appearing in 
\cite[Theorem 1.2]{dKsM2020},
it should be possible to sharpen the conclusion of \cite[Theorem 1.5]{dKsM2020}
into a bound of the form
\begin{align*}
\dim_H\bigl(\DI_{m,n}(c\psi_1)\bigr)<mn-\delta (1-c)^{\varkappa_d}\log^{\lambda_d-1}\bigl((1-c)^{-1}\bigr)
\end{align*}
for all $c<1$ sufficiently near $1$, where 
$d=m+n$ and $\delta>0$ is a constant which only depends on $m,n$.
\end{remark}

\subsection{Discussion of the proof of Theorem \ref*{thm:improvedirichleconwei}}\label{sec:discussion}
We next give 
a more detailed outline of our proof of Theorem \ref{thm:improvedirichleconwei}. 
For simplicity of presentation, we will only focus on the special case when $\bm{\alpha}=(\tfrac{1}{m},\ldots, \tfrac{1}{m})$ and $\bm{\beta}=(\tfrac{1}{n},\ldots,\tfrac{1}{n})$;
we comment in Remark~\ref{rmk:effequdoumixww} below on the modifications needed to treat general weights.

We start from the Dani Correspondence, \eqref{equ:dyintkw}, and discretize the 
shrinking target problem which appears there 
by introducing the following thickened 
targets:
\begin{align*}
B_k:=\bigcup_{0\leq s<1}a_{-s}\Delta^{-1}[0, r(k+s)],\qquad \textrm{for any integer}\ k>s_0.
\end{align*} 
It follows from this definition that for any $\Lambda\in X_d$, $a_k\Lambda\in B_k$ if and only if $a_s\Lambda\in \Delta^{-1}[0,r(s)]$ for some $k\leq s<k+1$. In particular, by \eqref{equ:dyintkw},
$A\in M_{m,n}(\R)$ is not $\psi$-Dirichlet
if and only if $a_k\Lambda_A\in B_k$ for infinitely many integers $k$. For any $k>s_0$ let us define 
\begin{align*}
E_k:=\left\{\Lambda_A\in \scrY\col a_k\Lambda_A\in B_k\right\}.
\end{align*}
Then, in view of the previous discussion and the identification
$\mathcal{Y}\cong M_{m,n}(\R/\Z)$,
we have
$$\DI^c_{m,n}(\psi)=\limsup_{k\to\infty}E_k.$$
Hence, 
by the Borel-Cantelli lemma, 
\begin{align*}
\left\lbrace\begin{array}{ll} 
\sum_k\textrm{Leb}(E_k)< \infty\ & \Longrightarrow \textrm{Leb}\left(\DI^c_{m,n}(\psi)\right)=0;\\
\sum_k\textrm{Leb}(E_k)= \infty\ \&\ \textrm{``quasi-independence conditions''} & \Longrightarrow \textrm{Leb}\left(\DI^c_{m,n}(\psi)\right)=1.
\end{array}\right.
\end{align*}

We thus need to understand 
when the sum $\sum_k\textrm{Leb}(E_k)$ diverges or converges, respectively. 
It follows from our definitions that 
$$
\textrm{Leb}(E_k)=\int_{\scrY}\chi_{B_k}(a_k\Lambda_A)\,\d A.
$$


{It is well known that the $a_s$-translates $a_s\scrY$ equidistribute in $X_d$ as $s\to\infty$;
since our shrinking target $B_k$ varies in the parameter $k$, we require an \textit{effective} version of 
this fact. Such a result was first proved by Kleinbock and Margulis \cite[Proposition 2.4.8]{KleinbockMargulis1996} building on Margulis's original thickening argument \cite{Margulis2004} and the exponential mixing of the diagonal flow $\{a_s\}_{s\in \R}$.}
Here we use the following explicit version (see Corollary \ref{cor:effequhorsp} below): 
there exists $\delta>0$ 
such that for any $f\in C_c^{\infty}(X_d)$ and any $s>0$,
\begin{align}\label{equ:effequidww}
\int_{\scrY}f(a_s\Lambda_A)\,\d A=\mu_d(f)+O\big(e^{-\delta s}\scrN(f)\big),
\end{align}
where the norm $\scrN(\cdot)$ is the maximum of a Lipschitz norm and a Sobolev $L^2$-norm (see Section \ref{sec:effequdomix}).
By approximating $\{\chi_{B_k}\}_{k>s_0}$ from above and below by smooth functions
and applying \eqref{equ:effequidww} together with an
explicit bound on the norm $\scrN(\cdot)$ (see Lemma \ref{lem:smapprox}),
it follows that (see Lemma~\ref{lem:auxest})
\begin{align*}
\sum_k\textrm{Leb}(E_k)=\infty\qquad\Longleftrightarrow\qquad \sum_k\mu_d\left(B_k\right)=\infty.
\end{align*}
Furthermore, it is not difficult to see from Theorem \ref{thm:mainthm} that the series 
$\sum_k\mu_d\left(B_k\right)$ diverges
if and only if the series in \eqref{equ:crticalseriesconj} diverges (see Theorem \ref{thm:volumethickening} and Lemma \ref{lem:danicor}). 
This in particular settles the convergence case of Theorem \ref{thm:improvedirichleconwei}.

For the divergence case, in addition to the assumption that the series in \eqref{equ:crticalseriesconj} diverges (which implies that $\sum_{k}\textrm{Leb}(E_k)=\infty$), one also needs to establish a certain \textit{quasi-independence} condition, see \eqref{equ:quasiinde}. Roughly speaking, we need to show that the quantities 
$$
\left|\textrm{Leb}\left(E_{i}\cap E_{j}\right)-\textrm{Leb}(E_{i})\textrm{Leb}(E_{j})\right|,\qquad i\neq j >s_0
$$ 
are small on average. Here note that
\begin{align*}
\textrm{Leb}\left(E_{i}\cap E_{j}\right)=\int_{\scrY}\chi_{B_i}(a_i\Lambda_A)\chi_{B_j}(a_j\Lambda_A)\,\d A.
\end{align*}
We now apply the \textit{effective doubly mixing} for the $a_s$-translates $\{a_s\scrY\}_{s>0}$.
This result is due to Kleinbock-Shi-Weiss \cite[Theorem 1.2]{KleinbockShiWeiss2017}; we use a more explicit version due to
Bj\"orklund-Gorodnik \cite[Corollary 2.4]{BjorklundGorodnik2019} which 
states that for any $f_1,f_2\in C_c^{\infty}(X_d)$ and any $s_1,s_2>0$,
\begin{align}\label{equ:effedobmixdww}
\int_{\scrY}f_1(a_{s_1}\Lambda_A)f_2(a_{s_2}\Lambda_A)\,\d A=\mu_d(f_1)\mu_d(f_2)
+O\left(e^{-\delta D(s_1,s_2)}\cN(f_1)\scrN(f_2)\right),
\end{align}
where $D(s_1,s_2):=\min\{s_1,s_2, |s_1-s_2|\}$. 
Combining this result with \eqref{equ:effequidww} we get
\begin{align*}
\left|\int_{\scrY}f_1(a_{s_1}\Lambda_A)f_2(a_{s_2}\Lambda_A)\,\d A-\textrm{Leb}(f_1)\textrm{Leb}(f_2)\right|\ll
e^{-\delta D(s_1,s_2)}\prod_{i=1}^2\max\left\{\cN(f_i),\mu_d(f_i)\right\}.
\end{align*}
Finally, by approximating $\{\chi_{B_k}\}_{k>s_0}$ from below by smooth functions, applying the above estimate (together with a trivial estimate when $D(s_1,s_2)$ is small, see \eqref{equ:trivialest}) and the bounds on the norm $\cN(\cdot)$ (see Lemma \ref{lem:smapprox}), 
we show that the divergence of the series in \eqref{equ:crticalseriesconj}
together with the additional technical assumption \eqref{equ:condipsi}, implies that 
the required quasi-independence condition \eqref{equ:quasiinde} is satisfied,
thus concluding the proof of 
the divergence case of Theorem \ref{thm:improvedirichleconwei}. 
%
%

We end our discussion {with} a few remarks. 

\setlength{\parskip}{2pt}

\begin{remark}
Our argument should be compared to that of Kleinbock and Margulis \cite{KleinbockMargulis1999},
where the shrinking targets are certain cusp neighborhoods: 
In \cite{KleinbockMargulis1999} the relevant shrinking target problem 
is first solved for the case of $a_s$-orbits starting at $\mu_d$-generic 
points in the ambient space $X_d$;
for this task it suffices to use,
in place of \eqref{equ:effequidww} and \eqref{equ:effedobmixdww}
respectively, the invariance of the measure $\mu_d$ 
and the exponential mixing of the $a_s$-flow.
Then by an approximation argument \cite[\S 8.7]{KleinbockMargulis1999},
the shrinking target property for $\mu_d$-generic points in $X_d$
is shown to imply the same property for generic points in the submanifold $\scrY$.
A key observation in this approximation step
is that \textit{all} shrinking targets, by virtue of being cusp neighborhoods,
remain essentially unaffected by perturbations 
from a \textit{fixed} neighborhood of the identity in the
neutral leaf of the $a_s$-flow, i.e.\ the centralizer of the $a_s$-flow in $\SL_d(\R)$. 
This, however, is no longer
the case in our setting, with the shrinking targets being compact sets. 
This is why we take the more direct approach using effective equidistribution and doubly mixing of the $a_s$-translates 
of $\scrY$, 
that is, \eqref{equ:effequidww} and \eqref{equ:effedobmixdww}.

One potential advantage of this more direct approach is that 
if \eqref{equ:effequidww} 
could be refined by
replacing the measure $\text{Leb}$ by 
a natural measure on some \textit{submanifold} of $\scrY$,
then by mimicking our analysis,
one could establish the $\psi$-Dirichlet property for generic points in that submanifold,
for any $\psi$ such that 
\eqref{equ:crticalseriesconj} converges.
See Remark \ref{rmk:samlei} below for a discussion
of the application along these lines
of a recent effective equidistribution result obtained by Chow and Yang \cite{ChowYang2019}.

We note that the use of equidistribution of $a_s$-translates of $\scrY$
in the study of the Dirichlet improvability problem is not new;
it has been applied several times in the more well-studied setting of
Dirichlet improvable vectors and matrices.
For $\min\{m,n\}=1$ and $\scrI\subset \scrY$ being an analytic curve in $\scrY$ satisfying certain explicit conditions, 
Shah \cite[Theorem 1.2]{Shah2009} proved that the $a_s$-translates of $\scrI$ 
equidistribute in $X_d$ with respect to $\mu_d$ 
as $s\to\infty$.
Shah's proof relies on Ratner's classification of measures invariant under unipotent flows
\cite{Ratner1991},
and his equidistribution theorem is not effective;
still it suffices for the deduction of the fact that generic points on the curve $\scrI$
are Dirichlet non-improvable, that is, lie outside of the set \eqref{equ:DI}.
(This is so since in this case, the relevant ``shrinking'' target is in fact a \textit{fixed} set of positive measure.)
Shah's results have been generalized and strengthened in various directions \cite{Shah2010,ShahYang2020,Yang2020,KleinbockSaxceShahYang2021}. In a recent breakthrough of Khalil and Luethi \cite{KhalilLuethi2021}, the authors refined \eqref{equ:effequidww} (for the case when $n=1$) by replacing $\text{Leb}$ {with} a certain fractal measure, from which they deduced a complete analogue of Khintchine's theorem with respect to this fractal measure. 
\end{remark}

\setlength{\parskip}{6pt}

\begin{remark}\label{rmk:teassu}
Another difficulty, which also stems from the fact that our targets are shrinking compact sets,
is the fact that the norm $\scrN(\cdot)$
unavoidably {grows (polynomially)} for the smooth functions
approximating the shrinking targets from above and below
(see Lemma \ref{lem:smapprox}).
While the impact of this {polynomial growth} of the norm can be eliminated in the convergence case due to the exponential decay in the parameter $s$ (i.e.\ the factor $e^{-\delta s}$ in the error term in \eqref{equ:effequidww}), it causes serious problems in the divergence case,
and this is exactly why we need to impose the extra assumption \eqref{equ:condipsi}. 
Let us here also note that this assumption \eqref{equ:condipsi} can be rephrased in terms of the measure of the shrinking targets 
as follows:
\begin{align*}
\liminf_{s_1\to\infty}\frac{\sum_{s_0<k\leq s_1}\mu_d(B_k)\log\left(\tfrac{1}{\mu_d(B_k)}\right)}{\left(\sum_{s_0<k\leq s_1}\mu_d(B_k)\right)^2}=0.
\end{align*}

\end{remark}

\begin{remark}\label{rmk:effequdoumixww}
In order to extend the argument outlined above to the case of
general weight vectors $\bm{\alpha}$ and $\bm{\beta}$, we have to consider a more general one-parameter flow 
$\{g_s\}_{s>0}\subset \SL_d(\R)$ associated to $\bm{\alpha}$ and $\bm{\beta}$ (see \eqref{equ:geneweflow}),
and use a dynamical interpretation of $\psi_{\bm{\alpha},\bm{\beta}}$-Dirichlet matrices 
which involves this $g_s$-flow
and 
generalizes \eqref{equ:dyintkw};
see Proposition \ref{prop:dyint} and Remark \ref{rmk:gen}.
We therefore need 
analogous effective equidistribution and doubly mixing results for the $g_s$-translates of $\scrY$. 
Fortunately, such more general (and considerably more difficult) effective results are known to hold, thanks to the work of Kleinbock-Margulis \cite[Theorem 1.3]{KleinbockMargulis2012} and, again, Kleinbock-Shi-Weiss \cite[Theorem 1.2]{KleinbockShiWeiss2017} 
and Bj\"orklund-Gorodnik \cite[Corollary 2.4]{BjorklundGorodnik2019}
(see Theorem \ref{thm:higherordermixing} below).
In fact in \cite{BjorklundGorodnik2019} a uniform treatment was given proving effective mixing of arbitrary order for 
the $g_s$-translates of $\scrY$;
however we will not make use of this.
\end{remark}

\subsection*{Notation and conventions}
Throughout the paper, the notation $\|\cdot\|$ denotes the supremum norm on various Euclidean spaces or matrix spaces (which can be viewed as Euclidean spaces on the matrix entries). Let $I\subset \R$ be an interval (not necessarily bounded). A function $f : I\to \R$ is called \textit{increasing} (resp.\ \textit{decreasing}) if $f(t_1)\leq f(t_2)$ (resp.\ $f(t_1)\geq f(t_2)$) whenever $t_1<t_2$. All the vectors in this paper are column vectors. For two positive quantities $A$ and $B$, we will use the notation $A\ll B$  {or $A=O(B)$} to mean that there is a constant $c>0$ such that $A\leq cB$, and we will use subscripts to indicate the dependence of the constant on parameters. We will write $A\asymp B$ for $A\ll B\ll A$. 

\noindent\textbf{Acknowledgements}
We are grateful to the anonymous referee for useful comments.

\section{Some preliminaries for Theorem \ref*{thm:mainthm}}\label{sec:premest}
Fix  an integer $d\geq 2$. In what follows we always denote $G=\SL_d(\R)$, $\G=\SL_d(\Z)$ and $X_d=G/\G$ the space of unimodular lattices in $\R^d$. Let $\mu_d$ be the unique $G$-invariant probability measure on $X_d$. Let $\Delta: X_d\to [0,\infty)$ be the function on $X_d$ defined as in \eqref{def:delta}. In this section, we collect some preliminary results for our proof of Theorem \ref{thm:mainthm}. In fact, for simplicity of presentation we will prove an equivalent measure estimate result. For any $r\in [0,1)$ let $\mathcal{C}_r\subset \R^d$ be the open ``${(1-r)}$-cube'', i.e. 
$$\mathcal{C}_r:=(r-1,1-r)^d.$$
Let $K_r\subset X_d$ be the set of unimodular lattices having no nonzero points in $\mathcal{C}_r$, i.e.
$$K_r:=\big\{\Lambda\in X_d\col \Lambda\cap \mathcal{C}_r=\{\bm{0}\}\big\}.$$
We note that by definition of $\Delta$, $K_r=\Delta^{-1}[0, -\log\left(1-r\right)]$, or equivalently, $\Delta^{-1}[0, r]=K_{1-e^{-r}}$. Since $1-e^{-r}=r+O(r^2)\asymp r$ for all $r\in (0, 1)$, Theorem \ref{thm:mainthm} 
can be equivalently restated as follows. 
\begin{thm}\label{thm:measestmainaux}
Let $\varkappa_d=\frac{d^2+d-4}{2}$ and $\lambda_d=\frac{d(d-1)}{2}$ be as in Theorem \ref{thm:improvedirichleconwei}. Then
\begin{align*}
\mu_d\left(K_r\right)\asymp_d 
r^{\varkappa_d+1}\log^{\lambda_d}\Bigl(\frac{1}{r}\Bigr),\qquad \textrm{as $r\to 0^+$}.
\end{align*}
\end{thm}

We will prove Theorem \ref{thm:measestmainaux} by proving a lower bound and an upper bound separately. 

\subsection{Haar measure and coordinates}\label{sec:haar}

Let $P< G$ be the maximal parabolic subgroup fixing the line spanned by $\bm{e}_d\in \R^d$, and let $N<G$ be the transpose of the unipotent radical of $P$. Here and {hereafter}, $\{\vece_i: 1\leq i\leq d\}$ denotes the standard orthonormal basis of $\R^d$. Explicitly, 
\begin{align*}
P=\{p\in G\col 
p\bm{e}_d=t\bm{e}_d\ \textrm{for some $t\neq 0$}\},
\end{align*}
and
\begin{align*}
N=\left\{u_{\bm{x}}:=\left(\begin{smallmatrix}
I_{d-1} & \bm{x}\\
\bm{0}^t & 1\end{smallmatrix}\right)\col \bm{x}\in \R^{d-1}\right\}.
\end{align*}
For any $p\in P$, let $\vecb_1:=p\bm{e}_1,\ldots, \bm{b}_{d-1}:=p\bm{e}_{d-1}$ be the first $d-1$ column vectors of $p$. We note that $p$ is uniquely determined by $\bm{b}_1,\ldots, \bm{b}_{d-1}$; we will sometimes denote $p\in P$ by $p_{\bm{b}_1,\ldots, \bm{b}_{d-1}}$ to indicate this dependence.
For any $g\in G$, let us denote by $\tilde{g}\in M_{d-1,d-1}(\R)$ the top left $(d-1)\times (d-1)$ block of $g$. If $\det \tilde{g} \neq 0$, then $g$ can be written uniquely as a product
\begin{align}\label{puCOORD}
g=p_{\bm{b}_1,\ldots, \bm{b}_{d-1}}u_{\bm{x}}\quad{\text{for some $p_{\bm{b}_1,\ldots, \bm{b}_{d-1}}\in P$ and }u_{\bm{x}}\in N}.
\end{align}

Let $\nu$ be the (left and right) Haar measure on $G$, normalized so that it agrees locally with $\mu_d$.
In terms of the coordinates in \eqref{puCOORD}, 
$\nu$ is given by 
\begin{equation}\label{equ:haar}
\d\nu(g)=\frac{1}{\zeta(2)\cdots \zeta(d)}\, \d \vecx\prod_{1\leq i\leq d-1}\d\vecb_i,
\end{equation}
where $\zeta(\cdot)$ is the Riemann zeta function,
and where $\d\vecx$ and $\d\vecb_i$ denote Lebesgue measure on $\R^{d-1}$ and $\R^d$,
respectively.
For later purpose, we also note that the lattice $\Lambda$ represented by $p_{\bm{b}_1,\ldots, \bm{b}_{d-1}}u_{\bm{x}}$, i.e. $\Lambda=p_{\bm{b}_1,\ldots, \bm{b}_{d-1}}u_{\bm{x}}\Z^d$, has a basis 
$$
\Lambda=\Z\bm{b}_1\oplus  \cdots\oplus \Z\bm{b}_{d-1}\oplus \Z\bm{b}_d,
$$
where $\bm{b}_d:=\sum_{j=1}^{d-1}x_j\bm{b}_j+(\det \tilde{p})^{-1}\bm{e}_d$ is the $d$-th column  vector of the matrix $p_{\bm{b}_1,\ldots, \bm{b}_{d-1}}u_{\bm{x}}$. Here $\tilde{p}$ is the top left $(d-1)\times (d-1)$ block of $p_{\bm{b}_1,\ldots, \bm{b}_{d-1}}$. 

For our computation of the upper bounds, it will be more convenient to use another set of coordinates: For any $g=(g_{ij})_{1\leq i,j\leq d}\in G$ with $\det \tilde{g} \neq 0$, as mentioned above, we can write $g$ uniquely as {in \eqref{puCOORD}}.
It is clear from this relation that $g$ and $p_{\bm{b}_1,\ldots,\bm{b}_{d-1}}$ share the same first $d-1$ column vectors, i.e. $g\bm{e}_j=\bm{b}_j$ for all $1\leq j\leq d-1$. Moreover, as noted above, for the $d$-th column vector we have
\begin{align*}
g\bm{e}_d=\sum_{j=1}^{d-1}x_j\bm{b}_j+(\det \tilde{p})^{-1}\bm{e}_d=\sum_{j=1}^{d-1}x_j\left(g\bm{e}_j\right)+ (\det \tilde{g})^{-1}\bm{e}_d.
\end{align*}
In particular, we have $(g_{1d},\ldots, g_{d-1,d})^t=\tilde{g}\bm{x}$, which further implies 
$$
\d\vecx=(\det \tilde{g})^{-1}\prod_{1\leq i\leq d-1}\d g_{id}.
$$
This relation, together with the relations $g\bm{e}_j=\bm{b}_j, 1\leq j\leq d-1$ and the Haar measure description \eqref{equ:haar}, 
immediately implies the following:
\begin{lem}\label{lem:haar2}
For any (Borel) subset $\cK$ of 
$\bigl\{g\in G\col|\det \tilde{g} -1|<\frac12\bigr\}$,
we have
\begin{align}\label{equ:haar2}
\nu(\cK)\asymp_d\int_{\cK}\prod_{\substack{1\leq i,j\leq d\\ (i,j)\neq (d,d)}}\d g_{ij}.
\end{align}
\end{lem}

\subsection{A small parameter for the lower bound}\label{sec:smpara}
To prove the lower bound, we will construct a subset {of} $K_r$ whose measure is of the same magnitude as $K_r$. For a lattice $\Lambda=g\Z^d\in X_d$, to show $\Lambda\in K_r$, by definition one needs to show $g\bm{m}\notin \mathcal{C}_r$ for all nonzero $\bm{m}\in \Z^d$. If $g\in G$ is sufficiently close to the identity element $I_d\in G$, so that $\Lambda$ has a basis close to the standard basis $\{\bm{e}_i: 1\leq i\leq d\}$, then one only needs to consider vectors $\bm{m}\in \Z^d$ with small supremum norms. For this reason, we will only focus on lattices that are close to $\Z^d$. 
Recall that the set $K_r$ certainly does not get concentrated near the lattice $\Z^d$ as $r\to0^+$;
indeed, we have $\cap_{r>0}K_r=K_0=\Delta^{-1}\{0\}$,
which as we have mentioned is a finite union of compact submanifolds of positive codimension $\varkappa_d+1$ in $X_d$
(see also Section \ref{Hajossec}).
The fact that it still suffices to consider a small neighborhood of $\Z^d$ when proving the \textit{lower} bound
in Theorem \ref{thm:measestmainaux}
is related to the fact that the \textit{mass}
of $K_r$ (with respect to $\mu_d$)
becomes concentrated near the lattice $\Z^d$ as $r\to0^+$;
see Remark \ref{bulkinZdnbhdREM}.

Explicitly, we fix a small norm ball in $X_d$ around $\Z^d$ as follows: 
For any $c>0$, let
\begin{align}\label{equ:normball}
G_c:=\left\{g\in G\col \|g-I_d\|< c\right\}
\end{align}
be the open ball in $G$ of radius $c$, centered at $I_d$ with respect to the supremum norm on $M_{d, d}(\R)$. 
Let $\pi: G\to X_d$ be the natural projection from $G$ to $X_d$. 
We fix a parameter $c_0\in(0,1)$ (which only depends on $d$) 
so small that $\pi|_{G_{c_0}}$ is injective 
and, for any vectors $\vecb_1,\ldots,\vecb_d\in\R^d$ satisfying 
$\|\vecb_i-\vece_i\|<c_0$ for all $1\leq i\leq d$,
every hyperplane of the form
\begin{align*}
m\vecb_j+\sum_{\substack{1\leq\ell\leq d\\ \ell\neq j}}\R\vecb_{\ell}
\end{align*}
with $m\in\R$, $|m|\geq2$ and $1\leq j\leq d$,
is disjoint from the cube $[-1,1]^d$.
In particular, if $\Lambda=g\Z^d$ for some $g\in G_{c_0}$, then $\Lambda$ has a basis $\{g\bm{e}_i: 1\leq i\leq d\}$ satisfying $\|g\bm{e}_i-\bm{e}_i\|\leq c_0$ for all $1\leq i\leq d$.
It follows that in order to prove that $\Lambda\in K_r$ for a given $r\in(0,1)$,
it suffices to verify that $\sum_{1\leq i\leq d}m_i\left(g\bm{e}_i\right)\notin \mathcal{C}_r$ for all 
$\bm{m}=(m_1,\ldots, m_d)\in \{-1,0,1\}^d\smallsetminus \{\bm{0}\}$.



\subsection{Haj\'os's Theorem and its consequences}\label{Hajossec}
Recall that
$$
K_0=\left\{\Lambda\in X_d\col \Lambda\cap (-1,1)^d=\{\bm{0}\}\right\}.
$$ 
As we mentioned in the introduction, 
the explicit description of $K_0$ was conjectured 
(and proved in two and three dimensions) by Minkowski,
and proved in full generality by Haj\'os in 1941 \cite{Hajos1941}:
\begin{thm}[Haj\'os]\label{thm:hajosmin}
Let $U$ be the subgroup of upper triangular unipotent matrices in $G$. Let $W$ be the subgroup of permutation matrices in $\GL_d(\Z)$. Then
\begin{align*}
K_0=\bigcup_{w\in W}\left(w Uw^{-1}\right)\Z^d.
\end{align*}
\end{thm}
If we set
\begin{align}\label{equ:ufudo}
U_0=\left\{(u_{ij})\in U\col -\tfrac12<u_{ij}\leq\tfrac12\text{ for all }1\leq i<j\leq d\right\}
\end{align}
so that $U_0$ is a fundamental domain for $U/(\Gamma\cap U)$, 
then we have the  following immediate corollary of Haj\'os's Theorem.
\begin{cor}\label{MinkowskiHajosThm}
Given any $\Lambda\in K_0$, 
there exist $w\in W$ and $u\in U_0$ such that $\Lambda=wu\Z^d=wuw^{-1}\Z^d$.
\end{cor}
\begin{proof}
Since $\Lambda\in K_0$, by Theorem \ref{thm:hajosmin} we can find $u'\in U$ and $w\in W$ such that
$\Lambda=wu'w^{-1}\Z^d$; but since $w^{-1}\Z^d=\Z^d$, we have
$\Lambda=wu'\Z^d$. Now using the fact that $U_0$ is a 
fundamental domain for 
$U/(\Gamma\cap U)$,
we can then find $u\in U_0$ such that
$u\Z^d=u'\Z^d$. Thus $\Lambda=wu'\Z^d=wu\Z^d=wuw^{-1}\Z^d$, finishing the proof.
\end{proof}
There is a geometric interpretation of 
$K_0$ in terms of \textit{lattice tilings by unit cubes}
\cite[Ch.\ 1.4]{SteinSzabo1994}:
Let us write $\tfrac12\mathcal{C}_0=(-\tfrac12,\tfrac12)^d$ 
for the unit cube obtained by dilating $\mathcal{C}_0$ by a factor $\tfrac12$.
Then for any $\Lambda\in X_d$, 
the family of cubes
$\bm{v}+\tfrac12\mathcal{C}_0$, with $\bm{v}$ running through the lattice $\Lambda$,
forms a tiling of $\R^d$ (modulo a null set) if and only if $\Lambda\in K_0$.
More generally, for any $r\in[0,1)$ we write $\frac12\scrC_r=\bigl(\frac12(r-1),\frac12(1-r)\bigr)^d$.
Then for any $\Lambda\in X_d$,
the condition $\Lambda\in K_r$, i.e.\ $\Lambda\cap \mathcal{C}_r=\{\bn\}$,
is equivalent {to} the condition that the cubes 
$\vecv+\tfrac12\scrC_r$ ($\bm{v}\in\Lambda$)
are \textit{pairwise disjoint}.
When this holds, we write
\begin{align*}
\fC_{\Lambda,r}:=\Lambda+\tfrac12\scrC_r
\end{align*}
for the union of these disjoint cubes.
This set is used in the statement of 
the following simple bound,
which is of crucial importance in our proof of the upper bound
in Theorem \ref{thm:measestmainaux}.

\begin{lem}\label{fCkeyLEM}
Let $\Lambda\in X_d$ and $r\in(0,\frac14)$ be such that $\Lambda\cap\scrC_r=\{\bn\}$,
and let $\scrU$ be a Borel subset of $\R^d$
which is disjoint from $\fC_{\Lambda,r}$ and
which is contained in some translate of the cube $(0,\frac34)^d$.
Then $\vol(\scrU)<dr$.
\end{lem}
\begin{proof}
The set $\fC_{\Lambda,r}$ is invariant under translation by any vector in $\Lambda$,
and if $\scrF\subset\R^d$ is any fundamental domain for $\R^d/\Lambda$,
then 
\begin{align*}
\vol(\scrF\cap\fC_{\Lambda,r})=\sum_{\vecv\in\Lambda}\vol\left(\scrF\cap(\vecv+\tfrac12\scrC_r)\right)
=\sum_{\vecv\in\Lambda}\vol\left((\scrF-\vecv)+\tfrac12\scrC_r\right)=\vol(\tfrac12\scrC_r)=(1-r)^d,
\end{align*} 
where the first equality holds since the cubes $\vecv+\frac12\scrC_r$ ($\vecv\in \Lambda$) are pairwise disjoint,
and the third 
equality holds since the sets $\scrF-\vecv$ ($\vecv\in\Lambda$)
form a partition of $\R^d$.
Hence
\begin{align*}
\vol(\scrF\smallsetminus\fC_{\Lambda,r})=1-(1-r)^d<dr.
\end{align*}
Next, since $\scrU$ is contained in a translate of $(0,\frac34)^d$,
the difference between any two vectors in $\scrU$ lies in $(-\frac34,\frac34)^d=\scrC_{1/4}\subset\scrC_r$,
and since $\Lambda\cap\scrC_r=\{\bn\}$, 
this implies that the points in $\scrU$ are pairwise inequivalent modulo $\Lambda$.
Hence the set $\big(\scrF\smallsetminus(\Lambda+\scrU)\big)\cup \scrU$ is another fundamental domain for $\R^d/\Lambda$, and it contains $\scrU$. After replacing $\scrF$ by this set, we have $\scrU\subset\scrF$; 
thus $\scrU\subset \scrF\smallsetminus\fC_{\Lambda,r}$,
and hence $\vol(\scrU)\leq\vol(\scrF\smallsetminus\fC_{\Lambda,r})<dr$.
\end{proof}

%

\section{Proof of the lower bound}\label{sec:lowerbound}

We keep the notation introduced in  Section \ref{sec:premest}.
In this section we prove the lower bound in Theorem \ref{thm:measestmainaux}.
We will do this 
by constructing, for every sufficiently small $r$,
an explicit subset $\underline{K}_r\subset X_d$ which we will show satisfies 
\begin{align*}
\underline{K}_r\subset K_r\qquad \textrm{and}\qquad \mu_d(\underline{K}_r)\gg_d 
r^{\varkappa_d+1}\log^{\lambda_d}\Bigl(\frac{1}{r}\Bigr).
\end{align*}

We start by giving 
a family of conditions which ensures that a given lattice is contained in $K_r$.
Recall from Section \ref{sec:smpara} that
$c_0\in(0,1)$ is a small fixed parameter 
with the property that for any 
$g\in G_{c_0}$ and $0<r<1$, we have $g\Z^d \in K_r$ if and only if
$\sum_{1\leq j\leq d}m_j\left(g\bm{e}_j\right)\notin \mathcal{C}_r$ for all $\bm{m}=(m_1,\ldots, m_d)^t\in \{-1,0,1\}^d\smallsetminus \{\bm{0}\}$.
\begin{lem}\label{Lem:general}
Let $c_0\in(0,1)$ be as above,
let $r\in (0,c_0/d)$ and let $\Lambda=p_{\bm{b}_1,\dots,\bm{b}_{d-1}}u_{\bm{x}}\Z^d\in X_d$ with 
$\bm{b}_j=(b_{1j},\ldots,b_{dj})^t\in\R^d$ ($j=1,\ldots,d-1$)
and $\bm{x}\in (0,c_0/d)^{d-1}$. 
Let $\tilde{p}=(b_{ij})_{1\leq i,j\leq d-1}$ be the top left $(d-1)\times (d-1)$ block of $p_{\bm{b}_1,\dots,\bm{b}_{d-1}}$ as before. 
Suppose $\bm{b}_1,\ldots,\bm{b}_{d-1}, \bm{x}$ further satisfies
\begin{align}\label{Fs3dimpfLEM1ass1gen}
b_{ij}, -b_{ji}\in (-c_0, 0),\ \forall\ 1\leq j< i\leq d-1; \ b_{d\ell}\in (-c_0, 0), b_{\ell \ell}\in (1-r,1), \  \forall\ 1\leq \ell\leq d-1, 
\end{align}
\begin{align}\label{Fs3dimpfLEM1ass2gen}
b_{ij}<\sum_{k=1}^{j-1} b_{ik},\qquad \forall\ 2\leq j< i\leq d,
\end{align}
\begin{align}\label{Fs3dimpf5gen}
1-r<\det \tilde{p}<(1-r)^{-1}  
\qquad \textrm{and}\qquad \sum_{j=1}^{d-1}|b_{dj}|x_j<(\det \tilde{p})^{-1}-(1-r),
\end{align}
and
\begin{align}\label{equ:newresgen}
b_{ii}+\sum_{j=1}^{d-1}b_{ij}x_j>1-r,\qquad \forall\ 1\leq i\leq d-1.
\end{align}
Then $\Lambda\in K_r$.
\end{lem}
\begin{proof}
Let us set $g:=p_{\bm{b}_1,\dots,\bm{b}_{d-1}}u_{\bm{x}}$
and $\vecb_d=(b_{1d},\ldots,b_{dd})^t:=g\vece_d$;
then $\vecb_j=g\vece_j$ for all $1\leq j\leq d$,
and $\bm{b}_d:=\sum_{j=1}^{d-1}x_j\bm{b}_j+(\det \tilde{p})^{-1}\bm{e}_d$;
in particular $b_{id}=\sum_{j=1}^{d-1}b_{ij}x_j$ for $1\leq i\leq d-1$.
%
%
%
We will start by proving that $g\in G_{c_0}$, i.e. $\|\bm{b}_j-\bm{e}_j\|<c_0$ for all $1\leq j\leq d$. 
In fact, if $1\leq j\leq d-1$ then $\|\bm{b}_j-\bm{e}_j\|<c_0$
is an immediate consequence of
\eqref{Fs3dimpfLEM1ass1gen} and $0<r<c_0/d<c_0$;
thus it remains to show $\|\bm{b}_d-\bm{e}_d\|<c_0$. 
For each $1\leq i\leq d-1$ we have
\begin{align*}
\left|b_{id}\right|=\Biggl|\sum_{j=1}^{d-1}b_{ij}x_j\Biggr|
\leq \sum_{j=1}^{d-1}\left|b_{ij}\right|\left|x_j\right|<\sum_{j=1}^{d-1}\frac{c_0}{d}<c_0,
\end{align*}
where for the second inequality we used 
the assumption that $\bm{x}\in (0,c_0/d)^{d-1}$
and the fact that
$|b_{ij}|<1$ for all $1\leq i,j\leq d-1$, 
which is immediate from \eqref{Fs3dimpfLEM1ass1gen}.
It remains to prove that $1-c_0<b_{dd}<1+c_0$. 
In fact, we have the following stronger bound:
\begin{align}\label{equ:nn}
1-r< b_{dd}<1+c_0.
\end{align}
Indeed, using $b_{dj}<0$ and $x_j>0$ ($1\leq j\leq d-1$)
in combination with \eqref{Fs3dimpf5gen}, we get
\begin{align*}
b_{dd}=\sum_{j=1}^{d-1}b_{dj}x_j+(\det \tilde{p})^{-1}=(\det \tilde{p})^{-1}-\sum_{j=1}^{d-1}\left|b_{dj}\right|x_j>1-r
\end{align*}
as well as
\begin{align*}
b_{dd}=\sum_{j=1}^{d-1}b_{dj}x_j+(\det \tilde{p})^{-1}<(\det \tilde{p})^{-1}<(1-r)^{-1}<1+2r\leq 1+c_0.
\end{align*}
(For the second {to} last inequality
we used the fact that $0<r<c_0/d<1/2$.)
This finishes the proof of \eqref{equ:nn}, and hence $g\in G_{c_0}$. 

Because of $g\in G_{c_0}$,
in order to show $\Lambda\in K_r$,
it suffices to prove that $g\bm{m}=\sum_{1\leq j\leq d}m_j\bm{b}_j\notin \mathcal{C}_r$ for all $\bm{m}\in \{-1,0,1\}^d\smallsetminus\{\bm{0}\}$.
Thus we now let the vector $\bm{m}=(m_1,\ldots, m_d)^t\in \{-1,0,1\}^d\smallsetminus\{\bm{0}\}$ be given,
and 
{show} that $g\bm{m}\notin \mathcal{C}_r$. 

First assume that all the nonzero entries of $\bm{m}$ are of the same sign. 
After 
replacing $\bm{m}$ by $-\bm{m}$ if necessary, we may assume $m_j\geq 0$ for all $1\leq j\leq d$. 
Let $k\in \{1,\ldots, d\}$ be the smallest integer such that $m_k=1$. 
If $k=d$ then 
$\bm{m}=\bm{e}_d$, and thus $g\bm{m}=\bm{b}_d\notin \mathcal{C}_r$ by \eqref{equ:nn}. 
In the remaining case when $k<d$, 
the $k$-th coordinate of $g\bm{m}$ satisfies
\begin{align}\label{allsamesignpf1}
m_kb_{kk}+\sum_{j=k+1}^dm_jb_{kj}\geq b_{kk}+m_db_{kd}\geq \min\{b_{kk}, b_{kk}+b_{kd}\},
\end{align}
where we used the fact that, by \eqref{Fs3dimpfLEM1ass1gen}, $b_{kj}>0$ for each $k<j\leq d-1$. 
Furthermore,
\begin{align}\label{allsamesignpf2}
\min\{b_{kk}, b_{kk}+b_{kd}\}>1-r,
\end{align}
since $b_{kk}>1-r$ by \eqref{Fs3dimpfLEM1ass1gen} 
and $b_{kk}+b_{kd}=b_{kk}+\sum_{j=1}^{d-1}b_{kj}x_j>1-r$ by \eqref{equ:newresgen}.
It follows from \eqref{allsamesignpf1} and \eqref{allsamesignpf2}
that the $k$-th coordinate of $g\bm{m}$ is larger than $1-r$,
and so $g\bm{m}\notin\scrC_r$.
This completes the proof in the case when 
all the nonzero entries of $\bm{m}$ are of the same sign.

It remains to treat the case when 
$\{-1,1\}\subset \{m_j: 1\leq j\leq d\}$.
Then let $1\leq i_1< i_2\leq d$ be the indices which record the 
{latest instance}
when the signs of the entries of $\bm{m}$ change,
i.e.\ the unique indices such that
$m_{i_1}m_{i_2}=-1$, $m_j=0$ for $i_1<j<i_2$ and $m_j\in \{m_{i_2}, 0\}$ for $i_2<j\leq d$
(the last two conditions are void if $i_1+1=i_2$ or $i_2=d$, respectively).
%
%
Again after replacing $\bm{m}$ by $-\bm{m}$ if necessary, we may assume 
$m_{i_1}=-1$, $m_{i_2}=1$, and thus $m_j\geq 0$ for all $i_2<j\leq d$. 
Now consider the $i_2$-th coordinate of $g\bm{m}$ which is
$$\sum_{1\leq j<i_1}m_jb_{i_2,j}-b_{i_2,i_1}+b_{i_2,i_2}+\sum_{i_2<j\leq d}m_jb_{i_2,j}.$$
Here we have
\begin{align}\label{nonsamesignpf1}
\sum_{1\leq j<i_1}m_jb_{i_2,j}-b_{i_2,i_1}\geq
\sum_{1\leq j<i_1}b_{i_2,j}-b_{i_2,i_1}>0,
\end{align}
where the first inequality holds since $|m_j|\leq 1$ and $b_{i_2,j}<0$ for all $j<i_1$,
and the second inequality holds by \eqref{Fs3dimpfLEM1ass2gen}
(except if $i_1=1$; in that case 
\eqref{nonsamesignpf1} simply says that $b_{i_2,i_1}<0$,
which holds by \eqref{Fs3dimpfLEM1ass1gen}).
%
Furthermore,
\begin{align}\label{nonsamesignpf2}
b_{i_2,i_2}+\sum_{i_2<j\leq d}m_jb_{i_2,j}
\quad\begin{cases}
=b_{dd}>1-r&\text{if }\: i_2=d,
\\[2pt]
\geq b_{i_2,i_2}+m_db_{i_2,d}>1-r &\text{if }\: i_2<d,
\end{cases}
\end{align}
where we used \eqref{equ:nn} in the case $i_2=d$,
and in the case $i_2<d$ we used the fact that
$m_j\geq0$ and $b_{i_2,j}>0$ for all $i_2<j\leq d-1$ (if any),
and then applied \eqref{allsamesignpf2} with $k=i_2$.
Adding the inequalities in \eqref{nonsamesignpf1} and \eqref{nonsamesignpf2},
we conclude that the $i_2$-th coordinate of $g\bm{m}$ is larger than $1-r$; hence $g\bm{m}\notin\mathcal{C}_r$. 
This concludes the proof of the lemma.
\end{proof}

We next give another family of conditions,
which implies the conditions in Lemma \ref{Lem:general},
and which is more suitable for the measure computations which
we are going to carry out.
\begin{prop}\label{prop:regulargen}
Let $c_0\in(0,1)$ be as above,
let $r\in (0,c_0/d)$ and let $\Lambda=p_{\bm{b}_1,\dots,\bm{b}_{d-1}}u_{\bm{x}}\Z^d\in X_d$ with 
$\bm{b}_j=(b_{1j},\ldots,b_{dj})^t\in\R^d$ ($j=1,\ldots,d-1$)
and $\bm{x}\in (0,c_0/d)^{d-1}$. 
Assume that 
\begin{align}\label{equ:con1}
b_{ij}, -b_{ji}\in (-c_0, 0),\ \forall\ 1\leq j< i\leq d-1; \ b_{d\ell}\in (-c_0, 0), b_{\ell \ell}\in (1-\tfrac{r}{2d},1), \  \forall\ 1\leq \ell\leq d-1,
\end{align}
\begin{align}\label{equ:con2}
b_{ij}<db_{i,j-1}\ (\Leftrightarrow |b_{ij}|>d|b_{i,j-1}|),\qquad \forall\ 2\leq j<i\leq d,
\end{align}
\begin{align}\label{equ:con3}
|b_{ij}b_{ji}|<\frac{r}{d!},\quad \forall 1\leq j<i\leq d-1,\qquad \textrm{and}\qquad \sum_{j=1}^{d-1}|b_{dj}|x_j<\frac{r}{2},
\end{align}
and
\begin{align}\label{equ:con4}
b_{kj}>b_{ij}\ (\Leftrightarrow |b_{kj}|<|b_{ij}|),\qquad \forall 1\leq j<k<i\leq d.
\end{align}
Then $\Lambda\in K_r$.
\end{prop}
\begin{proof}
In view of Lemma \ref{Lem:general}, it suffices to show that 
the conditions 
\eqref{Fs3dimpfLEM1ass1gen}, \eqref{Fs3dimpfLEM1ass2gen}, \eqref{Fs3dimpf5gen} and \eqref{equ:newresgen}
are fulfilled.
Among these, \eqref{Fs3dimpfLEM1ass1gen} is an immediate consequence of \eqref{equ:con1}.
Furthermore, \eqref{Fs3dimpfLEM1ass2gen} follows from \eqref{equ:con1} and \eqref{equ:con2}
by the following computation,
valid for any $2\leq j<i\leq d$:
\begin{align*}
b_{ij}<db_{i,j-1}<(j-1)b_{i,j-1}\leq\sum_{k=1}^{j-1}b_{ik}
\quad(<0),
\end{align*}
where the last relation is an equality when $j=2$,
while for $j\geq3$ it is a strict inequality which 
holds since \eqref{equ:con2} forces
$b_{i,j-1}<b_{i,j-2}<\cdots<b_{i,1}$.
Also \eqref{equ:newresgen} is easily proved:
Let $1\leq i\leq d-1$.
Then for every $1\leq j\leq d-1$ we have $x_j>0$
and $b_{ij}>b_{dj}$
(the latter holds by \eqref{equ:con4} if $j<i$ and by \eqref{equ:con1} if $j\geq i$).
Hence
$$b_{ii}+\sum_{j=1}^{d-1}b_{ij}x_j>b_{ii}+\sum_{j=1}^{d-1}b_{dj}x_j>1-\frac{r}{2d}-\frac{r}{2}>1-r,$$ 
where for the second last inequality we used
\eqref{equ:con1} and the second part of 
\eqref{equ:con3}.


It remains to prove 
\eqref{Fs3dimpf5gen}. 
We first note that the ordering assumptions in \eqref{equ:con2} and \eqref{equ:con4} imply
\begin{align}\label{equ:orderingcon}
|b_{i',j'}|\leq |b_{ij}|\qquad \textrm{whenever 
$1\leq j'<i'\leq d$, $1\leq j<i\leq d$
and 
$i'\leq i$ and $j'\leq j$.}
\end{align}
Now let $\tilde{p}=(b_{ij})_{1\leq i,j\leq d-1}$ be as in Lemma \ref{Lem:general}. 
Define
$$\varphi_{\tilde{p}}:=\sum_{\substack{\sigma\in S_{d-1}\\ \sigma\neq \id}}\prod_{1\leq i\leq d-1}\left|b_{\sigma(i)i}\right|,$$
where $S_{d-1}$ is the symmetric group over the finite set $\{1,2,\ldots, d-1\}$ 
and $\id\in S_{d-1}$ is its identity element.
Then
\begin{align}\label{varphitildeppf2}
\prod_{i=1}^{d-1}b_{ii}-\varphi_{\tilde{p}}\leq \det \tilde{p}\leq \prod_{i=1}^{d-1}b_{ii}+\varphi_{\tilde{p}}.
\end{align}

If $d=2$ then $S_{d-1}=\{\id\}$ and $\varphi_{\tilde{p}}=0$.
Now assume $d\geq3$, and consider an arbitrary permutation
$\sigma\in S_{d-1}\smallsetminus\{\id\}$.
Let $(i_1\ldots i_{\ell})$ be a cycle of $\sigma$ of length $\ell\geq 2$, meaning that 
$\sigma(i_j)=i_{j+1}$ for all $1\leq j\leq \ell-1$ and $\sigma(i_{\ell})=i_1$. 
Without loss of generality, we may assume $i_{\ell}=\max\{i_j: 1\leq j\leq \ell\}$. 
Let $i_{j_0}:=\min\left\{i_j: 1\leq j\leq \ell\right\}$. 
Then \eqref{equ:orderingcon} applies for the pairs 
$(i',j')=(\sigma(i_{j_0}), i_{j_0})$ and $(i,j)=(i_{\ell}, i_1)$,
so that $\bigl|b_{\sigma(i_{j_0}),i_{j_0}}\bigr|\leq \bigl|b_{i_{\ell},i_1}\bigr|$.
We also have $|b_{ij}|<1$ for all $1\leq i,j\leq d-1$, 
by \eqref{equ:con1}.
Hence
$$
\prod_{1\leq i\leq d-1}\left|b_{\sigma(i)i}\right|\leq\left|b_{\sigma(i_{j_0}),i_{j_0}}b_{i_1,i_{\ell}}\right|\leq \left|b_{i_{\ell},i_1}b_{i_1, i_{\ell}} \right|<\frac{r}{d!},
$$
where the last inequality holds by \eqref{equ:con3}.
The above holds for every $\sigma\in S_{d-1}\smallsetminus\{\id\}$; hence
\begin{align}\label{varphitildeppf1}
\varphi_{\tilde{p}}<(d-1)!\frac{r}{d!}=\frac{r}{d}.
\end{align}
Note that \eqref{varphitildeppf1} also holds when $d=2$, trivially.

Using \eqref{varphitildeppf2}, \eqref{varphitildeppf1},
and the fact that $1-\frac r{2d}<b_{ii}<1$ for all $1\leq i\leq d-1$ (cf. \eqref{equ:con1}), we get
\begin{align}\label{equ:assump}
\det \tilde{p} <\prod_{i=1}^{d-1}b_{ii}+\frac{r}{d}<1+\frac{r}{d}<(1-r)^{-1},
\end{align}
and
\begin{align*}
\det \tilde{p} > \prod_{i=1}^{d-1}b_{ii}-\frac{r}{d}>\left(1-\frac{r}{2d}\right)^{d-1}-\frac{r}{d}>1-\frac{r}{2}-\frac{r}{d}\geq 1-r.
\end{align*}
Hence we have proved the first condition in \eqref{Fs3dimpf5gen}. 
For the second condition in \eqref{Fs3dimpf5gen}, in view of the second condition in \eqref{equ:con3} it suffices to show $\tfrac{r}{2}<(\det \tilde{p})^{-1}-(1-r)$, or equivalently, {that $\det \tilde{p}$ is smaller than} $(1-\tfrac{r}{2})^{-1}$. But this is true since by \eqref{equ:assump}, $\det \tilde{p} <1+\tfrac{r}{d}< (1-\tfrac{r}{2})^{-1}$. This finishes the proof.
\end{proof}


We can now give the
\begin{proof}[Proof of the lower bound in Theorem \ref{thm:measestmainaux}]
Keep $0<r<c_0/d$, and define
\begin{align}\label{def:urlowerbd}
\underline{K}_r:=\left\{p_{\bm{b}_1,\dots,\bm{b}_{d-1}}u_{\bm{x}}\Z^d\in X_d\col \begin{array}{ll}
(\bm{b}_1,\ldots, \bm{b}_{d-1},\bm{x})\in  (\R^d)^{d-1}\times (0,c_0/d)^{d-1}\\
\textrm{\ \ satisfies}\ \eqref{equ:con1}, \eqref{equ:con2}, \eqref{equ:con3},\eqref{equ:con4}
\end{array}\right\}. 
\end{align} 
By Proposition \ref{prop:regulargen} we have $\underline{K}_r\subset K_r$. 
It remains to bound $\mu_d\left(\underline{K}_r\right)$
from below.
Set
\begin{align*}
{\underline\cK_r}:=\left\{p_{\bm{b}_1,\dots,\bm{b}_{d-1}}u_{\bm{x}}\in G\col \begin{array}{ll}
(\bm{b}_1,\ldots, \bm{b}_{d-1},\bm{x})\in  (\R^d)^{d-1}\times (0,c_0/d)^{d-1}\\
\textrm{\ \ satisfies}\ \eqref{equ:con1}, \eqref{equ:con2}, \eqref{equ:con3},\eqref{equ:con4}
\end{array}\right\},
\end{align*}
so that $\pi({\underline\cK_r})=\underline{K}_r$. 
Here $\pi: G\to X_d$ is the natural projection as before. 
By immediate inspection of the proofs of Proposition \ref{prop:regulargen}
and Lemma \ref{Lem:general},
we have ${\underline\cK_r}\subset G_{c_0}$,
and by our choice of $c_0$, $\pi|_{G_{c_0}}$ is injective
(see Section \ref{sec:smpara}).
Hence $\mu_d\left(\underline{K}_r\right)=\nu\left({\underline\cK_r}\right)$, 
and by \eqref{equ:haar} we have 
\begin{align*}
\nu\left({\underline\cK_r}\right)&\asymp_d \prod_{1\leq k\leq d-1}\left(\int_{1-\frac{r}{2d}}^1\d b_{kk}\right)
\int_{\cR}\delta\bigl(\,(b_{ij})_{1\leq j<i\leq d}\,\bigr) \prod_{1\leq j<i\leq d}\d b_{ij},
\end{align*}
where
$$
\cR:=\left\{(b_{ij})_{1\leq j<i\leq d}\in (-c_0, 0)^{d(d-1)/2} \col \textrm{$(b_{ij})$ satisfies \eqref{equ:con2} and \eqref{equ:con4}}\right\},
$$
and
\begin{align*}
\delta\bigl(\,(b_{ij})_{1\leq j<i\leq d}\,\bigr)
:&=\prod_{1\leq j<i\leq d-1}\int_0^{\min\left\{c_0,\frac{r}{d!|b_{ij}|}\right\}}\d b_{ji}\times \int_{\left\{\bm{x}\in \left(0,\tfrac{c_0}{d}\right)^{d-1}\col \sum_{j=1}^{d-1}|b_{dj}|x_j<\frac{r}{2}\right\}}\prod_{1\leq j\leq d-1}\d x_j\\
&\asymp_{d,c_0}\prod_{1\leq j<i\leq d-1}\min\left\{1,\frac{r}{|b_{ij}|}\right\}\times \prod_{1\leq j\leq d-1}\min\left\{1,\frac{r}{|b_{dj}|}\right\}\\
&=\prod_{1\leq j<i\leq d}\min\left\{1,\frac{r}{|b_{ij}|}\right\}.
\end{align*}
Hence 
\begin{align}\label{muduKrlb}
\mu_d\left(\underline{K}_r\right)=\nu\left({\underline\cK_r}\right)\asymp_{d, c_0}r^{d-1}\int_{\cR}\left(\prod_{1\leq j<i\leq d} \min\left\{1, \frac{r}{|b_{ij}|}\right\}\d b_{ij}\right).
\end{align}

Now for each $1\leq j<i\leq d$, we make a change of variable, $b_{ij}=-d^{j-1}z_{ij}$,
to simplify 
the ordering condition \eqref{equ:con2}. 
Then all the $z_{ij}$'s are positive, and the conditions \eqref{equ:con2} and \eqref{equ:con4} become
\begin{align}\label{equ:ordering}
z_{i'j'}<z_{ij} \quad \textrm{whenever 
$1\leq j'<i'\leq d$, $\:1\leq j<i\leq d$,
$\:i'\leq i$, $\:j'\leq j$ and $(i',j')\neq (i, j)$}.
%
\end{align}
Moreover, for any $j<i$, {the condition} $b_{ij}\in (-c_0,0)$ corresponds to $z_{ij}\in (0,c_0/d^{j-1})$,
and we note that each of these intervals contains the fixed interval
$(0,c_0/d^{d-1})$.
In fact, let us restrict each $z_{ij}$ to the even smaller interval
$(r,c_0/d^{d-1})$,
and assume $r<c_0/d^{d-1}$ so that this interval is non-empty;
then $r/|b_{ij}|=r d^{1-j}/z_{ij}<d^{1-j}\leq 1$,
so that $\min\left\{1,r/|b_{ij}|\right\}=rd^{1-j}/z_{ij}\asymp_d r/z_{ij}$.
Note also that $|\d b_{ij}|=d^{j-1}\,\d z_{ij}\asymp_d \d z_{ij}$;
hence we get from \eqref{muduKrlb}:
\begin{align*}
\mu_d(\underline{K}_r)&\gg_{d,c_0} 
r^{d-1+\frac{d(d-1)}{2}}\int_{\cR'}\prod_{1\leq j<i\leq d}\frac{\d z_{ij}}{z_{ij}}
=r^{\frac{(d-1)(d+2)}{2}}\int_{\cR'}\prod_{1\leq j<i\leq d}\frac{\d z_{ij}}{z_{ij}},
\end{align*}
where 
\begin{align*}
\cR':=\left\{(z_{ij})_{1\leq j<i\leq d}\in (r,c_0/d^{d-1})^{d(d-1)/2}\col (z_{ij})\ \textrm{satisfies \eqref{equ:ordering}}\right\}.
\end{align*}
The last integrand is invariant under every permutation of the variables $(z_{ij})_{1\leq j<i\leq d}$,
and the integration regions $\sigma(\cR')$ with $\sigma$ running through all these permutations cover (modulo a null set) the cube $(r, c_0/d^{d-1})^{d(d-1)/2}$;
hence we obtain
\begin{align}\label{uKrlb}
\mu_d\left(\underline{K}_r\right)&\gg_{d, c_0} r^{\frac{(d-1)(d+2)}{2}}\prod_{1\leq j<i\leq d}\left(\int_r^{c_0/d^{d-1}}\frac{\d z_{ij}}{z_{ij}}\right)
\asymp_{d, c_0} r^{\frac{(d-1)(d+2)}{2}}\log^{\frac{d(d-1)}{2}}\Bigl(\frac{1}{r}\Bigr),
\end{align}
where the last estimate is valid e.g.\ for all $0<r<c_0/(2d^{d-1})$.
Recalling 
that $c_0$ only depends on $d$,
we see that 
the lower bound in Theorem \ref{thm:measestmainaux} is now proved.
\end{proof}

\section{Proof of the upper bound}\label{upperboundsec}
We keep the notation introduced in Section \ref{sec:premest}. 
In this section we prove the upper bound in Theorem \ref{thm:measestmainaux}.
The main step in the proof is to show that for $r$ small,
$K_r$ is contained in a certain set of more explicit nature; see Proposition \ref{prop:increlimub} below.
To prepare for the statement of this result, 
we start by introducing the following set,
for any $r\in(0,1)$ and $C>0$:
\begin{align*}
{\overline{\cK}_{r,C}}:=\left\{g=(g_{ij})\in G\col 
\begin{array}{ll}
1-r\leq g_{ii}\leq 1+Cr,\ \forall\ 1\leq i\leq d\\
|g_{ij}|<1,\; |g_{ij}g_{ji}|\leq Cr,\ \forall\ 1\leq i\neq j\leq d\\
|\det \tilde{g}-1|<\tfrac12
\end{array}\right\}.
\end{align*}
Here $\tilde{g}$ denotes the top left $(d-1)\times(d-1)$ block of $g$ as before. In view of the Haar measure description \eqref{equ:haar2}, it is not difficult to compute the measure of ${\overline{\cK}_{r,C}}$:
\begin{lem}\label{lem:measuk}
For any $r\in (0,\frac12)$ and $C>0$ we have
\begin{align*}
\nu\bigl({\overline{\cK}_{r,C}}\bigr)\ll_{d, C} r^{\varkappa_d+1}\log^{\lambda_d}\Bigl(\frac{1}{r}\Bigr).
\end{align*}
\end{lem}
\begin{proof}
Let us define
\begin{align*}
{\overline{\cK}'_{r,C}}:=\left\{g=(g_{ij})\in G\col 
\begin{array}{ll}
1-r\leq g_{ii}\leq 1+Cr,\ \forall\ 1\leq i\leq d-1\\
|g_{ij}|<1,\; |g_{ij}g_{ji}|\leq Cr,\ \forall\ 1\leq i\neq j\leq d\\
|\det \tilde{g}-1|<\tfrac12
\end{array}\right\}
\end{align*}
by disregarding the restriction on $g_{dd}$. Then clearly ${\overline{\cK}_{r,C}}\subset {\overline{\cK}'_{r,C}}$. 
Moreover, in view of the Haar measure description \eqref{equ:haar2} we have (noting also that $\varkappa_d+1=\lambda_d+d-1$)
\begin{align*}
\nu\left({\overline{\cK}'_{r,C}}\right)&\ll_d \prod_{1\leq i\leq d-1}\int_{1-r}^{1+Cr}\d g_{ii}\prod_{1\leq j<i\leq d}\int_{-1}^1\int_{-1}^1{\chi_{\left\{(g_{ij},g_{ji})\col |g_{ij}g_{ji}|<Cr\right\}}}\,\d g_{ij}\d g_{ji}\\
&\asymp_C r^{d-1}\left(r\log\Bigl(\frac{1}{r}\Bigr)\right)^{\lambda_d}
=r^{\varkappa_d+1}\log^{\lambda_d}\Bigl(\frac{1}{r}\Bigr).
\end{align*}
Thus
\begin{align*}
\nu\left({\overline{\cK}_{r,C}}\right)\leq \nu\left({\overline{\cK}'_{r,C}}\right)\ll_{d,C} r^{\varkappa_d+1}\log^{\lambda_d}\Bigl(\frac{1}{r}\Bigr),
\end{align*}
finishing the proof.
\end{proof}
Recall that 
$\pi: G\to X_d$ is the natural projection from $G$ to $X_d$. 
\begin{prop}\label{prop:increlimub}
There exist $r_0>0$ and $C>0$ (depending only on $d$) such that for all $r\in (0, r_0)$
\begin{align}\label{equ:ubrelation}
K_r\subset\bigcup_{w\in W}\pi\bigl(w\,{\overline{\cK}_{r,C}}\,w^{-1}\bigr).
\end{align}
\end{prop}

Let us first give a quick
\begin{proof}[Proof of the upper bound in Theorem \ref{thm:measestmainaux} assuming Proposition \ref{prop:increlimub}]
The Haar measure $\nu$ on $G$ is preserved by conjugation by any element $w\in W$ 
(even though $w$ may be outside $G$).
Hence it follows from \eqref{equ:ubrelation} that, for any $r\in(0,r_0)$,
\begin{align*}
\mu_d(K_r)&\leq \sum_{w\in W}\mu_d\bigl(\pi(w\,{\overline{\cK}_{r,C}}\,w^{-1})\bigr)
\leq \sum_{w\in W}\nu(w\,{\overline{\cK}_{r,C}}\,w^{-1})
=d!\,\nu({\overline{\cK}_{r,C}}).
\end{align*}
Using this inequality, 
the upper bound in Theorem \ref{thm:measestmainaux} now follows from Lemma \ref{lem:measuk}.
\end{proof}

\begin{remark}\label{bulkinZdnbhdREM}
Proposition \ref{prop:increlimub},
in combination with the lower bound
in Theorem \ref{thm:measestmainaux},
also implies that
as $r\to0^+$, the mass of $K_r$ with respect to $\mu_d$
becomes concentrated near the lattice $\Z^d$.
In precise terms, if 
$\scrO$ is any fixed neighborhood of $\Z^d$ in $X_d$,
then 
\begin{align}\label{bulkinZdnbhdREMres}
\mu_d\left(K_r\smallsetminus\scrO\right)\ll_d r^{\varkappa_d+1}\log^{\lambda_d-1}\Bigl(\frac{1}{r}\Bigr),
\quad\text{and thus }\quad
\frac{\mu_d(K_r\cap\scrO)}{\mu_d(K_r)}\to1
\qquad \textrm{as $r\to 0^+$}.
\end{align}
Indeed, we can fix $\ve>0$ so that $\scrO$ contains the set
$\pi(G_{\ve})$ with $G_{\ve}\subset G$ the norm ball defined as in \eqref{equ:normball};
then by arguing along the same lines as above, 
the first relation in \eqref{bulkinZdnbhdREMres}
will follow from the following bound:
\begin{align*}
\nu\left({\overline{\cK}_{r,C}}\smallsetminus G_{\ve}\right)
\ll_{d, C} r^{\varkappa_d+1}\log^{\lambda_d-1}\Bigl(\frac{1}{r}\Bigr),
\end{align*}
for $C>0$ and $r$ small.
However, for $r<\ve/C$, $g=(g_{ij})\in{\overline{\cK}_{r,C}}$ forces $|g_{ii}-1|<\ve$ for all $i$,
and so the set
${\overline{\cK}_{r,C}}\smallsetminus G_{\ve}$ is contained in the union
$\cup_{i'\neq j'}\bigl\{g=(g_{ij})\in{\overline{\cK}_{r,C}}\col |g_{i'j'}|\geq\ve\bigr\}$.
Therefore, it suffices to prove that for any given
$1\leq i'\neq j'\leq d$ we have
\begin{align*}
\nu\bigl(\bigl\{g=(g_{ij})\in{\overline{\cK}_{r,C}}\col |g_{i'j'}|\geq\ve\bigr\}\bigr)\ll_{d,C}
r^{\varkappa_d+1}\log^{\lambda_d-1}\Bigl(\frac{1}{r}\Bigr).
\end{align*}
This is shown by following the proof of Lemma \ref{lem:measuk}
and using $$\int_{\ve\leq|x|\leq1}\int_{-1}^1{\chi_{\{(x,y)\col |xy|<Cr\}}}\,\d y\,\d x\ll_{\ve,C} r.$$ 
Finally, the second relation in \eqref{bulkinZdnbhdREMres} 
follows from the first relation 
combined with the lower bound 
in Theorem \ref{thm:measestmainaux}.
\end{remark}

The remainder of this section is devoted to the proof of Proposition \ref{prop:increlimub}.

\subsection{Bounds on diagonal entries}
Recall that $U_0$ is the fixed fundamental domain for ${U/(\G\cap U)}$ given in \eqref{equ:ufudo}. The next lemma shows that if a lattice $\Lambda\in K_r$ has a representative sufficiently close to some element in $U_0$, then the diagonal entries of such a representative satisfy the desired bounds. 
\begin{lem}\label{firstkeyboundLEM}
Let $g=(g_{ij})\in G$ and $r\in(0,\frac18)$,
and assume that $g\Z^d\cap\scrC_r=\{\bn\}$ 
and $\|g-u\|<\frac18$ for some $u\in U_0$.
Then
\begin{align}\label{firstkeyboundLEMres}
0\leq g_{ii}-(1-r)\ll_d r,\qquad\forall\: 1\leq i\leq d.
\end{align}
\end{lem}
\begin{proof}
Note that it follows from 
$\|g-u\|<\frac18$ and $u\in U_0$
that $|g_{ij}|<\frac58$ for all $1\leq i<j\leq d$,
$|g_{ii}-1|<\frac18$ for all $1\leq i\leq d$,
and $|g_{ij}|<\frac18$ for all $1\leq j<i\leq d$.
For each $i$ we have $g\vece_i\notin\scrC_r$, 
since $g\Z^d\cap\scrC_r=\{\bn\}$.
Combining this with the fact that 
$|g_{ji}|<\frac58<1-r$ for all $j\neq i$,
we conclude that $|g_{ii}|\geq 1-r$.
Since also $|g_{ii}-1|<\frac18$,
we must in fact have $g_{ii}\geq 1-r$,
i.e.\ we have proved the left inequality in \eqref{firstkeyboundLEMres}.

Next, let $i\in\{1,\ldots,d\}$ be given,
and let $\scrU$ be the open box 
\begin{align*}
\scrU=\scrI_1\times \cdots\times \scrI_d,
\quad\text{with }\:\left\lbrace\hspace{-4pt}\begin{array}{ll} 
\scrI_i=\bigl(\frac12(1-r),g_{ii}-\frac12(1-r)\bigr),
\\[5pt]
\scrI_j=\bigl(g_{ji}^+-\frac38,g_{ji}^-+\frac38\bigr)\  (j\neq i),
\end{array}\right.
\end{align*}
where $g_{ji}^+:=\max\{g_{ji},0\}$ and $g_{ji}^-:=\min\{g_{ji},0\}$.
Note that each interval $\scrI_j$ ($j\neq i$) has length $|\scrI_j|>\frac18$, since $|g_{ji}|<\frac58$;
furthermore $|\scrI_i|=g_{ii}-(1-r)\geq0$
(thus $\scrI_i$ and $\scrU$ are empty if $g_{ii}=1-r$, but otherwise non-empty).

We claim that $\scrU$ is disjoint from $\fC_{g\Z^d,r}:=g\Z^d+\tfrac12\scrC_r$.
Indeed, assume the opposite; then there is some $\vecv\in g\Z^d$ 
such that $\scrU\cap(\vecv+\frac12\scrC_r)\neq\varnothing$.
We must have $\vecv\neq\bn$ and $\vecv\neq g\vece_i$,
since $\scrU$ is, by construction, disjoint {both} from $\frac12\scrC_r$ and from $g\vece_i+\frac12\scrC_r$.
Pick a point 
$\vecx\in \scrU\cap(\vecv+\frac12\scrC_r)$.
It follows from 
$\vecx\in \vecv+\frac12\scrC_r$
and $\frac12(1-r)>\frac14$
that at least one of the points $\vecx-\frac14\vece_i$ or $\vecx+\frac14\vece_i$
\textit{also} lies in 
$\vecv+\frac12\scrC_r$.
But we have 
\begin{align*}
\vecx-\tfrac14\vece_i\in\tfrac12\scrC_r,
\end{align*}
since $x_j\in \scrI_j\subset\big(-\frac12(1-r),\frac12(1-r)\big)$ for all $j\neq i$
and $x_i-\frac14<g_{ii}-\frac12(1-r)-\frac14 
<\frac12(1-r)$.
We also have
\begin{align*}
\vecx+\tfrac14\vece_i\in g\vece_i+\tfrac12\scrC_r,
\end{align*}
since $x_j\in \scrI_j\subset g_{ji}+\big(-\frac12(1-r),\frac12(1-r)\big)$ for all $j\neq i$
and
$x_i+\frac14>\frac12(1-r)+\frac14>g_{ii}-\frac12(1-r)$. 
Hence we have arrived at a contradiction against the fact that 
$\vecv+\frac12\scrC_r$ is disjoint from both
$\frac12\scrC_r$ and $g\vece_i+\frac12\scrC_r$.
This completes the proof of the fact that 
$\scrU$ is disjoint from $\fC_{g\Z^d,r}$.

Note also that $\scrU$ is contained in a translate of the cube $(0,\frac34)^d$,
since each interval $\scrI_j$ has length at most $\frac34$.
Hence Lemma \ref{fCkeyLEM} applies,
and yields that $\vol(\scrU)<dr$.
But we have noted that $|\scrI_j|>\frac18$
for each $j\neq i$, and $|\scrI_i|=g_{ii}-(1-r)$;
hence $\vol(\scrU)\geq 8^{1-d}(g_{ii}-(1-r))$.
Combining these facts, we obtain the right bound in \eqref{firstkeyboundLEMres} 
with the implied constant $8^{d-1}d$.
\end{proof}

\subsection{A technical choice of lattice representatives}

For any $\ve>0$ let us write
\begin{align*}
U_{\ve}=\bigl\{(u_{ij})\in U\col -\tfrac12+4\ve<u_{ij}\leq\tfrac12+\ve\text{ for all }1\leq i<j\leq d\bigr\}.
\end{align*}

\begin{lem}\label{fdwiggleLEM}
There exist constants $0<a<1$ and $A>1$, which only depend on $d$,
such that the following holds:
given any $u\in U$ and $\ve\in(0,a)$, there
exist $\gamma\in \Gamma\cap U$ and  
$B\in[1,A]$ such that $u\gamma\in U_{B\ve}$.
\end{lem}
\begin{proof}
We will show that the statement of this lemma holds with 
$A:=4^{2^d}$ and $a:=2^{-(3+2^{d+1})}$.
Let us set
\begin{align*}
U[t]:=\bigl\{(u_{ij})\in U\col -\tfrac12+t<u_{ij}\leq\tfrac12+t\text{ for all }1\leq i<j\leq d\bigr\}
\qquad(t\in\R).
\end{align*}
Note that $U_0=U[0]$, 
and for each $t$, $U[t]$ is a fundamental domain for $U/(\Gamma\cap U)$.

For the given $u\in U$ and $\ve\in(0,a)$, 
and for each $
k\in\Z_{\geq0}$,
we let $u^{(k)}$ be the unique element in
$u\Gamma\cap U[4^k\ve]$,
and set
\begin{align*}
P^{(k)}=\sum_{j=2}^dP^{(k)}_j,
\quad\text{where}\quad
P^{(k)}_j:=\sum_{i=1}^{j-1}2^{i-1}\cdot {\chi_{(0,\infty)}(u_{ij}^{(k)})}.
\end{align*}
We now claim:
\begin{align}\label{fdwiggleLEMpf1}
\forall\, k\in\Z_{\geq0}:\quad
k<\log_4(\tfrac1{8\ve})\text{ and }u^{(k)}\notin U_{4^k\ve}
\quad\Longrightarrow\quad
P^{(k)}<P^{(k+1)}.
\end{align}

To prove \eqref{fdwiggleLEMpf1}, 
assume $0\leq k<\log_4(\frac1{8\ve})$ and $u^{(k)}\notin U_{4^k\ve}$.
Then $u^{(k+1)}\neq u^{(k)}$,
since otherwise $u^{(k)}$
would lie in the intersection $U[4^k\ve]\cap U[4^{k+1}\ve]=U_{4^k\ve}$.
Let us denote by $\vecu_j^{(k)}$ the $j$th column vector of $u^{(k)}$.
It follows that $\vecu_j^{(k+1)}\neq\vecu_j^{(k)}$ for at least one $j\in\{2,\ldots,d\}$,
and for each such index $j$ we can argue as follows:
Since $u^{(k)}, u^{(k+1)}\in u\G\cap U=u(\G\cap U)$, we have
$\bm{u}_j^{(k+1)}=\bm{u}_j^{(k)}+\sum_{1\leq i<j}m_i\bm{u}_i^{(k)}$ for some $m_1,\ldots, m_{j-1}\in \Z$. 
Because of $\vecu_j^{(k+1)}\neq\vecu_j^{(k)}$,
there exists $i\in\{1,\ldots,j\}$ such that 
$m_i\neq0$;
let us fix $i$ to be the \textit{largest} such index.
Thus $u_{ij}^{(k+1)}=u_{ij}^{(k)}+m_i$; in particular $\left|u_{ij}^{(k+1)}-u_{ij}^{(k)}\right|\geq 1$.  
On the other hand, note that $u_{ij}^{(k)}\in \left(-\frac12+4^k\ve, \frac12+4^k\ve\right]$, $u_{ij}^{(k+1)}\in \left(-\frac12+4^{k+1}\ve, \frac12+4^{k+1}\ve\right]$ and $\left(\frac12+4^{k+1}\ve)-(-\frac12+4^k\ve\right)=1+3\cdot 4^k\ve<2$ (since $k< \log_4(\frac{1}{8\ve})$).
Hence we must have $m_i=1$, implying that $$\vecu_j^{(k+1)}\in\vecu_j^{(k)}+\vecu_i^{(k)}+\Z\vecu_{i-1}^{(k)}+\cdots+\Z\vecu_1^{(k)}.$$
Thus 
$u_{i'j}^{(k+1)}=u_{i'j}^{(k)}$ for all $i'>i$,
and $u_{ij}^{(k+1)}=u_{ij}^{(k)}+1>0$
while $u_{ij}^{(k)}\leq-\frac12+4^{k+1}\ve<0$
(again since $k<\log_4(\frac1{8\ve})$). 
It follows that $P_j^{(k+1)}-P_j^{(k)}\geq 2^{i-1}-\sum_{i'<i}2^{i'-1}=1$.
On the other hand we clearly have $P_j^{(k+1)}=P_j^{(k)}$
for each $j\in\{2,\ldots,d\}$ such that 
$\vecu_j^{(k+1)}=\vecu_j^{(k)}$.
Hence $\sum_{j=2}^dP^{(k+1)}_j>\sum_{j=2}^dP^{(k)}_j$,
i.e.\ $P^{(k+1)}>P^{(k)}$. This finishes the proof of \eqref{fdwiggleLEMpf1}.

Next note that by definition,
for each $k$,
$P^{(k)}$ is a non-negative integer
satisfying $$P^{(k)}\leq\sum_{j=2}^d\sum_{i=0}^{j-2}2^i<2^d.$$
This implies that we cannot have $P^{(k)}<P^{(k+1)}$ for all
$k\in\{0,1,\ldots,2^d-1\}$.
But our assumption on $\ve$ implies that 
$\log_4(\frac1{8\ve})>2^d$ 
(recall that $0<\ve<a=2^{-(3+2^{d+1})}$);
hence it now follows from \eqref{fdwiggleLEMpf1} that 
$u^{(k)}\in U_{4^k\ve}$ for at least one $k\in\{0,1,\ldots,2^d-1\}$.
Furthermore, for such $k$ we have $4^k<4^{2^d}=A$ 
and $u^{(k)}\in u(\Gamma\cap U)$ by construction.
Hence the lemma is proved.
\end{proof}

Using Lemma \ref{fdwiggleLEM}, Corollary \ref{MinkowskiHajosThm} and a compactness argument we have the following technical lemma, which gives us a good choice of lattice representatives for lattices in $K_{r_0}$ for some small $r_0>0$.
\begin{lem}\label{MinkowskiHajosCor2}
There exist constants $0<a<1$ and $A>1$, which only depend on $d$,
such that the following holds:
for any ${\ve_0}\in(0,a)$ there exists $r_0>0$ such that
for every $\Lambda\in X_d$ satisfying 
$\Lambda\cap \scrC_{r_0}=\{\bn\}$,
there exist $g\in G$,
$w\in W$, $B\in[1,A]$ and $u\in U_{B{\ve_0}}$ such that
$\Lambda=g\Z^d$ and $\|g-wuw^{-1}\|<{\ve_0}$.
\end{lem}
\begin{proof}
Let $a$ and $A$ be as in Lemma \ref{fdwiggleLEM};
we will prove that the statement of the lemma holds with these $a,A$.
The proof is by contradiction;
thus we assume that the statement of the lemma is false,
i.e.\ we assume that 
there exist some ${\ve_0}\in(0,a)$ and a sequence
$r_1>r_2>\cdots$ in $(0,1)$ with $r_j\to0$,
and a corresponding sequence $\Lambda_1,\Lambda_2,\ldots$ in $X_d$,
such that $\Lambda_j\cap \scrC_{r_j}=\{\bn\}$ for each $j$,
and furthermore, for each $j$ we have that there do
\textit{not} exist any 
$g\in G$,
$w\in W$, $B\in[1,A]$ and $u\in U_{B{\ve_0}}$ satisfying
$\Lambda_j=g\Z^d$ and $\|g-wuw^{-1}\|<{\ve_0}$.
Now for every $j$ we have
$\scrC_{r_1}\subset \scrC_{r_j}$, and thus $\Lambda_j\cap \scrC_{r_1}=\{\bn\}$.
Hence by Mahler's Compactness Theorem,
after passing to a subsequence we may assume that 
$\Lambda_j$ tends to a limit point in $X_d$.
Let us call this limit point $\Lambda_0$;
thus $\Lambda_j\to \Lambda_0$ in $X_d$ as $j\to\infty$.
Let us also fix a representative $g_0\in G$ such that $\Lambda_0=g_0\Z^d$.

Recall that the standard topology on $X_d=G/\Gamma$
is given by the metric
\begin{align*}
\dist(g\Gamma ,g'\Gamma ):=\inf\{d(g\gamma,g')\col\gamma\in\Gamma\}
\qquad (g,g'\in G),
\end{align*}
where $d(\cdot,\cdot)$ is any fixed right $G$-invariant Riemannian metric on $G$.
Hence the fact that $\Lambda_j$ {converges to} $\Lambda_0=g_0\Z^d$
implies that there exist $g_1,g_2,\ldots\in G$ such that
$\Lambda_j=g_j\Z^d$ for each $j$ and $d(g_j,g_0)\to0$ as $j\to\infty$.

Using $\Lambda_j\cap \scrC_{r_j}=\{\bn\}$ for each $j$, and $r_j\to0$,
we claim that 
\begin{align}\label{MinkowskiHajosCorpf2}
\Lambda_0\cap(-1,1)^d=\{\bn\}.
\end{align}
Indeed, assume the opposite;
this means that there exists some $\vecm\in\Z^d\smallsetminus\{\bm{0}\}$
such that $g_0\vecm$ belongs to $(-1,1)^d$,
i.e.\ $\|g_0\vecm\|<1$.
Set $r:=\frac12(1-\|g_0\vecm\|)$; then $g_0\vecm\in\scrC_r$.
We have $g_j\vecm\to g_0\vecm$ as $j\to\infty$, 
since $d(g_j,g_0)\to0$;
hence for all sufficiently large $j$ we have $g_j\vecm\in \scrC_r$ (since $\scrC_r$ is open).
Also for all sufficiently large $j$ we have $r_j<r$.
Hence there exists some $j$ for which $r_j<r$ and 
$g_j\vecm\in \scrC_r\subset \scrC_{r_j}$.
This contradicts the fact that $\Lambda_j\cap \scrC_{r_j}=\{\bn\}$ for all $j$.
Hence \eqref{MinkowskiHajosCorpf2} is proved.

It follows from \eqref{MinkowskiHajosCorpf2} and Corollary \ref{MinkowskiHajosThm}
that 
{$\Lambda_0=wu'w^{-1}\Z^d$ for some
$w\in W$ and $u'\in U_0$}. 
Next, by Lemma \ref{fdwiggleLEM} 
(and since ${\ve_0}<a$),
there exist 
$\gamma\in\Gamma\cap U$ and $B\in[1,A]$
such that $u:=u'\gamma\in U_{B{\ve_0}}$.
Using $w^{-1}\Z^d=\Z^d$ and $\gamma^{-1}\Z^d=\Z^d$,
we then have $\Lambda_0=wuw^{-1}\Z^d$.
Hence $wuw^{-1}=g_0\gamma_0$ for some $\gamma_0\in\Gamma$.
Now $d(g_j\gamma_0 ,wuw^{-1})
=d(g_j\gamma_0, g_0\gamma_0)
=d(g_j,g_0)\to0$ as $j\to\infty$.
This implies that every matrix entry of $g_j\gamma_0$ tends to the corresponding entry of $wuw^{-1}$,
i.e.\ we have $\|g_j\gamma_0-wuw^{-1}\|\to0$ as $j\to\infty$.
In particular there exists some $j$ such that
$\|g_j\gamma_0-wuw^{-1}\|<{\ve_0}$.
Now we have a contradiction against our previous assumption;
namely for our chosen $j$,
if we set $g:=g_j\gamma_0$ then $\Lambda_j=g\Z^d$
and $\|g-wuw^{-1}\|<{\ve_0}$.
This completes the proof of Lemma \ref{MinkowskiHajosCor2}.
\end{proof}

\begin{remark}\label{MinkowskiHajosCor2REM}
By a similar compactness argument as in 
the proof of Lemma \ref{MinkowskiHajosCor2},
one can also prove a more basic statement:
for any $\ve_0>0$ there exists $r_0>0$ such that every $\Lambda\in K_{r_0}$ has a representative 
$g\in G$ (i.e.\ $\Lambda=g\Z^d$) satisfying $\|g-wuw^{-1}\|<\ve_0$ for some $w\in W$ and $u\in U_0$.
The purpose of the choice of the more technical lattice representatives in Lemma \ref{MinkowskiHajosCor2} 
(with $U_{B\ve_0}$ in place of $U_0$)
is to ensure the following
property, which is a key 
ingredient in the proof of the important Lemma \ref{secondkeyboundLEM} below:
for any $g=(g_{ij})\in G$ satisfying $\|g-u\|<\ve_0$ for some $u\in U_{B\ve_0}$, 
and  any $1\leq i<j\leq d$,
we have (cf. \eqref{secondkeyboundLEMpf1})
\begin{align*}
|g_{ki}-g_{kj}|<1-B\ve_0\quad \textrm{for all $k\notin\{i,j\}$}.
\end{align*}
A crucial consequence of this is that if $g\vece_i-g\vece_j\notin\cC_r$ for some $r<B\ve_0$, 
then 
either
$|g_{ii}-g_{ij}|>1-r$ or $|g_{ji}-g_{jj}|>1-r$ must hold.
\end{remark}

\subsection{Bounds on off-diagonal symmetric pairs}

The next lemma shows that if a lattice $\Lambda$ {in} $K_r$ has a representative as in Lemma \ref{MinkowskiHajosCor2}, then its entries satisfy the desired bounds for proving Proposition \ref{prop:increlimub}. 

\begin{lem}\label{secondkeyboundLEM}
Let $A>1$,
$0<\ve_0<(16A)^{-1}$ and $0<r_0<\tfrac12 \ve_0$.
Let $g=(g_{ij})\in G$ and $r\in(0,r_0)$,
and assume that 
$g\Z^d\cap\scrC_r=\{\bn\}$ 
and that there exist $B\in[1,A]$ and $u\in U_{B\ve_0}$ such that 
$\|g-u\|<\ve_0$.
Then
\begin{align}\label{secondkeyboundLEMpf2}
0\leq g_{ii}-(1-r)\ll_d r\quad\text{for all }\: 1\leq i\leq d,
\end{align}
and
\begin{align}\label{secondkeyboundLEMres}
|g_{ij}g_{ji}|\ll_{d,\ve_0} r\quad\text{for all }\: 1\leq i<j\leq d.
\end{align}
\end{lem}
\begin{proof}
Let us write $\ve=B\ve_0$.
Note that $\|g-u\|<\ve_0\leq \ve$ and $u\in U_{\ve}$ together imply that
\begin{align}\label{secondkeyboundLEMpf1}
-\tfrac12+3\ve<g_{ij}<\tfrac12+2\ve
\quad\text{ and }\quad 
|g_{ji}|<\ve,
\quad\forall 1\leq i<j\leq d.
\end{align}
{It also follows that 
$|g_{ii}-1|<\ve$ for all $i$;
however these inequalities may be 
sharpened using Lemma~\ref{firstkeyboundLEM}.}
Indeed, we have $\ve<\frac1{16}$ since
$A\ve_0<\frac1{16}$;
hence the above inequalities imply that $\|g-u'\|<\frac18$ for some 
$u'\in U_0$. 
Hence Lemma \ref{firstkeyboundLEM} applies, yielding that 
\eqref{secondkeyboundLEMpf2} holds.

Now let $1\leq i<j\leq d$ be given.
We separate the proof of \eqref{secondkeyboundLEMres}
into three cases.
(Note that \eqref{secondkeyboundLEMres} holds trivially if
$g_{ij}=0$ or $g_{ji}=0$;
hence we may without loss of generality assume 
$g_{ij}g_{ji}\neq0$.) 

\textbf{Case I:} $g_{ij}>0$ and $g_{ji}>0$.
In this case we will build the proof on the fact that
$g\vece_i-g\vece_j\notin\scrC_r$,
which holds since 
$g\Z^d\cap\scrC_r=\{\bn\}$.
Using \eqref{secondkeyboundLEMpf1} and $r<r_0<2r_0\leq \ve<\frac1{16}$,
it follows that
$|g_{ki}-g_{kj}|<1-\ve<1-r$
for all $k\notin\{i,j\}$.
Hence we must have 
either
$|g_{ii}-g_{ij}|\geq1-r$ 
or $|g_{ji}-g_{jj}|\geq1-r$.
If $|g_{ii}-g_{ij}|\geq1-r$, then because of $g_{ii}\geq 1-r$ and $0<g_{ij}<\frac12+2\ve$
it follows that $g_{ii}-g_{ij}\geq 1-r$,
and so by \eqref{secondkeyboundLEMpf2},
$0<g_{ij}\ll r$.
Similarly if $|g_{ji}-g_{jj}|\geq1-r$
then $0<g_{ji}\ll r$.
In both cases, it follows that 
\eqref{secondkeyboundLEMres} holds for our $i,j$.

\textbf{Case II:} $g_{ij}<0$ and $g_{ji}<0$.
In this case we will prove the desired bound by proving the stronger assertion that either
$|g_{ij}|\leq Cr$ or $|g_{ji}|\leq Cr$, with $C:=12d(\ve_0/8)^{2-d}$. 
Assume the opposite, i.e.\ assume that
\begin{align}\label{secondkeyboundLEMpf5}
g_{ij}<-Cr 
\qquad\text{and}\qquad
g_{ji}<-Cr   
\qquad (C:=12d(\ve_0/8)^{2-d}).
\end{align}
We will prove that this leads to a contradiction.

Set $\scrJ_r:=\bigl(-\tfrac12(1-r),\tfrac12(1-r)\bigr)$,
so that $\frac12\scrC_r=\scrJ_r^d$.
For each $k\notin\{i,j\}$ we introduce the following open interval:
\begin{align}\label{tIkdef}
\widetilde{\scrI}_k:=\scrJ_r\cap(g_{ki}+\scrJ_r)\cap(g_{kj}+\scrJ_r).
\end{align}
Using $|g_{ki}|,|g_{kj}|,|g_{ki}-g_{kj}|<1-\ve$
(see \eqref{secondkeyboundLEMpf1})
it follows that $\widetilde{\scrI}_k$ has length
$|\widetilde{\scrI}_k|>(1-r)-(1-\ve)>\tfrac12\ve_0$.
Define $\scrI_k\subset \widetilde{\scrI}_k$ to be the open interval of length $\frac14\ve_0$ with the same center as $\widetilde{\scrI}_k$.
Let us also set $\scrI_i=\bigl(\frac12(1-r)+g_{ij},\frac12(1-r)\bigr)$
and $\scrI_j=\bigl(g_{jj}-\frac12(1-r),\frac34\bigr)$.
Then by construction,
\begin{align}\label{secondkeyboundLEMpf3}
\scrI_i\subset \scrJ_r,
\qquad
\scrI_i\cap(g_{ii}+\scrJ_r)=\varnothing,
\qquad
\scrI_i\cap(g_{ij}+\scrJ_r)=\varnothing,
\end{align}
and 
\begin{align}\label{secondkeyboundLEMpf4}
\scrI_j\subset g_{jj}+\scrJ_r,
\qquad
\scrI_j\cap \scrJ_r=\varnothing.
\end{align}
Furthermore, 
$|\scrI_i|=|g_{ij}|$ and $\frac18<|\scrI_j|<\frac12$.
Now let $\scrU$ be the open box
$\scrU=\scrI_1\times\cdots\times \scrI_d$.
Then $\vol(\scrU)>\frac18|g_{ij}|\cdot(\frac14\ve_0)^{d-2}>dr$,
where we used the first part of our assumption \eqref{secondkeyboundLEMpf5}.
Note also that $|\scrI_k|<\frac34$ for all $k$.
Hence by Lemma \ref{fCkeyLEM},
$\scrU\cap\fC_{g\Z^d,r}\neq\varnothing$,
i.e.\ there exists some $\vecv\in g\Z^d$
such that 
\begin{align*}
\scrU\cap(\vecv+\tfrac12\scrC_r)\neq\varnothing.
\end{align*}

It follows from the disjointness relations in \eqref{secondkeyboundLEMpf3}
and \eqref{secondkeyboundLEMpf4}
that $\scrU$ is disjoint from the three cubes
$g\vece_i+\tfrac12\scrC_r$, $g\vece_j+\tfrac12\scrC_r$
and $\tfrac12\scrC_r$;
hence $\vecv\notin\{\bn,g\vece_j,g\vece_i\}$.
Let us fix a point $\vecy\in \scrU\cap(\vecv+\frac12\scrC_r)$.
Then the line $\vecy+\R\vece_i$ goes through both the cubes
$\vecv+\frac12\scrC_r$
and $g\vece_j+\frac12\scrC_r$
(the latter holds since $\scrI_k\subset g_{kj}+\scrJ_r$ for all $k\neq i$;
see \eqref{tIkdef} and \eqref{secondkeyboundLEMpf4});
hence since these two cubes are disjoint,
we must have $v_i\geq g_{ij}+1-r$.
It also follows from $y_i\in v_i+\scrJ_r$ and $y_i\in \scrI_i$
that $v_i<y_i+\frac12(1-r)<1-r$.
In summary:
\begin{align}\label{secondkeyboundLEMpf6}
g_{ij}+1-r\leq v_i<1-r.
\end{align}
Similarly, using the fact that the line $\vecy+\R\vece_j$ 
goes through the two disjoint cubes
$\vecv+\frac12\scrC_r$
and $\frac12\scrC_r$,
and also using $y_j\in v_j+\scrJ_r$ and $y_j\in \scrI_j$,
it follows that
\begin{align}\label{secondkeyboundLEMpf7}
1-r\leq v_j<y_j+\tfrac12(1-r)
<\tfrac54.
\end{align}

Next, for each $k\notin\{i,j\}$,
since $y_k\in \scrI_k$ and $y_k\in v_k+\scrJ_r$, 
we have that both of the intervals 
$(y_k-\tfrac18\ve_0,y_k]$ and $[y_k,y_k+\tfrac18\ve_0)$
are contained in $\widetilde{\scrI}_k$,
and at least one of them is contained in $v_k+\scrJ_r$;
hence there exists an open subinterval $\scrI_k'$ of
$\widetilde{\scrI}_k\cap (v_k+\scrJ_r)$ of length $\frac18 \ve_0$.
Let us also set $\scrI_i'=(\tfrac23,\tfrac34)$ and $\scrI_j'=\bigl(g_{ji}+\tfrac12(1-r),v_j-\tfrac12(1-r)\bigr)$,
and then let $\scrU'$ be the open box
$\scrU'=\scrI_1'\times\cdots\times \scrI_d'$.
Using \eqref{secondkeyboundLEMpf7}
we have
$|\scrI_j'|\geq -g_{ji}$;
hence
$\vol(\scrU')\geq(\frac18\ve_0)^{d-2}\cdot\frac1{12}\cdot|g_{ji}|>dr$,
where we used the second part of our assumption
\eqref{secondkeyboundLEMpf5}.
Note also that $|\scrI_k'|<\frac34$ for all $k$;
for $k=j$ this uses $v_j<\frac54$ 
(see \eqref{secondkeyboundLEMpf7})
and $g_{ji}>-\ve$ (see \eqref{secondkeyboundLEMpf1}).
Hence, by Lemma \ref{fCkeyLEM},
$\scrU'\cap\fC_{g\Z^d,r}\neq\varnothing$,
i.e.\ there exists some $\vecv'\in g\Z^d$
such that 
\begin{align*}
\scrU'\cap(\vecv'+\tfrac12\scrC_r)\neq\varnothing.
\end{align*}
Choose a point $\vecz\in \scrU'\cap(\vecv'+\tfrac12\scrC_r)$.
Note that for every $k\notin\{i,j\}$
we have $\scrI_k'\subset(g_{ki}+\scrJ_r)\cap(v_{k}+\scrJ_r)$ by construction.
Furthermore, using $1-r\leq g_{ii}<1+\ve$
we have $\scrI_i'\subset g_{ii}+\scrJ_r$,
and using \eqref{secondkeyboundLEMpf6}
we have $\scrI_i'\subset v_{i}+\scrJ_r$.
Hence the line $\vecz+\R\vece_j$ goes through  of
the cubes $g\vece_i+\frac12\scrC_r$
and $\vecv+\frac12\scrC_r$.
Of course this line also goes through the cube
$\vecv'+\tfrac12\scrC_r$.
Note also that $z_j\in \scrI_j'$,
and by construction,
$\scrI_j'$ is disjoint from and lies between the two intervals
$g_{ji}+\scrJ_r$ and $v_j+\scrJ_r$.
Hence we must have $\vecv'\notin\{g\vece_i,\vecv\}$;
thus the three cubes
$g\vece_i+\frac12\scrC_r$,
$\vecv+\frac12\scrC_r$
and $\vecv'+\frac12\scrC_r$
are pairwise disjoint,
and the two intervals
$g_{ji}+\scrJ_r$ and $v_j+\scrJ_r$ must lie at a distance $\geq1-r$ from each other.
However, this is impossible,
since $v_j-g_{ji}<\frac54+\ve<2(1-r)$.
This completes the proof in Case II.

\textbf{Case III:}
$g_{ij}g_{ji}<0$.
If $g_{ji}<0$ then let us swap the values of $i$ and $j$;
thus from now on we have $g_{ji}>0$ and $g_{ij}<0$,
but 
either $i<j$ or $i>j$.
If $g_{ji}\leq g_{jj}-(1-r)$, then 
$g_{ji}\ll r$ by \eqref{secondkeyboundLEMpf2},
and so \eqref{secondkeyboundLEMres} holds for our $i,j$.
Hence from now on we may assume 
$g_{ji}>g_{jj}-(1-r)$.
Now set:
\begin{align*}
&\scrI_{i}=\bigl(g_{ij}+\tfrac12(1-r),g_{ii}-\tfrac12(1-r)\bigr);
\\[2pt]
&\scrI_{j}=\bigl(g_{jj}-\tfrac12(1-r),g_{ji}+\tfrac12(1-r)\bigr);
\\[2pt]
&\text{and }\: \widetilde{\scrI}_k=(g_{ki}+\scrJ_r)\cap(g_{kj}+\scrJ_r)
\quad\text{for $k\notin\{i,j\}$}.
\end{align*}
These are non-empty intervals.
Indeed, 
$\scrI_i$ is non-empty since $g_{ii}\geq1-r$ and $g_{ij}<0$;
$\scrI_j$ is non-empty because of our assumption 
$g_{ji}>g_{jj}-(1-r)$,
and for each $k\notin\{i,j\}$,
it follows from $|g_{ki}-g_{kj}|<1-\ve$
(see \eqref{secondkeyboundLEMpf1})
that $\widetilde{\scrI}_k$ is non-empty with $|\widetilde{\scrI}_k|>(1-r)-(1-\ve)\geq\tfrac12\ve_0$.
Now for each $k\notin\{i,j\}$
we choose an open subinterval $\scrI_k$ of $\widetilde{\scrI}_k$
of length $\min\big\{\frac34,|\widetilde{\scrI}_k|\big\}$,
and then define $\scrU$ to be the open box
$\scrU=\scrI_1\times\cdots\times \scrI_d$.
We claim that
\begin{align}\label{secondkeyboundLEMpf10}
\scrU\cap\fC_{g\Z^d,r}=\varnothing.
\end{align}
Indeed, assume the opposite;
then 
there is some
$\vecv\in g\Z^d$ with
$\scrU\cap\left(\vecv+\frac12\scrC_r\right)\neq\varnothing$.
By construction,
$\scrI_i$ is disjoint from both the intervals
$g_{ij}+\scrJ_r$ and $g_{ii}+\scrJ_r$;
hence $\scrU$ is disjoint from the two cubes
$g\vece_j+\frac12\scrC_r$,
$g\vece_i+\frac12\scrC_r$,
and thus $\vecv\notin\{g\vece_j,g\vece_i\}$.
Let $\vecx$ be a point in 
$\scrU\cap\left(\vecv+\frac12\scrC_r\right)$.
Then $x_i\in \scrI_i\cap(v_i+\scrJ_r)$.
Note also that the three intervals $g_{ij}+\scrJ_r$, $\scrI_i$, $g_{ii}+\scrJ_r$
are adjacent to each other in this order
along the real line,
with the length of $\scrI_i$ being
\begin{align*}
|\scrI_i|=g_{ii}-g_{ij}-(1-r)<1+\ve+(\tfrac12-3\ve)-(1-r)=\tfrac12-2\ve+r<\tfrac12<1-r.
\end{align*}
But $v_i+\scrJ_r$ has length $1-r$;
hence there exists a number $x_i'\in v_i+\scrJ_r$
lying either in $g_{ii}+\scrJ_r$ or $g_{ij}+\scrJ_r$.
Noticing also that $x_k\in \scrI_k\subset (g_{ki}+\scrJ_r)\cap (g_{kj}+\scrJ_r)$
for all $k\neq i$, 
it now follows that the point 
$\vecx+(x_i'-x_i)\vece_i$ lies in the cube $\vecv+\frac12\scrC_r$
and also in one of the two cubes
$g\vece_i+\frac12\scrC_r$
or $g\vece_j+\frac12\scrC_r$.
This is a contradiction against the fact that 
$\vecv+\frac12\scrC_r$
is disjoint from both
$g\vece_i+\frac12\scrC_r$
and $g\vece_j+\frac12\scrC_r$;
hence we have completed the proof of \eqref{secondkeyboundLEMpf10}. 

We have $|\scrI_k|<\frac34$ for all $1\leq k\leq d$;
hence Lemma \ref{fCkeyLEM} applies,
giving $\vol(\scrU)<dr$.
But $|\scrI_k|\geq\frac12\ve_0$ for all $k\notin\{i,j\}$,
and $|\scrI_i|\geq|g_{ij}|$;
thus 
\begin{align*}
|g_{ij}|\bigl(g_{ji}-g_{jj}+1-r\bigr)
=|g_{ij}||\scrI_j|\ll\vol(\scrU)\ll r.
\end{align*}
Furthermore, $|g_{ij}|\big(g_{jj}-(1-r)\big)\ll r$ by \eqref{secondkeyboundLEMpf2}.
Adding the last two bounds,
we conclude that 
\eqref{secondkeyboundLEMres} holds for our $i,j$.
\end{proof}

\subsection{Proof of Proposition \ref*{prop:increlimub}}
Finally, we can give the 

%
\begin{proof}[Proof of Proposition \ref{prop:increlimub}] 
Choose $0<a<1$ and $A>1$ as in Lemma~\ref{MinkowskiHajosCor2}.
Fix a number $0<\beta<\frac1{16}$
so small that for every matrix $g\in G$ which has distance
(w.r.t.\ $\|\cdot\|$) less than $\beta$ to a matrix in $U_0$,
we have $|\det \tilde{g}-1|<\frac12$,
where $\tilde{g}$ is the top left $(d-1)\times(d-1)$ block of $g$.
Fix a number $0<\ve_0<\min\bigl\{a,\beta/(2A)\bigr\}$,
and for this $\ve_0$, take $r_0>0$ as in Lemma \ref{MinkowskiHajosCor2}.
Note that the defining property of $r_0$
trivially remains valid if we decrease $r_0$;
hence we may assume that $0<r_0<\frac12\ve_0$.
Let $C>0$ be the maximum of the implied constants in the two
``$\ll$'' bounds in Lemma \ref{secondkeyboundLEM},
for our fixed $d$ and $\ve_0$. 

Now let $r\in(0,r_0)$ and $\Lambda\in K_r$ be given.
This means that $\Lambda\cap\scrC_r=\{\bn\}$,
and a fortiori,
$\Lambda\cap\scrC_{r_0}=\{\bn\}$.
Hence by our choice of
$A,a,\ve_0,r_0$ (see the statement of Lemma \ref{MinkowskiHajosCor2}),
there exist $g'\in G$,
$w\in W$, $B\in[1,A]$ and $u\in U_{B{\ve_0}}$ such that
$\Lambda=g'\Z^d$ and $\|g'-wuw^{-1}\|<{\ve_0}$.
Note that the norm $\|\cdot\|$ is preserved by left and right multiplication by elements from $W$;
hence letting $g=(g_{ij}):=w^{-1}g'w$ we have $\|g-u\|<\ve_0$, 
and also $g\Z^d\cap\scrC_r=w^{-1}(g'\Z^d\cap\scrC_r)=\{\bn\}$ 
(this is true since $w\Z^d=\Z^d$ and $\scrC_r=w^{-1}\scrC_r$).
Hence by Lemma \ref{secondkeyboundLEM},
and by our choice of $C$,
we have $0\leq g_{ii}-(1-r)\leq Cr$ for all $i$ 
and $|g_{ij}g_{ji}|\leq Cr$ for all $i\neq j$. 
Furthermore, it follows from $u\in U_{B{\ve_0}}$ and $\|g-u\|<\ve_0$
that $\|g-u'\|<(B+1)\ve_0$ for some $u'\in U_0$;
hence a fortiori
$\|g-u'\|<2A\ve_0<\beta$,
which implies that $|\det \tilde{g}-1|<\frac12$
by our choice of $\beta$.
It also follows that 
$|g_{ij}|<\frac12+2A\ve_0<1$ for all $i\neq j$. 
Hence $g\in{\overline{\cK}_{r,C}}$, 
and thus $g'=wgw^{-1}\in w\,{\overline{\cK}_{r,C}}\,w^{-1}$
and $\Lambda=g'\Z^d\in\pi(w\,{\overline{\cK}_{r,C}}\,w^{-1})$, finishing the proof.
\end{proof}

%
%

\section{Measure estimates of the thickenings}
Fix $m,n\in \N$ and let $d=m+n$. Let $\bm{\alpha}\in \R^m$ and $\bm{\beta}\in \R^n$ be two fixed {weight vectors} as in Theorem~\ref{thm:improvedirichleconwei}.
As mentioned in Remark \ref{rmk:effequdoumixww}, in order to incorporate the case of general weights, 
we need to consider a more general one-parameter subgroup of $G$ associated to $\bm{\alpha}$ and $\bm{\beta}$. 
Explicitly, for any $s\in \R$ let us define
\begin{align}\label{equ:geneweflow}
g_s=g_{s}^{\bm{\alpha},\bm{\beta}}:=\diag(e^{\alpha_1s},\ldots, e^{\alpha_ms}, e^{-\beta_1s},\ldots, e^{-\beta_ns}) \in G.
\end{align}
Let $\Delta: X_d\to [0,\infty)$ be the function defined in \eqref{def:delta}. 
  The main result of this section is an asymptotic estimate for the measure of the thickened set
\begin{align*}
\widetilde{\Delta}_r:=\bigcup_{0\leq s<1}g_{-s}\,\Delta^{-1}[0, r],
\end{align*}
when $r>0$ is small.
\begin{thm}\label{thm:volumethickening}
Let $\varkappa_d=\frac{d^2+d-4}{2}$ and $\lambda_d=\frac{d(d-1)}{2}$ be as in Theorem \ref{thm:improvedirichleconwei}. Then 
\begin{align*}
\mu_d\bigl(\widetilde{\Delta}_r\bigr)\asymp_{d} r^{\varkappa_d}\log^{\lambda_d}\Bigl(\frac{1}{r}\Bigr),\qquad \textrm{as $r\to 0^+$},
\end{align*}
{where the implicit constant is independent of $r$ and the two weight vectors $\bm{\alpha}$ and $\bm{\beta}$.}
\end{thm}


Just as for Theorem \ref{thm:measestmainaux}, 
we prove Theorem \ref{thm:volumethickening} by proving the upper and lower bounds separately. 

\subsection{Proof of the upper bound}\label{sec:upper}
Note that for any $(\bm{x},\bm{y})\in \R^m\times \R^n$, 
$$
g_{-s}\left(\begin{smallmatrix}
\bm{x} \\
\bm{y}\end{smallmatrix}\right)=(e^{-\alpha_1s}x_1,\ldots, e^{-\alpha_ms}x_m, e^{\beta_1s}y_1,\ldots, e^{\beta_ns}y_n)^t.
$$
This implies that
\begin{align*}
|\Delta(g_{-s} \Lambda)-\Delta(\Lambda)|\leq \max 
\bigl\{\alpha_1,\ldots,\alpha_m,\beta_1,\ldots,\beta_n\bigr\}|s|< |s|,\qquad \forall\ s\in \R,\  \Lambda\in X_d.
\end{align*}
Hence
\begin{align}\label{equ:increla}
g_{-s} \Delta^{-1}\big[0,r]\subset \Delta^{-1}[0, r+|s|\big],\qquad \forall\ s\in \R.
\end{align}
Given any $r\in (0,1)$, let $q=\left \lceil{1/r}\right \rceil$. 
Using \eqref{equ:increla} and the fact that $1/q\leq r$, we have 
\begin{align*}
\widetilde{\Delta}_r=\bigcup_{0\leq s<1}g_{-s}\Delta^{-1}[0, r]
=\bigcup_{k=0}^{q-1}\,\bigcup_{0\leq s<1/q}g_{-k/q}\,g_{-s}\, \Delta^{-1}[0, r]\subset \bigcup_{k=0}^{q-1}g_{-k/q}\, \Delta^{-1}[0, 2r],
\end{align*}
implying that (using the $G$-invariance of $\mu_d$ and $q 
\asymp r^{-1}$)
\begin{align*}
\mu_d\bigl(\widetilde{\Delta}_r\bigr)\leq 
\sum_{k=0}^{q-1}\mu_d\bigl(g_{-k/q}\, \Delta^{-1}[0, 2r]\bigr)\asymp r^{-1}\mu_d\big(\Delta^{-1}[0, 2r]\big).
\end{align*}
Finally, by Theorem \ref{thm:mainthm} we get
\begin{align*}
\mu_d\bigl(\widetilde{\Delta}_r\bigr)\ll_d r^{-1}r^{\varkappa_d+1}\log^{\lambda_d}\Bigl(\frac{1}{r}\Bigr)=r^{\varkappa_d}\log^{\lambda_d}\Bigl(\frac{1}{r}\Bigr),\qquad \textrm{as $r\to 0^+$}.
\end{align*}
This finishes the proof of the upper bound in Theorem \ref{thm:volumethickening}.
\subsection{Proof of the lower bound}
In this subsection we prove the lower bound in Theorem \ref{thm:volumethickening}. 
By the discussion in the beginning of Section \ref{sec:premest}, 
we may replace the set $\Delta^{-1}[0,r]$ by $K_r$, that is, it suffices to prove the following lower bound
\begin{align}\label{equ:allowerbdtick}
\mu_d\Biggl(\,\bigcup_{0\leq s<1}g_{-s}\, K_{r}\Biggr)\gg_d r^{\varkappa_d}\log^{\lambda_d}\Bigl(\frac{1}{r}\Bigr),\qquad \text{as $r\to 0^+$}.
\end{align}

The following lemma is the crucial ingredient in our proof of \eqref{equ:allowerbdtick}.
Let $c_0$ be the small parameter which we fixed in Section \ref{sec:smpara};
after possibly shrinking $c_0$, we may without loss of generality assume that $0<c_0<(3e)^{-1}$.
For $r\in (0, c_0/d)$, let $\underline{K}_r\subset K_r$ be as in \eqref{def:urlowerbd}.
\begin{lem}\label{crucialdisjLEM}
For any $r\in (0,c_0/d)$ and $s\in[r,1)$,
the two sets 
$\underline{K}_r$ and $g_{-s}\,\underline{K}_r$ are disjoint.
\end{lem}
\begin{proof}
{Assume the opposite; then there exist
$r\in (0,c_0/d)$, $s\in[r,1)$
such that $\underline{K}_r\cap g_{-s}\underline{K}_r\neq\varnothing$,
or equivalently $g_s\underline{K}_r\cap \underline{K}_r\neq\varnothing$.
Pick a lattice $\Lambda\in g_s\underline{K}_r\cap \underline{K}_r$;
thus $\Lambda\in\uK_r$ and $g_{-s}\Lambda\in\uK_r$.
}
By the definition of $\uK_r$ in \eqref{def:urlowerbd},
we now have $\Lambda=p_{\bm{b}_1,\dots,\bm{b}_{d-1}}u_{\bm{x}}\Z^d$ for some
vectors $\bm{b}_j=(b_{1j},\ldots,b_{dj})^t\in\R^d$ ($j=1,\ldots,d-1$)
and $\bm{x}\in (0,c_0/d)^{d-1}$
satisfying 
\eqref{equ:con1}, \eqref{equ:con2}, \eqref{equ:con3} {and} \eqref{equ:con4};
{moreover}, we also have
$g_{-s}\Lambda=p_{\bm{b}'_1,\dots,\bm{b}'_{d-1}}u_{\bm{x}'}\Z^d$ for some
vectors $\bm{b}'_j=(b_{1j}',\ldots,b_{dj}')^t\in\R^d$ ($j=1,\ldots,d-1$)
and $\bm{x}'\in (0,c_0/d)^{d-1}$
which again satisfy
\eqref{equ:con1}, \eqref{equ:con2}, \eqref{equ:con3} {and} \eqref{equ:con4}.

Because of $\vecalf\in(\R_{>0})^m$ and $\sum_{i=1}^m\alpha_i=1$,
there exists an index $1\leq i\leq m$ such that $\alpha_i\geq \tfrac{1}{m}$. 
Fixing such an $i$, we consider the vector 
\begin{align*}
\vecy=(y_1,\ldots,y_d)^t:=g_{-s}\vecb_i
=\bigl(e^{-\alpha_1s}b_{1i},\ldots,e^{-\alpha_ms}b_{mi},e^{\beta_1s}b_{m+1,i},\ldots,e^{\beta_ns}b_{d,i}\bigr)^t.
\end{align*}
By \eqref{equ:con1} we have 
$|b_{ji}|<c_0$ for all $j\neq i$, and 
$1-\frac r{2d}<b_{ii}<1$.
Hence $|y_j|<ec_0$ for all $j\neq i$ (since $0<s<1$ and $\beta_\ell\leq1$ for all $1\leq\ell\leq n$),
and $0<y_i<e^{-\alpha_is}<e^{-s/m}<1-\frac s{2m}<1-\frac r{2d}$.
We have $\vecb_i\in\Lambda\smallsetminus\{\bn\}$
and thus $\vecy=g_{-s}\vecb_i\in g_{-s}\Lambda\smallsetminus\{\bn\}$;
also $g_{-s}\Lambda\in\uK_r\subset K_r$,
and hence $\vecy\notin\scrC_r$.
But for all $j\neq i$ we have $|y_j|<ec_0<\frac13<1-r$
(indeed, recall that 
$0<c_0<(3e)^{-1}$);
also $y_i>0$; hence $\vecy\notin\scrC_r$ implies $y_i\geq 1-r$.
In summary:  
\begin{align*}
|y_j|<ec_0\quad(\forall\ j\neq i)
\qquad\text{and}\quad
1-r\leq y_i<1-\tfrac r{2d}.
\end{align*}
Furthermore, $\vecb_i'\in g_{-s}\Lambda$,
since $g_{-s}\Lambda=p_{\bm{b}'_1,\dots,\bm{b}'_{d-1}}u_{\bm{x}'}\Z^d$;
and 
by \eqref{equ:con1} we have 
$|b'_{ji}|<c_0$ for all $j\neq i$ and 
$1-\frac r{2d}<b'_{ii}<1$.
It follows that $|b'_{ji}-y_j|<(e+1)c_0<(e+1)(3e)^{-1}<\frac12<1-r$ for all $j\neq i$,
and $|b'_{ii}-y_i|<r<1-r$;
hence $\vecb'_i-\vecy\in\scrC_r$.
Note also that $\vecb_i'-\vecy\in g_{-s}\Lambda$,
and $\vecb_i'-\vecy\neq\bn$, since $y_i<1-\frac r{2d}<b'_{ii}$;
hence we have obtained a contradiction against $g_{-s}\Lambda\in\uK_r\subset K_r$.
This completes the proof of the lemma.
\end{proof}

It follows from Lemma \ref{crucialdisjLEM} that 
for any $r\in(0,c_0/d)$,
the sets $g_{-kr}\,\uK_r$, 
for $k$ running through the integers in the interval
$0\leq k<1/r$,
are pairwise disjoint.
(Indeed, if $g_{-kr}\,\uK_r\cap g_{-k'r}\,\uK_r\neq\varnothing$
for some $0\leq k<k'<1/r$ then $\uK_r\cap g_{(k-k')r}\,\uK_r\neq\varnothing$,
contradicting Lemma \ref{crucialdisjLEM}.)
Hence, using also $K_r\supset\uK_r$,
we have
\begin{align*}
\mu_d\Biggl(\,\bigcup_{0\leq s<1}g_{-s}\, K_{r}\Biggr)
\geq\mu_d\Biggl(\,\bigcup_{0\leq k<1/r}g_{-kr}\, \uK_{r}\Biggr)
=\sum_{0\leq k<1/r}\mu_d\bigl(g_{-kr}\, \uK_{r}\bigr)
=\#\bigl(\Z\cap[0,1/r)\bigr)\cdot\mu_d\bigl(\uK_r\bigr),
\end{align*}
Here $\#\big(\Z\cap[0,1/r)\big)\gg r^{-1}$,
and for $r$ sufficiently small we have
$\mu_d\bigl(\uK_r\bigr)\gg r^{\varkappa_d+1}\log^{\lambda_d}\bigl(\frac{1}{r}\bigr)$
by \eqref{uKrlb}.
Hence we obtain the lower bound \eqref{equ:allowerbdtick},
and the proof of Theorem \ref{thm:volumethickening} is complete.
\hfill$\square$



\section{Some preliminaries for Theorem \ref*{thm:improvedirichleconwei}}

In this section we collect some preliminary results for our proof of Theorem \ref{thm:improvedirichleconwei}.

\subsection{Dynamical interpretation of weighted $\psi$-Dirichlet matrices}
Let $m,n\in\N$ and let $\bm{\alpha}\in\R^m$ and $\bm{\beta}\in \R^n$ be two fixed weight vectors as before.
Let $t_0>0$ and let $\psi: [t_0,\infty)\to (0,\infty)$ be a continuous 
decreasing function which tends to zero at infinity. 
In this subsection we give a dynamical interpretation of $\psi_{\bm{\alpha},\bm{\beta}}$-Dirichlet matrices which generalizes \cite[Proposition 4.5]{KleinbockWadleigh2018}; see Proposition \ref{prop:dyint}. Let us first introduce the following modified Dani Correspondence which is a special case of \cite[Lemma 8.3]{KleinbockMargulis1999}.
\begin{lem}\label{lem:danicor}
Fix $m, n\in \N$ and let $d=m+n$. Let $t_0>0$, and let $\psi:[t_0,\infty)\to (0,\infty)$ be a continuous, decreasing function satisfying \eqref{equ:con1psi} and \eqref{equ:con2psi}. Then there exists a unique continuous, decreasing function 
$$r=r_{\psi}: [s_0,\infty)\to (0,\infty),\quad \textrm{where $s_0=\frac{m}{d}\log t_0-\frac{n}{d}\log \psi(t_0)$},$$
such that 
\begin{equation}\label{equ:con2id}
\textrm{the function $s\mapsto s+mr(s)$ is increasing,}
\end{equation}
and
\begin{equation}\label{equ:con3id}
\psi\left(e^{s-nr(s)}\right)=e^{-s-mr(s)}\quad \textrm{for all $s\geq s_0$}.
\end{equation}
Conversely, given $s_0\in\R$ and a continuous, decreasing function $r: [s_0,\infty)\to (0,\infty)$ satisfying \eqref{equ:con2id}, there exists a unique continuous, decreasing function $\psi=\psi_r: [t_0,\infty)\to (0,\infty)$ with $t_0=e^{s_0-nr(s_0)}$ satisfying \eqref{equ:con1psi}, \eqref{equ:con2psi} and \eqref{equ:con3id}. Furthermore, 
for any fixed $\alpha, \beta> 0$ the series 
\begin{align}\label{equ:crticalseries}
\sum_{k\geq t_0}\frac{\big(1-k\psi(k)\big)^{\alpha}\big(-\log\left(1-k\psi(k)\right)\big)^{\beta}}{k}
\end{align}
diverges if and only if the series
\begin{align}\label{equ:series}
\sum_{k\geq s_0}r(k)^{\alpha}\log^{\beta}\Bigl(1+\frac{1}{r(k)}\Bigr)
\end{align}
diverges.
\end{lem}

\begin{proof}
The $\psi$ and $r$-functions determine each other uniquely via the 
relation
\begin{align}\label{equ:keyrelation}
\psi(t)^{1/m}e^{s/m}=t^{1/n}e^{-s/n}=e^{-r(s)},
\end{align}
which captures the moment when the $a_s$-flow transforms the long {and} thin `rectangle'
$$
\big\{(\bm{x}, \bm{y})\in \R^m\times \R^n\col \|\bm{x}\|^m<\psi(t),\ \|\bm{y}\|^n<t\big\}
$$ 
determined by \eqref{equ:dirichlet} into a cube (with side length $2e^{-r(s)}$). Here $a_s=\diag(e^{s/m}I_m,e^{-s/n}I_n)$, as defined in the introduction. 
This correspondence between $\psi(\cdot)$ and $r(\cdot)$ is a special case of 
\cite[Lemma 8.3]{KleinbockMargulis1999},
as here we assume that $\psi(\cdot)$ additionally satisfies \eqref{equ:con1psi} and \eqref{equ:con2psi}, which on the $r$-function side corresponds respectively to the assumptions that $r(\cdot)$ is decreasing and $r(s)>0$ for all $s\geq s_0$. 
The equivalence of these additional assumptions is easily checked using the following three relations,
which follow from \eqref{equ:keyrelation}:
\begin{align}\label{equ:keyrel}
e^{-dr(s)}=t\psi(t),\qquad s=\frac{m}{d}\log t-\frac{n}{d}\log\psi(t),\qquad \textrm{and}\qquad t=e^{s-nr(s)}.
\end{align}

Finally we prove the equivalence of the divergence of the two series.
If $\lim_{t\to\infty}t\psi(t)<1$ then both the functions $1-t\psi(t)$ and $r(s)$ are bounded away from zero
(and positive), which implies that the two series in
\eqref{equ:crticalseries} and \eqref{equ:series} are divergent.
Hence from now on we may assume that 
$\lim_{t\to\infty}t\psi(t)=1$.
Then $\lim_{s\to\infty}r(s)=0$
(by \eqref{equ:keyrel}),
and 
$F_{\psi}(t):=1-t\psi(t)$
is a decreasing function taking values in the interval $(0,1)$ and satisfying $\lim_{t\to\infty}F_{\psi}(t)=0$.
After enlarging $s_0$ (thus also enlarging $t_0$)
we may assume that $0<r(s)<1/d$ for all $s\geq s_0$.
Then by \eqref{equ:keyrel} we have,
with $t=t(s)=e^{s-nr(s)}$:
\begin{align}\label{divequivpf1}
\frac d2r(s)< F_{\psi}(t) <d\,r(s)\quad\text{and }\quad e^{s-1}<t<e^s,\qquad\forall\ s\geq s_0.
\end{align}
It follows that
$r(s)^\alpha\log^\beta\bigl(1+\frac1{r(s)}\bigr)\asymp_{d,\alpha,\beta} 
F_{\psi}(t)^{\alpha}\log^{\beta}\bigl(\frac{1}{F_{\psi}(t)}\bigr)$
for all $s\geq s_0$.
Hence, using also $e^{k-1}<t(k)<e^k$ (see \eqref{divequivpf1}),
the fact that $F_\psi(t)$ is decreasing, 
and $\sum_{e^k\leq j<e^{k+1}}\frac1j\asymp1$ ($\forall k\geq1$),
we have for all sufficiently large integers $k$:
\begin{align*}
r(k)^\alpha\log^\beta\Bigl(1+\frac1{r(k)}\Bigr)
\ll_{d,\alpha,\beta} F_\psi(e^{k-1})^\alpha\log^\beta\Bigl(\frac1{F_\psi(e^{k-1})}\Bigr)
\ll_{\alpha,\beta}\sum_{e^{k-2}\leq j<e^{k-1}}\frac1jF_\psi(j)^\alpha\log^\beta\Bigl(\frac1{F_\psi(j)}\Bigr),
\end{align*}
and similarly
\begin{align*}
r(k)^\alpha\log^\beta\Bigl(1+\frac1{r(k)}\Bigr)
\gg_{d,\alpha,\beta} F_\psi(e^k)^\alpha\log^\beta\Bigl(\frac1{F_\psi(e^k)}\Bigr)
\gg_{\alpha,\beta}\sum_{e^k\leq j<e^{k+1}}\frac1jF_\psi(j)^\alpha\log^\beta\Bigl(\frac1{F_\psi(j)}\Bigr).
\end{align*}
It follows that the series in \eqref{equ:series} diverges if and only if 
$\sum_{j} j^{-1}F_\psi(j)^\alpha\log^\beta\bigl(\frac1{F_\psi(j)}\bigr)$
diverges,
that is, if and only if the series in \eqref{equ:crticalseries} diverges.
\end{proof}
\begin{remark}\label{rmk:equivinte}
Let $\psi$ and $r$ be as in Lemma \ref{lem:danicor} with $\lim_{t\to\infty}t\psi(t)=1$. Let $F_{\psi}(t)=1-t\psi(t)$ be as above. 
Assume that the series \eqref{equ:crticalseries} 
(and thus also the series \eqref{equ:series})
diverges for some $\alpha,\beta>0$.
It is then not difficult to see from the proof of Lemma~\ref{lem:danicor} 
that for any $\beta'\geq \beta$ and for all large $s_1>s_0$,
\begin{align*}
\sum_{s_0\leq k\leq s_1}r(k)^{\alpha}\log^{\beta'}\Bigl(1+\frac{1}{r(k)}\Bigr)\asymp_{d,\alpha,\beta'}
\sum_{t_0\leq k\leq e^{s_1}}k^{-1}F_{\psi}(k)^{\alpha}\log^{\beta'}\Bigl(\frac{1}{F_{\psi}(k)}\Bigr).
\end{align*}
In particular it follows that for any $\beta'>\beta$ we have the following equivalence:
\begin{align*}
\liminf_{s_1\to\infty}
\frac{\sum_{s_0<k\leq s_1}r(k)^{\alpha}\log^{\beta'}\bigl(1+\frac1{r(k)}\bigr)}
{\Bigl(\sum_{s_0<k\leq s_1}r(k)^{\alpha}\log^{\beta}\bigl(1+\frac1{r(k)}\bigr)\Bigr)^2}
=0
\hspace{180pt}
\\
\Longleftrightarrow
\quad\liminf_{t_1\to\infty}\frac{\sum_{t_0\leq k\leq t_1}k^{-1}F_{\psi}(k)^{\alpha}\log^{\beta'}\left(\frac{1}{F_{\psi}(k)}\right)}
{\left(\sum_{t_0\leq k\leq t_1}k^{-1}F_{\psi}(k)^{\alpha}\log^{\beta}\left(\frac{1}{F_{\psi}(k)}\right)\right)^2}=0.
\end{align*}
Similarly, the above two limits inferior remain bounded simultaneously.
\end{remark}

We now state the dynamical interpretation of $\psi_{\bm{\alpha},\bm{\beta}}$-Dirichlet matrices. 
\begin{prop}
\label{prop:dyint}
Let $\psi$ be as in Theorem \ref{thm:improvedirichleconwei}, and let $r=r_{\psi}$ be as in Lemma \ref{lem:danicor}. Let $\{g_s\}_{s\in\R}$ be the one-parameter subgroup associated to the two fixed weight vectors $\bm{\alpha}$ and $\bm{\beta}$ as in \eqref{equ:geneweflow}.
Set
$$\omega_1:=\max\{m\alpha_i, n\beta_j\col 1\leq i\leq m, 1\leq j\leq n\}\quad\text{and}\quad\omega_2:=\min\{m\alpha_i, n\beta_j\col 1\leq i\leq m, 1\leq j\leq n\}.$$
Then for any $A\in M_{m,n}(\R)$ we have,
with $\Lambda_A$ as in \eqref{equ:submainfold}:
\begin{itemize}
\item[(1)] if $\Delta(g_s\Lambda_A)>\omega_1r(s)$ for all sufficiently large $s$, then $A$ is $\psi_{\bm{\alpha},\bm{\beta}}$-Dirichlet;
\item[(2)] if $\Delta(g_s\Lambda_A)\leq \omega_2r(s)$ for an unbounded set of $s$, then $A$ is not $\psi_{\bm{\alpha},\bm{\beta}}$-Dirichlet.
\end{itemize} 
\end{prop}
\begin{remark}\label{rmk:gen}
When $\bm{\alpha}=(\tfrac1m,\ldots, \tfrac1m)\in \R^m$ and $\bm{\beta}=(\tfrac1n,\ldots,\tfrac1n)\in\R^n$, then $\omega_1=\omega_2=1$ and Proposition \ref{prop:dyint} recovers \cite[Proposition 4.5]{KleinbockWadleigh2018}.
\end{remark}
\begin{proof}[Proof of Proposition \ref{prop:dyint}]
For any $t>\max\{t_0,1\}$, define
\begin{align*}
\mathcal{R}_t=\mathcal{R}^{\bm{\alpha},\bm{\beta}}_t
:=\big\{(\bm{x},\bm{y})\in \R^m\times \R^n\col \|\bm{x}\|_{\bm{\alpha}}<\psi(t), \|\bm{y}\|_{\bm{\beta}}<t\big\},
\end{align*}
 so that $(\bm{p},\bm{q})\in \Z^m\times (\Z^n\smallsetminus \{\bm{0}\})$ is a solution to \eqref{equ:dirichletwei} if and only if 
 $\left(A\bm{q}-\bm{p}, \bm{q}\right)\in \mathcal{R}_t$. 
On the other hand, the lattice $\Lambda_A$ consists exactly of the points
\begin{align*}
\left(\begin{matrix}
I_m & A\\
0 & I_n\end{matrix}\right)
\left(\begin{matrix}
-\bm{p} \\
\bm{q}\end{matrix}\right)
=\left(\begin{matrix}
A\bm{q}-\bm{p} \\
\bm{q}\end{matrix}\right) \qquad \textrm{for}\ (\bm{p},\bm{q})\in \Z^m\times \Z^n.
\end{align*}
Moreover, if $(A\bm{q}-\bm{p},\bm{q})\in \Lambda_A\cap \scrR_t$ is nonzero for some $(\bm{p},\bm{q})\in \Z^m\times \Z^n$, then we must have $\bm{q}\neq \bm{0}$. 
Indeed, otherwise we would have $\|A\bm{q}-\bm{p}\|_{\bm{\alpha}}=\|\bm{p}\|_{\bm{\alpha}}\geq 1$, but $(A\bm{q}-\bm{p},\bm{q})\in \scrR_t$ implies that $\|A\bm{q}-\bm{p}\|_{\bm{\alpha}}<\psi(t)<1/t<1$
(since $t>\max\{t_0,1\}$), contradicting $\|A\bm{q}-\bm{p}\|_{\bm{\alpha}}\geq 1$.
Thus there exists a solution $(\bm{p},\bm{q})\in \Z^m\times (\Z^n\smallsetminus \{\bm{0}\})$ to \eqref{equ:dirichletwei} if and only if $\Lambda_A\cap \mathcal{R}_t\neq \{\bm{0}\}$,
implying that $A\in M_{m,n}(\R)$ is $\psi_{\bm{\alpha},\bm{\beta}}$-Dirichlet if and only if $\Lambda_A\cap \mathcal{R}_t\neq \{\bm{0}\}$ for all 
sufficiently large $t$. 
Now let $s=s(t)=\frac{m}{d}\log t-\frac{n}{d}\log \psi(t)$;
then $s\to\infty$ if and only if $t\to\infty$, and by \eqref{equ:keyrelation} we have
\begin{align*}
g_s \mathcal{R}_t=\left\{(\bm{x}', \bm{y}')\in \R^m\times \R^n\col \|\bm{x}'\|_{\bm{\alpha}}<e^{-mr(s)},\ \|\bm{y}'\|_{\bm{\beta}}<e^{-nr(s)}\right\}=:\mathcal{E}_s.
\end{align*}
It follows that $A$ is $\psi_{\bm{\alpha},\bm{\beta}}$-Dirichlet if and only if $g_s\Lambda_A\cap \mathcal{E}_s\neq \{\bm{0}\}$ for all sufficiently large $s$. Next, note that we have the following simple relation:
\begin{align}\label{equ:relatcube}
\bigl(-e^{-\omega_1r(s)}, e^{-\omega_1r(s)}\bigr)^d\subset \mathcal{E}_s\subset \bigl(-e^{-\omega_2r(s)}, e^{-\omega_2r(s)}\bigr)^d,
\end{align}
with $\omega_1,\omega_2$ defined as in the statement of the proposition. Note also that $\Delta(g_s\Lambda_A)>\omega_1r(s)$ is equivalent with $g_s\Lambda_A\cap \left(-e^{-\omega_1r(s)}, e^{-\omega_1r(s)}\right)^d\neq \{\bm{0}\}$. 
Hence, using the first inclusion relation in \eqref{equ:relatcube} we have
$$
\Delta(g_s\Lambda_A)>\omega_1r(s)\ \textrm{for all sufficiently large $s$}\ \Rightarrow g_s\Lambda_A\cap \mathcal{E}_s\neq\{\bm{0}\}\ \textrm{for all sufficiently large $s$},
$$ 
and the latter condition implies that $A$ is $\psi_{\bm{\alpha},\bm{\beta}}$-Dirichlet.
We have thus proved part (1) of the proposition. 
Similarly, part (2) follows using the second inclusion relation in \eqref{equ:relatcube}.
\end{proof}
Let $\psi$ and $r=r_{\psi}$ be as above. 
For any integer $k>s_0$, let us define
\begin{align*}
\underline{B}_k:=\bigcup_{0\leq s<1}g_{-s} \Delta^{-1}[0, \omega_2r(k+s)]\quad \textrm{and}\quad \overline{B}_k:=\bigcup_{0\leq s<1}g_{-s} \Delta^{-1}[0, \omega_1r(k+s)].
\end{align*}
It follows 
that for any $\Lambda\in X_d$, we have $g_k \Lambda\in \underline{B}_k$ 
(respectively $g_k \Lambda\in \overline{B}_k$)
if and only if
there is some $k\leq s<k+1$ such that
$g_s \Lambda\in \Delta^{-1}[0, \omega_2r(s)]$
(respectively $g_s\Lambda\in \Delta^{-1}[0, \omega_1r(s)]$).
In particular, in view of Proposition \ref{prop:dyint}, 
a given matrix $A\in M_{m,n}(\R)$ is $\psi_{\bm{\alpha},\bm{\beta}}$-Dirichlet if $g_k \Lambda_A\notin \overline{B}_k$ for all sufficiently large $k$, or equivalently, if $g_k\Lambda_A\in \overline{B}_k$ holds only finitely often.
Similarly, $A$ is not $\psi_{\bm{\alpha},\bm{\beta}}$-Dirichlet if $g_k \Lambda_A\in \underline{B}_k$ holds infinitely often.

\subsection{Effective equidistribution and doubly mixing for certain $g_s$-translates}\label{sec:effequdomix}
Let $m,n\in\N$ and $d=m+n$ be as before.
Let
$$
\mathcal{Y} = \left\{\Lambda_A: A\in M_{m,n}(\R)\right\}\cong M_{m,n}(\R/\Z)
$$
be defined as in \eqref{equ:submainfold}, and recall that $\scrY$ is equipped with the probability Lebesgue measure, $\textrm{Leb}$. 

As mentioned in Remark \ref{rmk:effequdoumixww}, we will need an effective equidistribution and doubly mixing result for the $g_s$-translates $\{g_s\scrY\}_{s>0}$ which is analogous to \eqref{equ:effequidww} and \eqref{equ:effedobmixdww} respectively. In fact, we will state a corollary of a more general effective mixing result of arbitrary order proved by Bj\"orklund and Gorodnik \cite[Theorem 2.2]{BjorklundGorodnik2019}. To state their result, let us first fix some notation.

Let $\mathfrak{g}=\mathfrak{sl}_d(\R)$ be the Lie algebra of $G$. For each $Y\in \mathfrak{g}$, let us denote by $\mathcal{D}_Y$ the corresponding Lie derivative (a first order differential operator) on $C^{\infty}(G)$ defined by
\begin{align*}
\mathcal{D}_Y(f)(g):=\frac{\text{d} }{\d t}f(\exp(tY)g)\big|_{t=0},\qquad  f\in C^{\infty}(G).
\end{align*}
Here $\exp: \mathfrak{g}\to G$ denotes the usual exponential map from $\mathfrak{g}$ to $G$. Note that this definition naturally extends to the function space $C_c^{\infty}(X_d)$ since we can view elements in $C_c^{\infty}(X_d)$ as right $\G$-invariant smooth functions on $G$.
Fix an ordered basis $\{Y_1,\ldots, Y_{a}\}$ of $\mathfrak{g}$. 
Then every monomial $Z=Y_1^{\ell_1}\cdots Y_{a}^{\ell_{a}}$ defines a 
differential operator of degree $\deg(Z):=\ell_1+\cdots+\ell_{a}$, via
\begin{align*}
\mathcal{D}_Z:=\mathcal{D}_{Y_1}^{\ell_1}\circ\cdots\circ\mathcal{D}_{Y_{a}}^{\ell_{a}}.
\end{align*}
Now for each $\ell\in \N$ we define the ``$L^2$, degree $\ell$'' Sobolev norm on $C_c^{\infty}(X_d)$ by
\begin{align*}
\|f\|_{{H^{\ell}}}:=\biggl(\sum_{\deg(Z)\leq \ell}\int_{X_d}|\mathcal{D}_Z(f)|^2\,\d\mu_d\biggr)^{1/2},
\end{align*}
where the summation is over all the monomials $Z$ in $\{Y_1,\ldots, Y_a\}$ with degree no greater than $\ell$. 
Fix a metric $\textrm{dist}(\cdot,\cdot)$ on $X_d= G/\G$ which is induced from a right $G$-invariant
Riemannian metric on $G$. We also define the following Lipschitz (semi-)norm on $C_c^{\infty}(X_d)$ with respect to this metric:
\begin{align*}
\|f\|_{\Lip}:=\sup\left\{\frac{|f(x_1)-f(x_2)|}{\textrm{dist}(x_1,x_2)}\col x_1, x_2\in X_d,\ x_1\neq x_2\right\},\qquad f\in C_c^{\infty}(X_d).
\end{align*} 
Let us also write $\|\cdot\|_{C^0}$ for the uniform norm on $C_c(X_d)$.
Finally, for any $f\in C_c^{\infty}(X_d)$ we define
\begin{align*}
\mathcal{N}_{\ell}(f):=\max\left\{\|f\|_{C^0},\|f\|_{\Lip},\|f\|_{{H^{\ell}}}\right\}.
\end{align*}
We can now state the result which we need from 
\cite{BjorklundGorodnik2019}.
\begin{thm}[{{\cite[Corollary 2.4]{BjorklundGorodnik2019}}}]\label{thm:higherordermixing}
There exist $\ell\in\N$ and $\delta>0$ such that for every $b\in \N$ and any 
$f_0\in C^{\infty}(\mathcal{Y})$, $f_1,\ldots, f_b\in C_c^{\infty}(X_d)$ and $s_1,\ldots, s_b>0$, we have
\begin{align*}
\int_{\mathcal{Y}}f_0(\Lambda_A)\left(\prod_{i=1}^bf_i(g_{s_i}\Lambda_A)\right)\,\d A&=\textrm{Leb}(f_0)\prod_{i=1}^b\mu_d(f_i)+O_{b,f_0}\left(e^{-\delta D(s_1,\ldots, s_b)}\prod_{i=1}^b\mathcal{N}_\ell(f_i)\right),
\end{align*}
where $D(s_1,\ldots, s_b):=\min\left\{s_i, |s_i-s_j|\col 1\leq i\neq j\leq b\right\}$.
\end{thm}

(In fact, in \cite[Corollary 2.4]{BjorklundGorodnik2019}, the error term is also explicit in terms of $f_0$.)

Taking {$f_0\equiv 1$} 
on $\mathcal{Y}$ and $b=1,2$, we get the following effective equidistribution and doubly mixing of the family of $g_s$-translates $\{g_s \mathcal{Y}\}_{s>0}$ in $X_d$.
\begin{cor}\label{cor:effequhorsp}
Let $\ell\in\N$ and $\delta>0$ be as in Theorem \ref{thm:higherordermixing}. 
Then for any $f, f_1, f_2\in C^{\infty}_c(X_d)$ and $s, s_1, s_2>0$,
\begin{align}\label{equ:effequ}
\int_{\mathcal{Y}}f(g_s\Lambda_A)\,\d A=\mu_d(f)+O\left(e^{-\delta s}\mathcal{N}_\ell(f)\right),
\end{align}
and
\begin{align}\label{equ:effmixing}
\int_{\mathcal{Y}}f_1(g_{s_1}\Lambda_A)f_2(g_{s_2}\Lambda_A)\,\d A=\mu_d(f_1)\mu_d(f_2)+O\left(e^{-\delta\min\{s_1, s_2, |s_1-s_2|\}}\mathcal{N}_\ell(f_1)\mathcal{N}_\ell(f_2)\right).
\end{align}
\end{cor}

\subsection{Smooth approximations and estimates on norms}
In this subsection we prove the existence of smooth functions
$\phi\in C_c^{\infty}(X_d)$
bounding our shrinking targets from above and below in an appropriate sense,
with control on the norm $\scrN_\ell(\phi)$. We follow the strategy {of} \cite[Theorem 1.1]{KleinbockMargulis2018} while allowing the small identity neighborhoods of $G$ (against which we convolve) to shrink. 


%

Recall that
\begin{align*}
\widetilde{\Delta}_r:=\bigcup_{0\leq s<1}g_{-s} \Delta^{-1}[0,r]\qquad {(0<r<1)}.
\end{align*}

\begin{lem}\label{lem:smapprox}
Let $\ve>0$. For any $0<r<1$, there exists $\phi_r\in C^{\infty}_c(X_d)$ satisfying
$\chi_{\widetilde{\Delta}_r}\leq \phi_r\leq \chi_{\widetilde{\Delta}_{2r}}$
and, for any $\ell\in\N$:
\begin{align}\label{lem:smapproxres}
\mathcal{N}_\ell(\phi_r)\ll_{\ell,d,\ve} r^{-L},
\qquad\text{with }\:
L:=1+\max\left\{0,\ell+\ve-\tfrac{d^2}4-\tfrac d4\right\}.
\end{align}
\end{lem}

(Note that the implied constant in the bound in \eqref{lem:smapproxres} is independent of $r$.)

To prepare for the proof,
let us define, for any $r>0$,
\begin{align*}
\mathcal{O}_r:=\left\{g\in G\col \max\{\|g-I_d\|, \|g^{-1}-I_d\|\}<\frac{r}{10d}\right\}.
\end{align*}
Here the norm is the supremum norm on the matrix space $M_{d,d}(\R)$. Clearly, $\mathcal{O}_r$ is an open neighborhood of the identity element in $G$ and it is invariant under inversion.
Let $\nu$ be the normalized Haar measure of $G$ as in Section \ref{sec:haar};
recall that $\nu$ locally agrees with $\mu_d$. 

We will need the following auxiliary lemma. 
\begin{lem}\label{lem:smappest}
For every $0<r<1$, there exists a
function ${\theta}_r\in C_c^\infty(G)$
satisfying 
${\theta}_r\geq0$,
$\supp({\theta}_r)\subset \mathcal{O}_{r/2}$,
$\int_{G}{\theta}_r(g)\,d\nu(g)=1$
and 
$\|\mathcal{D}_Z({\theta}_r)\|_{L^{\infty}(G)}\ll_{\ell,d} r^{1-d^2-\ell}$
for every monomial $Z=Y_1^{\ell_1}\cdots Y_a^{\ell_a}$, where $\ell=\deg(Z)=\ell_1+\cdots+\ell_a$.
\end{lem}
(The implied constant in the bound on 
$\|\mathcal{D}_Z({\theta}_r)\|_{L^{\infty}(G)}$ depends only on 
$\ell$ and $d$, and not on $r$.)

\begin{proof}[Proof (sketch)]
Let $\varphi:\Omega\to\R^a$ ($a=d^2-1$) be an arbitrary $C^\infty$ coordinate chart of an open neighbourhood 
$\Omega$ of 
$I_d$ in $G$,
with $\varphi(I_d)=\bn$. 
Let $\eta\in C_c^\infty(\R^a)$ be a fixed bump function in $\R^a$,
i.e.\  a function satisfying $\eta\geq0$ and
$\int_{\R^a}\eta\,\d \vecx=1$.
We may assume that the support of $\eta$ is contained in the unit ball centered at the origin,
$\scrB_1^a$.
For each $t>0$ define $\eta_t\in C_c^\infty(\R^a)$ through
\begin{align*}
\eta_t(\vecx)=t^{-a}\eta(t^{-1}\vecx)\qquad (\vecx\in\R^a),
\end{align*}
and note that $\supp(\eta_t)\subset\scrB_t^a$ and $\int_{\R^a}\eta_t\,\d\vecx=1$.
Let us choose the constant $c>0$ so small that $\scrB_c^a\subset\varphi(\Omega)$
and $\varphi^{-1}(\scrB_{cr}^a)\subset\scrO_{r/2}$ for all $0<r<1$.
This is possible since the matrix entries of $g$ and $g^{-1}$
are $C^\infty$ functions of $g\in G$.
Now we may simply set, for each $0<r<1$,
\begin{align*}
{\theta}_r:=v_r\cdot (\eta_{cr}\circ \varphi),
\end{align*}
where $v_r>0$ is chosen so as to make 
$\int_{G}{\theta}_r\,\d\nu=1$.
One verifies that the limit $\lim_{r\to0^+}v_r$ exists and is a positive real number.
Using this fact, and recalling $a=d^2-1$, 
all the properties stated in the lemma are straightforward to verify.
\end{proof}

\begin{proof}[Proof of Lemma \ref{lem:smapprox}]
We claim that for any $r_1, r_2>0$,
\begin{align}\label{equ:increlathick}
\mathcal{O}_{r_1} \widetilde{\Delta}_{r_2}\subset \widetilde{\Delta}_{r_1+r_2}.
\end{align}
First we note that for any $h\in M_{d,d}(\R)$ and $\bm{v}\in \R^d$, $\|h \bm{v}\|\leq d\|h\|\|\bm{v}\|$. 
This implies that for any $r>0$ and any $g\in \mathcal{O}_{10r}$ and $\bm{v}\in \R^d$,
\begin{align*}
\|g \bm{v}\|\leq \|\bm{v}\|+\|(g-I_d)\bm{v}\|\leq (1+r)\|\bm{v}\|.
\end{align*}
Hence for all $r>0$,
\begin{align*}
\Delta(\Lambda)-\Delta(g\Lambda)\leq \log (1+r)< r,\qquad \forall\ g\in \mathcal{O}_{10r},\ \Lambda\in X_d.
\end{align*}
Similarly, since $\mathcal{O}_{10r}$ is invariant under inversion, we also have
\begin{align*}
\Delta(g \Lambda)-\Delta(\Lambda)=\Delta(g\Lambda)-\Delta(g^{-1}g \Lambda)< r,\qquad \forall\ g\in \mathcal{O}_{10r},\ \Lambda\in X_d.
\end{align*}
Thus 
\begin{align*}
\mathcal{O}_{10r_1} \Delta^{-1}[0, r_2]\subset \Delta^{-1}[0, r_1+r_2],
\qquad \forall\ r_1, r_2>0.
\end{align*}
Now to prove the relation \eqref{equ:increlathick}, in view of the definition of $\widetilde{\Delta}_r$, it suffices to show that for any \linebreak $g\in \mathcal{O}_{r_1}$, 
$0\leq s<1$ and $\Lambda\in \Delta^{-1}[0, r_2]$  there exists some $0\leq s'<1$ such that $$gg_{-s}\Lambda\in g_{-s'} \Delta^{-1}[0, r_1+r_2],$$ or equivalently, $g_{s'}gg_{-s} \Lambda\in \Delta^{-1}[0, r_1+r_2]$. We take $s'=s$. By direct computation and using $\alpha_i,\beta_j\in (0, 1)$ for all $1\leq i\leq m, 1\leq j\leq n$, we have 
\begin{align*}
\max\big\{\|g_{s}gg_{-s}-I_d\|, \|(g_{s}gg_{-s})^{-1}-I_d\|\big\}&=\max\big\{\|g_s(g-I_d)g_{-s}\|, \|g_s(g^{-1}-I_d)g_{-s}\|\big\}\\
&<e^2\max\big\{\|g-I_d\|, \|g^{-1}-I_d\|\big\}<e^2\frac{r_1}{10d}<\frac{r_1}{d}.
\end{align*}
Thus $g_sgg_{-s}\in \mathcal{O}_{10r_1}$, implying that
\begin{align*}
g_sgg_{-s}\Lambda\in \mathcal{O}_{10r_1} \Delta^{-1}[0, r_2]\subset \Delta^{-1}[0, r_1+r_2].
\end{align*}
This finishes the proof of \eqref{equ:increlathick}. 

Given any $0<r<1$, we now 
choose ${\theta}_r$ as in Lemma \ref{lem:smappest},
and then define our approximating function $\phi_r\in C_c^{\infty}(X_d)$ as the convolution 
\begin{align*}
\phi_r(x):={\theta}_r\ast \chi_{\widetilde{\Delta}_{3r/2}}(x)=\int_{G}{\theta}_r(g)\chi_{\widetilde{\Delta}_{3r/2}}(g^{-1}x)\, \d\nu(g).
\end{align*}
It follows from ${\theta}_r\geq0$ and $\int_G{\theta}_r\,d\nu=1$
that $\phi_r$ takes values in $[0,1]$. 
Moreover, for any \linebreak $g\in \supp({\theta}_r)\subset \mathcal{O}_{r/2}$ (so that $g^{-1}$ is also contained in $\mathcal{O}_{r/2}$ since $\mathcal{O}_{r/2}$ is invariant under inversion) and for any $x\in \widetilde{\Delta}_r$, we have by \eqref{equ:increlathick}
\begin{align*}
g^{-1}x\in \mathcal{O}_{r/2} \widetilde{\Delta}_r\subset \widetilde{\Delta}_{3r/2},
\end{align*}
implying that for any $x\in \widetilde{\Delta}_r$,
\begin{align*}
\phi_r(x)=\int_{\supp({\theta}_r)}{\theta}_r(g)\chi_{\widetilde{\Delta}_{3r/2}}(g^{-1}x)\,\d\nu(g)=\int_{\supp({\theta}_r)}{\theta}_r(g)\,\d\nu(g)=1.
\end{align*}
Thus $\chi_{\widetilde{\Delta}_r}\leq \phi_r$. 
Next, we claim that $\supp(\phi_r)\subset \widetilde{\Delta}_{2r}$.
To prove this, note that since $\supp({\theta}_r)$ is compact and contained in $\scrO_{r/2}$,
there exists some ${\epsilon}\in(0,r/2)$ such that $\supp({\theta}_r)\subset\scrO_{r/2-{\epsilon}}$.
Now if $x\in\supp(\phi_r)$ then there exists some
$x'\in\scrO_{{\epsilon}}\, x$ with $\phi_r(x')>0$;
and by the definition of $\phi_r$ there then exists some $g\in\supp({\theta}_r)\subset\scrO_{r/2-{\epsilon}}$
such that $g^{-1}x'\in\tDelta_{3r/2}$.
Hence $x'\in\scrO_{r/2-{\epsilon}}\tDelta_{3r/2}\subset\tDelta_{2r-{\epsilon}}$,
and (since $\scrO_{{\epsilon}}$ is invariant under inversion) $x\in\scrO_{{\epsilon}}\,x'\subset\scrO_{{\epsilon}}\,\tDelta_{2r-{\epsilon}}\subset\tDelta_{2r}$.
We have thus proved that $\supp(\phi_r)\subset \widetilde{\Delta}_{2r}$.
Using this inclusion 
together with the fact that $\phi_r$ takes values in $[0, 1]$, 
we conclude that $\phi_r\leq \chi_{\widetilde{\Delta}_{2r}}$. (Note that 
$\phi_r\leq \chi_{\widetilde{\Delta}_{2r}}$ follows already from the easier fact that 
for any $x'\in X_d$, $\phi_r(x')>0$ implies $x'\in\widetilde{\Delta}_{2r}$.
However we need some control on $\supp(\phi_r)$ below when we discuss derivatives of $\phi_r$.)

For the norm bounds, we first note that using the invariance of the Haar measure, for any $Y\in \mathfrak{g}$ we have $\mathcal{D}_Y(\phi_r)=\big(\mathcal{D}_Y({\theta}_r)\big)\ast \chi_{\widetilde{\Delta}_{3r/2}}$. More generally,
for any monomial $Z$ in $\{Y_1,\ldots, Y_a\}$,
\begin{align}\label{equ:diffrela}
\mathcal{D}_Z(\phi_r)=\big(\mathcal{D}_Z({\theta}_r)\big)\ast \chi_{\widetilde{\Delta}_{3r/2}}. 
\end{align}
Recall from Lemma \ref{lem:smappest}
that 
$\supp({\theta}_r)\subset \mathcal{O}_{r/2}$
and
$\|\mathcal{D}_Z({\theta}_r)\|_{L^{\infty}(G)}\ll_{\ell_Z,d} r^{1-d^2-\ell_Z}$,
where $\ell_Z$ is the degree of $Z$.
Furthermore, it is easily verified that
$\nu(\mathcal{O}_{r/2})\asymp_d r^{d^2-1}$.
Using these facts, we have, for every $x\in X_d$
and every monomial $Z$ of degree $\ell_Z\leq\ell$,
\begin{align}\label{lem:smapproxpf1}
|\scrD_Z(\phi_r)(x)|=\left|\int_{\supp({\theta}_r)}\mathcal{D}_Z({\theta}_r)(g)\chi_{\widetilde{\Delta}_{3r/2}}(g^{-1}x)\,\d\nu(g)\right|
\ll_{d,\ell}r^{1-d^2-\ell_Z}\nu(\mathcal{O}_{r/2})\ll_{d} r^{-\ell_Z}.
\end{align}
Hence, using also $\supp(\phi_r)\subset\tDelta_{2r}$ and 
Theorem \ref{thm:volumethickening}, we get
\begin{align*}
\|\phi_r\|_{{H^{\ell}}}\ll_{d,\ell}\mu_d(\tDelta_{2r})^{1/2} r^{-\ell}
\ll_{d,\ve} r^{\frac{d^2}4+\frac d4-1-\ell-\ve}.
\end{align*}

Finally, using the fact that
for any $0<r<1$,
the support of $\phi_r$ is contained in the fixed precompact set $\tDelta_2$,
we have
$\|\phi_r\|_{\Lip}\ll_d\sup_{x\in X_d}\sup_{j\in\{1,\ldots,a\}}\left|\scrD_{Y_j}(\phi_r)(x)\right|$, 
and hence by \eqref{lem:smapproxpf1},
\begin{align*}
\|\phi_r\|_{\Lip}\ll_{d} r^{-1}.
\end{align*}
Now the bound in \eqref{lem:smapproxres} follows, via the definition of the norm $\mathcal{N}_\ell$.
\end{proof}

\section{Proof of Theorem \ref*{thm:improvedirichleconwei}}
In this section, building on the analysis developed in the previous section, 
we give the proof of Theorem \ref{thm:improvedirichleconwei}. 
We keep the notation as in the previous section. In particular, throughout this section, 
we fix constants $\delta>0$ and $\ell\in\N$ as in Corollary \ref{cor:effequhorsp},
and for each $0<r<1$ we fix a 
function $\phi_r\in C^{\infty}_c(X_d)$ as in Lemma \ref{lem:smapprox}.
Taking $\varepsilon=1$ in Lemma \ref{lem:smapprox} 
we have that the norm bound in \eqref{lem:smapproxres} holds for 
$L:=1+\max\left\{0,\ell+1-\frac{d^2}{4}-\frac{d}{4}\right\}$. 

\subsection{Application of effective equidistribution}
For any $0<r<1$, taking $f=\phi_r$ in the effective equidistribution result
\eqref{equ:effequ} 
and applying the norm estimate \eqref{lem:smapproxres}, we get for any $s>0$,
\begin{align}\label{equ:mainest}
\int_{\scrY}\phi_r(g_s\Lambda_A)\,\d A=\mu_d(\phi_r)+O_d\left(e^{-\delta s}r^{-L}\right).
\end{align}
When $r$ is small, the above integral should be expected to be small as well; 
however,  the error term in \eqref{equ:mainest} blows up  as $r\to 0^+$. 
To remedy this issue,
for $r$ very small we will instead prove an upper bound on the integral,
obtained by applying \eqref{equ:mainest} for a suitable enlargened $r$-value.
The result is as follows:
\begin{lem}\label{lem:auxest}
Let $\eta:=\frac{\delta}{\varkappa_d+L}$. There exists $r_0\in (0, \frac12)$ such that for any $0<r< r_0$ and 
$s>0$,
\begin{align}\label{equ:auxest}
\int_{\scrY}\phi_r(g_s\Lambda_A)\,\d A
\:\: \begin{cases}
{\displaystyle \asymp_d r^{\varkappa_d}\log^{\lambda_d}\bigl(\tfrac{1}{r}\bigr)} & \textrm{if $r>e^{-\eta s}$},
\\[8pt]
{\displaystyle \ll_d e^{-\varkappa_d\eta s/2}} & \textrm{if $r\leq e^{-\eta s}$.}
\end{cases}
\end{align}
In particular, for any 
sequence $\{\rho_k\}_{k\in\N}\subset (0, \frac12)$
with $\lim_{k\to\infty}\rho_k=0$, we have
\begin{align}\label{equ:diveeqspser}
\sum_k\int_{\scrY}\phi_{\rho_k}(g_k\Lambda_A)\,\d A=\infty\quad 
\Longleftrightarrow\quad \sum_k \rho_k^{\varkappa_d}\log^{\lambda_d}\left(\frac{1}{\rho_k}\right)=\infty.
\end{align}
\end{lem} 

\begin{proof}
First we note that by Theorem \ref{thm:volumethickening} and the relation $\chi_{\widetilde{\Delta}_r}\leq \phi_r\leq \chi_{\widetilde{\Delta}_{2r}}$ (see Lemma \ref{lem:smapprox}),
we have
$\mu_d(\phi_r)\asymp_d r^{\varkappa_d}\log^{\lambda_d}\left(\frac{1}{r}\right)$.
Furthermore, if 
$r>e^{-\eta s}$, then 
the ratio of the main term and the error term in 
\eqref{equ:mainest} satisfies:
\begin{align*}
\frac{\mu_d(\phi_r)}{e^{-\delta s}r^{-L}}
\asymp_d\frac{r^{\varkappa_d}\log^{\lambda_d}\bigl(\tfrac{1}{r}\bigr)}{e^{-\delta s}r^{-L}}>\log^{\lambda_d}\bigl(\tfrac{1}{r}\bigr),
\end{align*}
which we can force to be as large as we like by taking the constant $r_0$ sufficiently small
(in a way which only depends on $d$).
Hence it follows from \eqref{equ:mainest}
that \eqref{equ:auxest} holds
in the case $r>e^{-\eta s}$.

Next assume $r\leq e^{-\eta s}$.
Set
$\rho:=2e^{-\eta s}$.
If $\rho<r_0$, then by what we proved in the previous paragraph,
\begin{align*}
\int_{\scrY}\phi_{\rho}(g_s\Lambda_A)\,\d A\asymp_d \rho^{\varkappa_d}\log^{\lambda_d}\bigl(\tfrac{1}{\rho}\bigr)
\ll_d e^{-\varkappa_d\eta s/2},
\end{align*}
and hence the bound in \eqref{equ:auxest} follows,
since 
$\phi_r\leq \chi_{\widetilde{\Delta}_{2r}}\leq \chi_{\widetilde{\Delta}_{\rho}}\leq \phi_{\rho}$
(again see Lemma \ref{lem:smapprox}).
In the remaining case when $\rho\geq r_0$,
we have $s\ll_d 1$ and $e^{-\varkappa_d\eta s/2}\gg_d1$,
and hence the bound in \eqref{equ:auxest} holds simply because of $\phi_r\leq1$.
This completes the proof of \eqref{equ:auxest}.

For the last part of the lemma, since $\lim_{k\to\infty}\rho_k=0$, after possibly deleting finitely many terms from
the two sums in \eqref{equ:diveeqspser},
we may assume $\rho_k<r_0$ for all appearing terms.
Next, using the second bound in \eqref{equ:auxest}
and the fact that both {of}  the series
$\sum_{k} e^{-\varkappa_d\eta k/2}$
and
$\sum_{\rho_k\leq e^{-\eta k}} \rho_k^{\varkappa_d}\log^{\lambda_d}\left(\frac{1}{\rho_k}\right)$
converge,
it follows that the two divergence statements in 
\eqref{equ:diveeqspser} 
remain unaffected if all the terms 
for which $\rho_k\leq e^{-\eta k}$ are removed from the respective series.   
After this operation, the equivalence in 
\eqref{equ:diveeqspser} is an immediate consequence of the first 
relation in \eqref{equ:auxest}.
\end{proof}

\subsection{The convergence case} 

This case is now easily handled using Lemma \ref{lem:auxest}.
\begin{proof}[Proof of the convergence case of Theorem \ref{thm:improvedirichleconwei}]
Let $r=r_{\psi}: [s_0,\infty)\to (0,\infty)$ be the continuous, decreasing 
function corresponding to $\psi$ as in Lemma \ref{lem:danicor}. 
First note that since the series \eqref{equ:crticalseriesconj} converges, we have $\lim_{t\to\infty}t\psi(t)=1$ (or equivalently, $\lim_{s\to\infty}r(s)=0$ as seen from the proof of Lemma \ref{lem:danicor}). Moreover, by the last part of Lemma \ref{lem:danicor}, the fact that the series \eqref{equ:crticalseriesconj} converges implies that the series $\sum_k r(k)^{\varkappa_d}\log^{\lambda_d}\bigl(1+\tfrac{1}{r(k)}\bigr)$ also converges.

Now for each $k>s_0$ let us define 
\begin{align*}
\overline{B}_k:=\bigcup_{0\leq s<1}g_{-s} \Delta^{-1}[0, \omega_1r(k+s)]\quad \textrm{and}\quad \overline{E}_k:=\left\{\Lambda_A\in \mathcal{Y}\col g_k \Lambda_A\in \overline{B}_{k}\right\},
\end{align*}
where $\omega_1:=\max\{m\alpha_i, n\beta_j\col 1\leq i\leq m, 1\leq j\leq n\}$ is as in Proposition \ref{prop:dyint}.
In view of Proposition \ref{prop:dyint} (and the paragraph after it), it suffices to show that for $\textrm{Leb}$-a.e. $\Lambda_A\in \mathcal{Y}$, $g_k \Lambda_A\in \overline{B}_k$ for only finitely many $k>s_0$,
or equivalently, that the limsup set $\limsup_{k\to\infty}\overline{E}_k$ is of zero measure. 
Thus in view of the Borel-Cantelli lemma, it suffices to show that $\sum_k\textrm{Leb}(\overline{E}_k)<\infty$.

To prove this, we will approximate the shrinking targets $\{\overline{B}_k\}_{k>s_0}$ from above. Since 
{$r(s)\to 0$ as $s\to\infty$}, by enlarging $s_0$ if necessary 
(equivalently, enlarging $t_0$ {as in} Lemma \ref{lem:danicor}),
we may assume $\omega_1r(s)\in (0, \frac12)$ for all $s>s_0$. Moreover, by Lemma \ref{lem:danicor}, $r(\cdot)$ is decreasing; thus with $\rho_k:=\omega_1r(k)$, {we have}
\begin{align}\label{equ:inclrelasp}
\overline{B}_k 
\subset \widetilde{\Delta}_{\rho_k},
\qquad \forall \,k>s_0.
\end{align}
Recall that for each
$0<r<1$ we have fixed a 
function $\phi_r\in C^{\infty}_c(X_d)$ as in Lemma \ref{lem:smapprox}.
Now for each $k>s_0$, we have
$\chi_{\overline{B}_k}\leq\chi_{\widetilde{\Delta}_{\rho_k}}\leq\phi_{\rho_k}$
(by \eqref{equ:inclrelasp} and Lemma \ref{lem:smapprox}),
implying that
\begin{align}\label{LebIkupperbound}
\textrm{Leb}(\overline{E}_k)&=\int_{\mathcal{Y}}\chi_{\overline{B}_k}(g_k\Lambda_A)\,\d A\leq \int_{\mathcal{Y}}\phi_{\rho_k}(g_k \Lambda_A)\,\d A.
\end{align}
Next, it follows from $\rho_k=\omega_1r(k)$ and the convergence of $\sum_k r(k)^{\varkappa_d}\log^{\lambda_d}\bigl(1+\tfrac{1}{r(k)}\bigr)$
that the series $\sum_k \rho_k^{\varkappa_d}\log^{\lambda_d}\bigl(\tfrac{1}{\rho_k}\bigr)$ also converges;
{in addition}, $\lim_{k\to\infty}\rho_k=0$ since $\lim_{k\to\infty}r(k)=0$.
Hence by the last part of Lemma \ref{lem:auxest} combined with \eqref{LebIkupperbound},
we have $\sum_{k}\textrm{Leb}\left(\overline{E}_k\right)<\infty$, finishing the proof.
\end{proof}
\begin{remark}\label{rmk:samlei}
Let $(m,n)=(2,1)$; thus $d=3$. In \cite[Theorem 1.1]{ChowYang2019}, Chow and Yang proved an effective equidistribution result for certain Diophantine lines in $\scrY$ translated under the full (two dimensional) diagonal subgroup of $G$ along certain restricted directions. In particular, their result implies the following: Let $(a,b)\in \R^2$ be a \textit{Diophantine} vector (see \cite[p.\ 2]{ChowYang2019} for the definition), and let $J\subset \R$ be a compact subinterval. Then these exist constants $\ell'\in\N$, $c\in (0,1)$ and $\delta'>0$ such that for any {pair of weights} $\bm{\alpha}=(\frac{c'}{1+c'}, \frac{1}{1+c'})$ with $0<c'\leq c$, for any $f\in C_c^{\infty}(X_3)$ and for any $s>0$
\begin{align}\label{equ:effequcy}
\frac{1}{|J|}\int_{J}f\left(g_s\Lambda_{\bm{v}(x)}\right)\d x=\mu_3(f)+O\left(e^{-\delta' s}\|f\|_{L^{\infty}_{\ell'}}\right),
\end{align}
where $g_s=g^{\bm{\alpha},1}_s=\diag(e^{\alpha_1 s}, e^{\alpha_2 s}, e^{-s})$ with $\bm{\alpha}$ as above, $\bm{v}(x):=(ax+b, x)^t\in \R^2$ and $\|\cdot\|_{L^{\infty}_{\ell'}}$ 
the ``$L^{\infty}$, degree $\ell'$'' Sobolev norm defined by
\begin{align*}
\|f\|_{L^{\infty}_{\ell'}}:=\sum_{\deg(Z)\leq \ell'}\|\mathcal{D}_Z(f)\|_{C^0}.
\end{align*}  
Here $\|\cdot\|_{C^0}$ is the uniform norm on $C^{\infty}_c(X_3)$ as before. On the other hand, it is easy to see from Lemma \ref{lem:smappest} and the relation \eqref{equ:diffrela} that there exists $L'>1$ such that 
\begin{align}\label{equ:normestcy}
\|\phi_r\|_{L^{\infty}_{\ell'}}\ll_{\ell'} r^{-L'},\qquad \forall\ 0<r<1.
\end{align}
Using a similar analysis with \eqref{equ:effequcy} and \eqref{equ:normestcy} in place of \eqref{equ:effequ} and \eqref{lem:smapproxres} respectively, we can conclude that if the series \eqref{equ:crticalseriesconj} (with $d=3$) converges, then for $\textrm{Leb}$-a.e. $x\in J$ the column vector $\bm{v}(x)=(ax+b, x)^t$ is $\psi_{\bm{\alpha},1}$-Dirichlet.
\end{remark}
\subsection{The divergence case}
In this subsection we prove the divergence case of Theorem \ref{thm:improvedirichleconwei}. 
We first record from
\cite{KochenStone1964}
the following divergence Borel-Cantelli lemma which we will use.


\begin{lem}\label{lem:borcandiv}
Let $(\mathcal{X},\mu)$ be a probability space. Let $\{h_k\}_{k\in\N}$ be a sequence of measurable functions on $\mathcal{X}$ taking values in $[0,1]$. Let $b_k:=\mu(h_k)$. Suppose $\sum_kb_k=\infty$ and 
\begin{align}\label{equ:quasiinde}
\liminf_{k_2\to\infty}\frac{\int_{\mathcal{X}}\left(\sum_{i=k_1}^{k_2}h_i(x)-\sum_{i=k_1}^{k_2}b_i\right)^2\d\mu(x)}{\left(\sum_{i=k_1}^{k_2}b_i\right)^2}=0\quad \textrm{for some $k_1\in\N$}. 
\end{align}
Then for $\mu$-a.e. $x\in \mathcal{X}$, $h_k(x)>0$ infinitely often.
\end{lem}
\begin{proof}
Let $Y_1,Y_2,\ldots$ be the sequence of random variables defined by
$Y_k(x)=\sum_{i=k_1}^{k+k_1-1}h_i(x)$ ({$k
\in\N$}).
Note that
\begin{align*}
\int_{\mathcal{X}}\biggl(\sum_{i=k_1}^{k_2}h_i(x)-\sum_{i=k_1}^{k_2}b_i\biggr)^2\, \d\mu(x)
=\mu\bigl(Y_{k_2-k_1+1}^2\bigr)-\mu\bigl(Y_{k_2-k_1+1}\bigr)^2;
\end{align*}
hence \eqref{equ:quasiinde} implies that
$\limsup_{k\to\infty}\mu(Y_k)^2/\mu(Y_k^2)=1$.
 {Therefore} by part (iii) of the main theorem in 
\cite{KochenStone1964},
for $\mu$-a.e.\ $x$ we have
$\limsup_{k\to\infty}Y_k(x)/\mu(Y_k)>0$.
Also $\mu(Y_k)=\sum_{i=k_1}^{k+k_1-1}b_i\to\infty$ as $k\to\infty$.
Hence it follows that
for $\mu$-a.e.\ $x$ we have
$\lim_{k\to\infty}Y_k(x)=+\infty$,
and in particular $h_i(x)>0$ for infinitely many $i$.
\end{proof}
\begin{remark}\label{rmk:posme}
If one replaces the assumption \eqref{equ:quasiinde} by the weaker assumption that 
\begin{align*}
\liminf_{k_2\to\infty}\frac{\int_{\mathcal{X}}\left(\sum_{i=k_1}^{k_2}h_i(x)-\sum_{i=k_1}^{k_2}b_i\right)^2\d\mu(x)}{\left(\sum_{i=k_1}^{k_2}b_i\right)^2}=:C<\infty\quad \textrm{for some $k_1\in\N$},
\end{align*}
then by the application of part (iii) of the main theorem in 
\cite{KochenStone1964} we get instead
$$
\mu\left(\left\{x\in X\col \limsup_{k\to\infty}\frac{Y_k(x)}{\mu(Y_k)}>0\right\}\right)\geq \frac{1}{1+C}.
$$ 
In particular, there is a positive measure set of $x\in X$ such that $h_i(x)>0$ infinitely often.
\end{remark}

\begin{proof}[Proof of the divergence case of Theorem \ref{thm:improvedirichleconwei}]
First we note that in view of Remark \ref{rmk:kw} we may assume $\lim_{t\to\infty}t\psi(t)=1$.
Let $r=r_{\psi}$ be the continuous, decreasing 
function corresponding to $\psi$ as in Lemma \ref{lem:danicor};
then from the proof of that lemma we have
$\lim_{s\to\infty}r(s)=0$. 
Also by Lemma~\ref{lem:danicor}, since the series \eqref{equ:crticalseriesconj} diverges, the series $\sum_kr(k)^{\varkappa_d}\log^{\lambda_d}\bigl(1+\tfrac{1}{r(k)}\bigr)$ also diverges. 
Moreover, by Remark \ref{rmk:equivinte}, condition \eqref{equ:condipsi} is equivalent to
\begin{align}\label{equ:serassump2}
\liminf_{s_1\to\infty}
\frac{\sum_{s_0<k\leq s_1}r(k)^{\varkappa_d}\log^{\lambda_d+1}\bigl(1+\frac1{r(k)}\bigr)}
{\Bigl(\sum_{s_0<k\leq s_1}r(k)^{\varkappa_d}\log^{\lambda_d}\bigl(1+\frac1{r(k)}\bigr)\Bigr)^2}
=0.
\end{align}
Now for any $k>s_0$ let
\begin{align*}
\underline{B}_k=\bigcup_{0\leq s<1}g_{-s}\Delta^{-1}[0, \omega_2r(k+s)]
\end{align*}
be as before with $\omega_2:=\min\{m\alpha_i, n\beta_j\col 1\leq i\leq m, 1\leq j\leq n\}$ as in Proposition \ref{prop:dyint}. In view of Proposition \ref{prop:dyint} (and the paragraph after it) it suffices to show that for $\Leb$-a.e. $\Lambda_A\in \mathcal{Y}$, $g_k \Lambda_A\in \underline{B}_k$ infinitely often. 

In this case we will approximate the shrinking targets $\left\{\underline{B}_k\right\}_{k>s_0}$ from below. 
Recall that for each
$0<r<1$ we have fixed a 
function $\phi_r\in C^{\infty}_c(X_d)$ as in Lemma \ref{lem:smapprox}.
Again we may assume $\omega_2r(s)\in (0,\frac12)$ for all $s>s_0$, and since $r(\cdot)$ is decreasing, we have $\widetilde{\Delta}_{\omega_2r(k+1)}\subset \underline{B}_{k}$, implying that, with $\rho_k:=\omega_2r(k+1)/2$:
\begin{align}\label{equ:inclreldiv}
f_k:=\phi_{\rho_k}\leq \chi_{\widetilde{\Delta}_{\omega_2r(k+1)}}\leq\chi_{\underline{B}_k},\qquad \forall \,k>s_0.
\end{align}
Let us set, for each $k>s_0$,
$$
b_k:=\int_{\scrY}f_k(g_k\Lambda_A)\,\d A=\int_{\scrY}\phi_{\rho_k}(g_k\Lambda_A)\,\d A.
$$ 
Then, similarly {to} the proof of the convergence case,
by applying the last part of Lemma \ref{lem:auxest}
and using the relation $\rho_k=\omega_2r(k+1)/2$ and 
the facts that the series $\sum_kr(k)^{\varkappa_d}\log^{\lambda_d}\bigl(1+\tfrac{1}{r(k)}\bigr)$ diverges
and $\lim_{s\to\infty}r(s)=0$,
it follows that the series $\sum_kb_k$ also diverges. 

Now for each $k>s_0$, let $h_k$ be the function on $\mathcal{Y}$ defined by $h_k(\Lambda_A):=f_k(g_k \Lambda_A)$. 
Then in view of the definition of $f_k:=\phi_{\rho_k}$ and the relation 
$\chi_{\widetilde{\Delta}_{\rho_k}}\leq \phi_{\rho_k}\leq \chi_{\widetilde{\Delta}_{2\rho_k}}$, {the function}
$h_k$ takes values in $[0,1]$, and
\begin{align*}
\int_{\mathcal{Y}}h_k(\Lambda_A)\,\d A=\int_{\mathcal{Y}}f_k(g_k \Lambda_A)\,\d A=b_k.
\end{align*}
We will apply Lemma \ref{lem:borcandiv} to the 
probability space $(\mathcal{Y}, \textrm{Leb})$
and the sequence $\{h_k\}_{k>s_0}$. 
We have already seen that $\sum_kb_k=\infty$;
thus in view of Lemma \ref{lem:borcandiv} it suffices to show that 
$\{h_k\}_{k> s_0}$ satisfies 
condition \eqref{equ:quasiinde}. 

Let us take $C>0$ sufficiently large so that for all $k>C$, $\rho_k\in (0, r_0)$, {where $r_0$ is} the constant as in Lemma \ref{lem:auxest}. For any $k_2>k_1> C$, let us denote
\begin{align*}
Q_{k_1,k_2}&:=\int_{\mathcal{Y}}\left(\sum_{i=k_1}^{k_2}h_i(\Lambda_A)-\sum_{i=k_1}^{k_2}b_i\right)^2\,\d A
=\sum_{k_1\leq i,j\leq k_2}\int_{\mathcal{Y}}\left(h_i(\Lambda_A)h_j(\Lambda_A)-b_ib_j\right)\,\d A.
\end{align*}
Using the fact that for each 
$k_1\leq i\leq k_2$, 
$$
\int_{\mathcal{Y}}\left(h^2_i(\Lambda_A)-b_i^2\right)\,\d A\leq \int_{\mathcal{Y}}h_i(\Lambda_A)\,\d A=b_i,
$$
we have
\begin{align*}
Q_{k_1,k_2}
&\leq \sum_{i=k_1}^{k_2}b_i+2\sum_{k_1\leq i<j\leq k_2}\int_{\mathcal{Y}}\left(h_i(\Lambda_A)h_j(\Lambda_A)-b_ib_j\right)\,\d A.
\end{align*}
Fix $k_1\leq i<j\leq k_2$; we will use two different estimates for the term $\int_{\mathcal{Y}}\big(h_i(\Lambda_A)h_j(\Lambda_A)-b_ib_j\big)\,\d A$ depending on whether $\min\{i, j-i\}$ is large or small. First, applying the effective doubly mixing \eqref{equ:effmixing} to the pair $(f_i, f_j)$, we get 
\begin{align}\label{equ:corelation2pre}
\int_{\mathcal{Y}}f_i(g_i \Lambda_A)f_j(g_j\Lambda_A)\,\d A=\mu_d(f_i)\mu_d(f_j)+O\left(e^{-\delta\min\{i, j-i\}}\mathcal{N}_{\ell}(f_i)\mathcal{N}_{\ell}(f_j)\right).
\end{align}
On the other hand, by \eqref{equ:effequ} we have
\begin{align}\label{thm:divcasewecopf2}
b_k=\mu_d(f_k)+O\left(e^{-\delta k}\mathcal{N}_{\ell}(f_k)\right),\quad \forall\ k>C.
\end{align}
Combining 
\eqref{thm:divcasewecopf2}, \eqref{equ:corelation2pre}, the norm estimate $\scrN_{\ell}(f_k)=\scrN_{\ell}(\phi_{\rho_k})\ll_{d} \rho_k^{-L}$ (by \eqref{lem:smapproxres}) and noting that $\int_{\scrY}h_i(\Lambda_A)h_j(\Lambda_A)\,\d A=\int_{\mathcal{Y}}f_i(g_i \Lambda_A)f_j(g_j\Lambda_A)\,\d A$,
we have
\begin{align*}
\left|\int_{\scrY}\left(h_i(\Lambda_A)h_j(\Lambda_A)-b_ib_j\right)\,\d A\right|\ll_d e^{-\delta\min\{i, j-i\}}\rho_i^{-L}\rho_j^{-L}<e^{-\delta\min\{i, j-i\}}\rho_j^{-2L}.
\end{align*}
On the other hand, using the trivial 
estimate $|h_ih_j-b_ib_j|\leq h_ih_j+ b_ib_j\leq h_j+b_j$ we have
\begin{align}\label{equ:trivialest}
\biggl|\int_{\mathcal{Y}}\left(h_i(\Lambda_A)h_j(\Lambda_A)-b_ib_j\right)\,\d A\biggr|
&\leq \int_{\mathcal{Y}}\left(h_j(\Lambda_A)+b_j\right)\,\d A
=2b_j.
\end{align}
Combining these two bounds,
we conclude:
\begin{align}\label{thm:divcasewecopf1}
Q_{k_1,k_2}\ll_d \sum_{i=k_1}^{k_2}b_i+
\sum_{j=k_1+1}^{k_2}\sum_{i=k_1}^{j-1}
\min\left\{e^{-\delta\min\{i, j-i\}}\rho_j^{-2L},b_j\right\}.
\end{align}
In order to bound the above inner sum we replace $k_1$ by $1$ and use the symmetry $i\mapsto j-i$
to get
\begin{align}\label{thm:divcasewecopf1a}
\sum_{i=k_1}^{j-1}\min\left\{e^{-\delta\min\{i, j-i\}}\rho_j^{-2L},b_j\right\}
\leq 2\sum_{i=1}^{j-1}\min\left\{e^{-\delta\,i}\rho_j^{-2L},b_j\right\}.
\end{align}
In the last sum,
all terms are $\leq b_j$,
and there are 
at most $O_d\bigl(\log\bigl(2+b_j^{-1}\rho_j^{-2L}\bigr)\bigr)$
terms which are equal to $b_j$
(indeed, remember that $\delta$ depends only on $d$).
Furthermore, if there are any terms which are less than $b_j$,
then these are bounded above by
$b_j,b_j e^{-\delta},b_j e^{-2\delta},\ldots$,
and so their sum is $O_d(b_j)$.
It follows that the last sum in \eqref{thm:divcasewecopf1a}
is $O_d\left(b_j\min\left\{j,\log\bigl(2+b_j^{-1}\rho_j^{-2L}\bigr)\right\}\right)$,
and hence from \eqref{thm:divcasewecopf1} we get
\begin{align*}
Q_{k_1,k_2}\ll_d \sum_{i=k_1}^{k_2}b_i+
\sum_{j=k_1+1}^{k_2}b_j\min\left\{j,\log\bigl(2+b_j^{-1}\rho_j^{-2L}\bigr)\right\}
\ll\sum_{j=k_1}^{k_2}b_j\min\left\{j,\log\bigl(2+b_j^{-1}\rho_j^{-2L}\bigr)\right\}.
\end{align*}
Let $\eta=\frac{\delta}{\varkappa_d+L}$ be as in Lemma \ref{lem:auxest} and set $\kappa:=\frac{\varkappa_d\eta}{2}$. 
Then by \eqref{equ:auxest} we have for each $k_1\leq j\leq k_2$:
\begin{align}\label{equ:divcasewecopf8}
b_j=\int_{\scrY}\phi_{\rho_j}(g_j\Lambda_A)\,\d A
\:\: \begin{cases}
{\displaystyle \asymp_d \rho_j^{\varkappa_d}\log^{\lambda_d}\bigl(\tfrac{1}{\rho_j}\bigr)} & \textrm{if $\rho_j>e^{-\eta j}$,}
\\[8pt]
{\displaystyle \ll_d e^{-\kappa j}} & \textrm{if $\rho_j\leq e^{-\eta j}$.}
\end{cases}
\end{align}
Thus for any $k_2>k_1>C$ we have (recalling that $\rho_j=\omega_2r(j+1)/2$)
\begin{align*}
Q_{k_1,k_2}
&\ll_d\sum_{\substack{k_1\leq j\leq k_2\\ (\rho_j\leq e^{-\eta j})}}e^{-\kappa j}j
+\sum_{\substack{k_1\leq j\leq k_2\\ (\rho_j> e^{-\eta j})}}
\rho_j^{\varkappa_d}\log^{\lambda_d}\Bigl(\frac1{\rho_j}\Bigr)\log\left(2+\rho_j^{-\varkappa_d-2L}\log^{-\lambda_d}\Bigl(\frac1{\rho_j}\Bigr)\right)
\hspace{40pt}
\\
&\ll_d 1+\sum_{j=k_1}^{k_2}\rho_j^{\varkappa_d}\log^{\lambda_d+1}\Bigl(\frac1{\rho_j}\Bigr)\ll 1+\sum_{j=k_1+1}^{k_2+1}r(j)^{\varkappa_d}\log^{\lambda_d+1}\Bigl(1+\frac1{r(j)}\Bigr).
\end{align*}
Similarly, by \eqref{equ:divcasewecopf8},
for any fixed $k_1>C$ we have as $k_2\to\infty$:
\begin{align*}
\sum_{j=k_1}^{k_2}b_j
\gg_d\sum_{\substack{k_1\leq j\leq k_2\\ (\rho_j> e^{-\eta j})}}
\rho_j^{\varkappa_d}\log^{\lambda_d}\Bigl(\frac1{\rho_j}\Bigr)
\gg\sum_{j=k_1}^{k_2}\rho_j^{\varkappa_d}\log^{\lambda_d}\Bigl(\frac1{\rho_j}\Bigr)\gg_{\omega_2}\sum_{j=k_1+1}^{k_2+1}r(j)^{\varkappa_d}\log^{\lambda_d}\Bigl(1+\frac1{r(j)}\Bigr),
\end{align*}
where the second relation holds since
the series $\sum_{j=k_1}^{\infty}\rho_j^{\varkappa_d}\log^{\lambda_d}\bigl(\frac1{\rho_j}\bigr)$ diverges (this follows from the relation $\rho_j=\omega_2r(j+1)/2$ and {the} fact that the series $\sum_jr(j)^{\varkappa_d}\log^{\lambda_d}\bigl(1+\tfrac{1}{r(j)}\bigr)$ diverges), while the same sum restricted to those $j$ for which
$\rho_j\leq e^{-\eta j}$ is convergent.

Combining the last two bounds, we conclude that for any $k_1>C$, and for $k_2$ sufficiently large,
\begin{align}\label{thm:divcasewecopf21}
\frac{Q_{k_1,k_2}}{\left(\sum_{j=k_1}^{k_2}b_j\right)^2}\ll_{d,\omega_2}
\frac{1+\sum_{j=k_1+1}^{k_2+1}r(j)^{\varkappa_d}\log^{\lambda_d+1}\bigl(1+\frac1{r(j)}\bigr)}
{\Bigl(\sum_{j=k_1+1}^{k_2+1}r(j)^{\varkappa_d}\log^{\lambda_d}\bigl(1+\frac1{r(j)}\bigr)\Bigr)^2}.
\end{align}
Since the series $\sum_j r(j)^{\varkappa_d}\log^{\lambda_d}\bigl(1+\frac1{r(j)}\bigr)$ diverges, 
condition \eqref{equ:serassump2} implies that 
the limit inferior of the expression in the right hand side of \eqref{thm:divcasewecopf21}
tends to zero as $k_2\to\infty$.
Hence \eqref{equ:quasiinde} holds.
We have also noted that $\sum_kb_k=\infty$.
Hence by Lemma \ref{lem:borcandiv},
for $\textrm{Leb}$-a.e.\ $A\in M_{m,n}(\R/\Z)$ {we have}  $h_k(\Lambda_A)=f_k(g_k \Lambda_A)>0$ infinitely often. 
Together with 
 \eqref{equ:inclreldiv}, {this} implies that for $\textrm{Leb}$-a.e.\ $A\in M_{m,n}(\R/\Z)$, {the lattice} $g_k\Lambda_A$ {belongs to} $\supp(f_k)\subset \underline{B}_k$ { for infinitely many $k\in\N$}. This finishes the proof.
\end{proof}
\begin{remark}\label{rmk:final}

For the divergence case in Theorem \ref{thm:improvedirichleconwei}, we note that if one replaces the assumption \eqref{equ:condipsi} by the weaker assumption that 
\begin{align*}
\liminf_{t_1\to\infty}\frac{\sum_{t_0\leq k\leq t_1}k^{-1}F_{\psi}(k)^{\varkappa_d}\log^{\lambda_d+1}\left(\frac{1}{F_{\psi}(k)}\right)}
{\left(\sum_{t_0\leq k\leq t_1}k^{-1}F_{\psi}(k)^{\varkappa_d}\log^{\lambda_d}\Bigl(\frac{1}{F_{\psi}(k)}\Bigr)\right)^2}< \infty,
\end{align*} 
then, in view of Remark \ref{rmk:equivinte}, Remark \ref{rmk:posme} and the estimate \eqref{thm:divcasewecopf21}, we can conclude that, under this weaker assumption, $\DI_{\bm{\alpha},\bm{\beta}}(\psi)^c$ is of positive Lebesgue measure.
{It is an interesting question, to which we do not know the answer,
whether $\DI_{\bm{\alpha},\bm{\beta}}(\psi)$ must always be of
zero or full Lebesgue measure. See \cite{BeresnevichVelani2008} for an analogous question in the setting of asymptotic approximation.}
\end{remark}

\bibliographystyle{abbrv}
\bibliography{DKbibliog}

\end{normalsize}
\end{document}